\documentclass[12pt]{article}
\usepackage[utf8]{inputenc}
\usepackage[margin=1in]{geometry}
\usepackage{amsmath,amssymb,amsthm,enumerate,enumitem,hyperref,algorithm,algorithmic,tikz}
\usepackage[mathscr]{euscript}

\tikzset{every picture/.style={line width=0.75pt}} 

\theoremstyle{definition}
\newtheorem{dfn}{Definition}[section] 

\newtheorem{prop}[dfn]{Proposition} 
\newtheorem{exm}[dfn]{Example}
\newtheorem{lem}[dfn]{Lemma}
\newtheorem{rem}[dfn]{Remark}

\newtheorem{cor}[dfn]{Corollary}
\newtheorem{con}[dfn]{Conjecture}

\newtheorem{prob}[dfn]{Problem}

\newcommand{\N}{\mathbb{N}}

\newcommand{\OO}{\mathcal{O}}

\newcommand{\abs}[1]{\left|#1\right|}
\newcommand{\floor}[1]{\lfloor#1\rfloor}
\newcommand{\ceil}[1]{\lceil#1\rceil}
\DeclareMathOperator{\diff}{diff}
\DeclareMathOperator{\spann}{span}
\DeclareMathOperator{\aut}{Aut}
\newcommand{\comment}[1]{}

\title{\vspace{-2.0cm} \bf $H$-Chromatic Symmetric Functions}

\author{Nancy Mae Eagles\footnote{University of Edinburgh, Edinburgh, Scotland, UK}, Ang\`ele M. Foley\footnote{Wilfrid Laurier University, Waterloo, Ontario, Canada}, Alice Huang\footnote{University of Toronto, Toronto, Ontario, Canada}, \\ Elene  Karangozishvili\footnote{Lafayette College, Easton, Pennsylvania, USA}, Annan Yu \footnote{Vanderbilt University, Nashville, Tennessee, USA}}
\date{\today}

\begin{document}

\maketitle

\begin{abstract}
We introduce $H$-chromatic symmetric functions, $X_{G}^{H}$, which use the $H$-coloring of a graph $G$ to define  a generalization of Stanley's chromatic symmetric functions.  
We say two graphs $G_1$ and $G_2$ are $H$-chromatically equivalent if $X_{G_1}^{H} = X_{G_2}^{H}$, and use this idea to study uniqueness results for $H$-chromatic symmetric functions, with a particular emphasis on the case $H$ is a complete bipartite graph.
We also show that several of the classical bases of the space of symmetric functions, i.e.\ the monomial symmetric functions, power sum symmetric functions, and elementary symmetric functions, can be realized as $H$-chromatic symmetric functions.
We end with some conjectures and open problems.

\end{abstract}


\section{Introduction}
\label{Introduction}

In this paper we introduce a new variation on chromatic symmetric functions: the $H$-chromatic symmetric functions, $X^H_G$. The original chromatic symmetric functions, $X_G$, were defined in 1995 by Stanley \cite{stanley1995symmetric} using the colorings of a graph $G$ to label the indeterminates in a symmetric function. 
The variation we  define employs the $H$-colorings of a graph $G$, using $H$ to restrict the colorings of $G$ and thus produce a different symmetric function.

Stanley's original paper produced two main conjectures, neither of which has been fully solved although significant progress has been made (e.g. \cite{GuayPaq}, \cite{HarPre}, \cite{MarMorWag}, \cite{OrScott}, \cite{SharWa}).  The first concerned $e$-positivity of clawfree, incomparability graphs. The second concerned the uniqueness of chromatic symmetric functions for trees.  We take inspiration from the second of these and look at  uniqueness results for $H$-chromatic symmetric functions.
 We can approach this from two angles and the two overarching questions are:

\begin{enumerate}
\item Given a fixed graph $H$, do  there exist two graphs, $G_1$ and $G_2$, such that $X^{H}_{G_{1}} = X^{H}_{G_{2}}$?
\item Given a fixed graph $G$, do there exist two graphs, $H_1$ and $H_2$, such that $X^{H_{1}}_{G} = X^{H_{2}}_{G}$?
\end{enumerate}
With respect to the first question we introduce the term {\em $H$-chromatically equivalent} and  say that two graphs $G_1$ and $G_2$ are $H$-chromatically equivalent if $X_{G_1}^{H} = X_{G_2}^{H}$. 
At the end of this section we detail the individual questions and results we explore. \\

The concept of a $H$-coloring has been well-studied in graph theory, particularly from a complexity angle (one of the classic papers in this area is \cite{Hell1990complexity}).
Informally, given a graph $H$ and a coloring of $H$, a $H$-coloring of a graph $G$ colors the vertices of $G$ using the colors on $H$ while respecting the adjacencies in $H$. 
Formally, let $\phi$ be a labeling (i.e. an injection on the vertex set) of $H$. A \textit{proper $(H, \phi)$-coloring} of $G$ is a map $\kappa: V(G) \rightarrow \phi(V(H))$ such that $v_1 \sim_G v_2$ implies $\phi^{-1}(\kappa(v_1)) \sim_H \phi^{-1}(\kappa(v_2))$. We say that $G$ is \textit{$\it{H}$-colorable} if there is a labeling of $H$, $\phi$, such that there is a proper $(\it{H},\phi)$-coloring of $G$.  We define our $H$-chromatic symmetric functions as follows:

\begin{dfn}\label{dfn.HCSF}
Let $G, H$ be two graphs and $n$ be the number of vertices in $H$. Let $V(G) = \{v_1, v_2, ..., v_k\}$. Then, $$w_G^H := \sum_{\phi} \sum_{\kappa} x_{\kappa(v_1)}...x_{\kappa(v_k)},$$ where $\phi$ ranges over all labelings $\phi: V(H) \rightarrow [n]$ of $H$ and $\kappa$ ranges over all proper $(H,\phi)$-colorings, is a symmetric polynomial in indeterminates $x_1, x_2, ..., x_n$. Therefore, $w_G^H$ extends naturally to a homogeneous symmetric function of degree $k$. We denote this symmetric function by $X_G^H$, and call it the \textit{$H$-chromatic symmetric function} of $G$.
\end{dfn}

\begin{figure}[ht]
\centering
\includegraphics[scale=0.2]{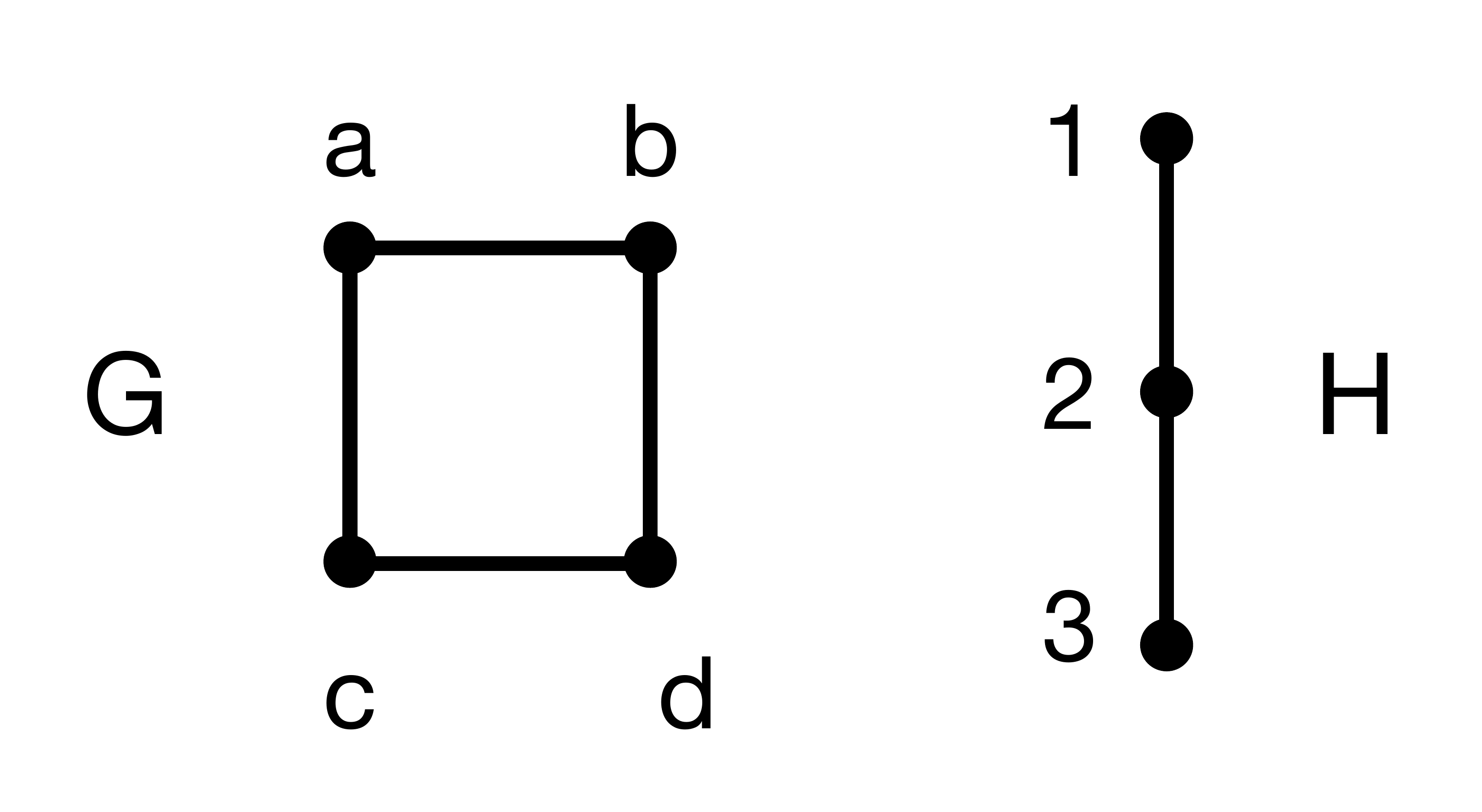}
\caption{Example of a $H$-chromatic symmetric function: $8x_1x_2^2x_3 + 8x_1^2x_2x_3 + $ $8x_1x_2x_3^2 + 8x^2_1x_2^2$ $ + 8 x_2^2x_3^2 + 8 x_1^2 x_3^2$.}
\end{figure}

Note that it might have been tempting to define the $H$-chromatic symmetric function as $\sum_\phi \sum_{\kappa} \prod_{j=1}^n x_{\kappa(v_j)}$, where $\phi$ ranges over all possible labelings $\phi: V(H) \rightarrow \N$, $\kappa$ ranges over all possible $(H,\phi)$-colorings, and $V(G) = \{v_1, v_2, ..., v_n\}$. However, unless we restrict the colours allowed for $H$, this definition will, in general, result in possibly infinite coefficients. A simple example that demonstrates this  is $G = K_1$ and $H = K_2$, in which case the term $x_1$  appears infinitely many times.

When $n \geq k$, $X_G^{K_n}$ is simply a scalar multiple of $X_G$. In contrast to the ordinary chromatic symmetric functions, the $H$-chromatic symmetric functions may sometimes vanish---namely, when $G$ is not $H$-colorable. It is natural to define $X_G^H = 0$ in that case. When $H$ has a loop, however, then $X_G^H$ never vanishes.   While we could confine ourselves to simple $H$'s, we will discuss a specific type of non-simple graph in Section \ref{augmentedstars}.\\

Throughout the paper we will address the following specific questions:  
\begin{itemize}
\item Given a pair of non-isomorphic graphs $G_1, G_2$, does there exist $H$ such that $X_{G_1}^H \neq X_{G_2}^H$?(Section \ref{auto})
\item Given a set of (mutually non-isomorphic) graphs $\mathcal{G}$, what is the smallest size of a set of graphs $\mathcal{H}$ such that for any $G_1, G_2 \in \mathcal{G}$, there exists an $H \in \mathcal{H}$ with $X_{G_1}^H \neq X_{G_2}^H$? (Section \ref{distinguishers})

\item Given a fixed complete bipartite graph $H$, and two non-isomorphic graphs $G_1, G_2$,  what is the easiest way to determine whether $G_1$ and $G_2$ are $K_{m,n}$-chromatically equivalent? (Section \ref{completebipartite})

\item Given a pair of non-isomorphic graphs, $G_1, G_2$, 
what are conditions that guarantee $X_{G_1}^{S_{n+1}^1} = X_{G_2}^{S_{n+1}^1}$,  where $S_{n+1}^{1}$ is the star graph on $n+1$ vertices  ($n \geq 2$) with an extra loop at its center? (Section \ref{augmentedstars})
\item What is the description of a set of $H$-chromatic symmetric functions that is a basis for $\Lambda^n$, the space of symmetric functions of degree $n$?  (Section \ref{completebasis})
\item If $G$ is the star graph $S_{n}$, are the $H$-chromatic symmetric functions unique? (Section \ref{degreeseq})
\item If $G$ is fixed, is it possible to find a set of $H$-chromatic functions that is a basis for $\Lambda^n$? (Section \ref{furtheravenues})

\end{itemize}

As we said at the beginning of this introduction, there are two overarching questions to consider, one where we fix $H$ and consider when
$X^{H}_{G_{1}} = X^{H}_{G_{2}}$,
and one where we fix $G$ and consider when
 $X^{H_{1}}_{G} = X^{H_{2}}_{G}$.  In fact, the majority of the problems we consider in this paper are in the first category.  Among our results, those that involve determining when $X^{H_{1}}_{G} = X^{H_{2}}_{G}$, or related problems, are  Results \ref{prop311},  \ref{lem.stardistinctH}--\ref{prop516},  \ref{prop.augstarCSF},  \ref{lem.Hcharac}--\ref{prop812}, and \ref{prop91}.
 A possible avenue for future work would be to consider more widely when $X^{H_{1}}_{G} = X^{H_{2}}_{G}$.

The paper is organized as follows. Section \ref{background} provides background definitions both from graph theory and symmetric function theory.  Section \ref{basics} introduces the fundamental concepts specific to $H$-chromatic symmetric functions and establishes a selection of general results.  Section \ref{distinguishers}  
explores constructing a single $H$ that distinguishes all graphs in a set $\mathcal{G}$.
Section \ref{completebipartite} considers questions about $H$-chromatic symmetric functions where $H$ is a complete bipartite graph. Section \ref{augmentedstars} looks at uniqueness questions when $H$ is an augmented star graph---a star graph with an extra loop at its center. Section \ref{completebasis} provides a basis for $\Lambda_n$ using $H$-chromatic symmetric functions defined by complete multibipartite graphs.  Section \ref{degreeseq} builds on Section \ref{completebasis} and  furthers our investigation into $H$-chromatic symmetric functions where $H$ is complete bipartite, in particular, the star graph. We close in Section \ref{furtheravenues} with possible future directions of research.


\section{Background}
\label{background}

We will first recall some standard notation and concepts that will be useful later. See \cite{BondyMurty} and \cite{Rosen} for background details on graphs. A (undirected) \textit{graph} consists of a vertex set $V(G)$, an edge set $E(G)$, and an incidence relation $\psi_G$ that maps each edge to an unordered pair of vertices. These two vertices are called the \textit{endpoints} of the edge. An edge is called a \textit{loop} if its two endpoints are the same. A graph has parallel edges, also known as multiple edges, if there are two or more edges between the same two vertices in the graph. A graph is called simple if it has no loops or parallel edges. When $G$ is a simple graph, we sometimes omit its incidence relation, and denote the edges $(u,v)$ directly by their endpoints.

The number of vertices in a graph $G$ is $|V| = \abs{V(G)}$, and the number of edges is $|E| = \abs{E(G)}$. Two vertices $v_1,v_2 \in V(G)$ are said to be \textit{adjacent}, written $v_1 \sim_G v_2$, if they are the endpoints of an edge in $E(G)$. The \textit{degree} of a vertex $v$, written $\deg(v)$, is the number of terminal points of edges incident with $v$. Therefore, a loop contributes 2 to the degree of a vertex. A graph $G' = (V(G'), E(G'), \psi_{G'})$ is a \textit{subgraph} of a graph $G = (V(G), E(G), \psi_{G})$ if $V(G') \subset V(G)$, $E(G') \subset E(G)$, and $\psi_{G'} = \psi_G|E(G')$. An {\em induced subgraph} $G'$ of $G$ is a subset of vertices in $G$ paired with any edges from $G$ whose endpoints are in this subset. More specifically, given a graph $G=(V,E)$, an induced subgraph, $H$, on the vertex set $V' \subset V$ has vertex set $V'$ and edge set 
$E'= \{(i,j) | (i, j) \in \binom{V'}{2} \cap E\}$. A graph $G$ is said to be $H$\textit{-free} if $H$ is not an induced subgraph of $G$. 

Given two graphs $G, H$, we say $G$ is \textit{isomorphic} to $H$ if there exists a pair of bijections $\phi: V(G) \rightarrow V(H), \theta: E(G) \rightarrow E(H)$, such that for every $e \in E(G)$, $\psi_G(e) = (u, v)$ if and only if $\psi_H(\theta(e)) = (\phi(u), \phi(v))$. We call such a pair of bijections an \textit{isomorphism} from $G$ to $H$. When $G, H$ are both simple, an isomorphism can equivalently be represented by a single bijection $\phi: V(G) \rightarrow V(H)$ such that $\phi(u) \sim_H \phi(v)$ if and only if $u \sim_G v$ for every $u, v \in V(G)$. We shall also call such $\phi$ an isomorphism. An \textit{automorphism} of $G$ is an isomorphism from $G$ to itself.

Let $G = (V, E, \psi_G)$ and $X \subset V$. The \textit{coboundary} of $X$, denoted by $\partial X$, is the set $\{e \in E \mid \text{exactly one endpoint of $e$ is in $X$}\}$. A graph $G$ is said to be \textit{disconnected} if there exists a non-empty proper subset of $V$ whose coboundary is empty. Otherwise, $G$ is said to be \textit{connected}. A \textit{connected component}, or sometimes simply \textit{component}, of $G$ is a maximal (under set inclusion) connected subgraph of $G$. Two graphs are \textit{disjoint} if their vertex sets (and edge sets) are disjoint. Let $G_1 = (V_1, E_1, \psi_{G_1}), G_2 = (V_2, E_2, \psi_{G_2})$ be two disjoint graphs, their disjoint union, denoted by $G_1 \uplus G_2$, is the graph $(V_1 \uplus V_2, E_1 \uplus E_2, \psi_{G_1 \uplus G_2})$, where $\psi_{G_1 \uplus G_2}$ is the map whose domain is $E_1 \uplus E_2$ and $\psi_{G_1 \uplus G_2}|E_1 = \psi_{G_1}$, $\psi_{G_1 \uplus G_2}|E_2 = \psi_{G_2}$.

Let $G = (V, E, \psi_G)$ and $X \subset V$. $X$ is called an \textit{independent set} of $G$ if the vertices in $X$ are mutually non-adjacent. $X$ is called a \textit{clique} if the vertices in $X$ are mutually adjacent. A \textit{complete graph} on $n$ vertices, denoted by $K_n$, is a simple graph on $n$ vertices whose entire vertex set is a clique. An \textit{edgeless graph} (or \textit{null graph}) on $n$ vertices, denoted by $N_n$, is the graph on $n$ vertices whose edge set is empty. A \textit{cycle} of size $n$, denoted by $C_n$, is a connected graph on $n$ vertices such that every vertex has a degree of $2$. (So $C_1$ is a single loop.) A \textit{path} of size $n$, denoted by $P_n$, is $C_n$ with an edge removed. 

We say that a graph is \textit{bipartite} if its vertices $V$ can be partitioned into two independent sets, $V_1$ and $V_2$. A simple graph  is \textit{complete bipartite} if it is bipartite such that every vertex in $V_1$ is adjacent to every vertex in $V_2$. We denote by $K_{a, b}$ the complete bipartite graphs whose two parts contain $a, b$ vertices, respectively. A \textit{star graph} on $(n+1)$ vertices, denoted by $S_{n+1}$, is the complete bipartite $K_{1,n}$. We call the vertex of degree $n$ the \textit{center} of the star, and the rest the \textit{leaves}. A \textit{complete n-partite} graph is a simple graph whose vertex set can be partitioned into $n$ independent sets, such that there is an edge between any two vertices that are not in the same independent set. A graph is \textit{cyclic} if it contains at least one subgraph that is a cycle. Otherwise, it is \textit{acyclic}. A \textit{forest} is an acyclic graph. A \textit{tree} is a connected forest.

Let $G = (V, E, \psi_G)$ be a simple graph. The \textit{complement} of $G$, denoted by $\overline{G}$, is $\overline{G} = (V, {V \choose 2} \setminus E)$. That is, given $u \neq v$, $(u, v) \in E(\overline{G})$ if and only if $(u, v) \notin E(G)$.

Let $G = (V, E, \psi_G)$ be a graph. A \textit{proper coloring}, or sometimes simply \textit{coloring}, of a graph is a function $\kappa: V \rightarrow \N$ that satisfies $\kappa(u) \neq \kappa(v)$ if $u \sim v$. In this definition, we do not require $u \neq v$. Hence, a graph with a loop can never have a proper coloring. The \textit{chromatic number} of $G$, denoted by $\chi(G)$, is $\min \{\abs{\kappa(V)} \mid \kappa \text{ is a proper coloring of } G\}$.

Let $G = (V(G), E(G), \psi_G), H = (V(H), E(H), \psi_H)$ be two graphs. A \textit{labeling} of $H$ is an injective function $\phi: V(H) \rightarrow \N$. Let $\phi$ be a labeling of $H$. Recall from Section \ref{Introduction} that a \textit{proper $(H, \phi)$-coloring} of $G$ is a map $\kappa: V(G) \rightarrow \phi(V(H))$ such that $v_1 \sim_G v_2$ implies $\phi^{-1}(\kappa(v_1)) \sim_H \phi^{-1}(\kappa(v_2))$. When we have a known fixed labeling $\phi$, we sometimes omit it and write just $H$-coloring. We also use $H$-coloring as a name of the subject, even if no labelings may have been presented. This will be the main subject of this paper. We shall remark that we do not require $G, H$ to be simple graphs. For example, when $H$ contains a loop, then a proper $H$-coloring may assign the same color to two adjacent vertices in $G$. In particular, a proper coloring is not always a proper $H$-coloring.

We say that $\lambda = (\lambda_1,\dots, \lambda_k)$ is a \textit{partition} of a positive integer $n$, written $\lambda \vdash n$, if $\lambda$ is a list of weakly decreasing positive integers whose sum is $n$. We refer to each $\lambda_i$ as a \textit{part}. The function $r_i(\lambda)$ gives the number of parts of $\lambda$ equal to $i$. The length of this partition is denoted by $l(\lambda) := k$. Let $\kappa: V(G) \rightarrow \N$ be a coloring. Let $k_i$ be the number of vertices that are colored by $i$. If we arrange the non-zero terms in $\{k_i\}_{i \in \N}$ in weakly decreasing order, we will obtain a partition of $n$, where $n$ is the number of vertices of $G$. We call this partition the \textit{type} of $\kappa$.

\textit{Symmetric functions}, $f(x)$, are formal power series in infinitely many variables, $x = (x_1,x_2,\ldots)$, that are invariant under permutations of the subscripts. See \cite{MacDon} or \cite{StanleyBook} for background on notation and terminology for symmetric functions. 
The terms of a symmetric function are referred to as \textit{monomials}.  A symmetric function is said to be a \textit{homogeneous symmetric function of degree $k$} if the sum of the powers of each monomial is equal to $k$. The space of homogeneous symmetric functions of degree $k$ is denoted $\Lambda^k$, and the graded vector space $\Lambda = \oplus_{k=0}^\infty \Lambda^k$ is the space of symmetric functions. Let $\lambda = (\lambda_1, \lambda_2, ..., \lambda_l)$ be a partition of $k$. The $\textit{monomial symmetric functions}$ are defined by $$m_\lambda = \sum_{\sigma} x_{\sigma(1)}^{\lambda_1} x_{\sigma(2)}^{\lambda_2} ... x_{\sigma(l)}^{\lambda_l},$$ where $\sigma$ varies over all $l$-tuples of positive integers. The set $\{m_\lambda \mid \lambda \vdash k\}$ is a basis for $\Lambda^k$.

Other classical symmetric functions are the {\em elementary} and {\em power sum symmetric functions}:
$$e_n= m_{(1^n)} = \sum_{i_1 < i_2 < \ldots < i_n} x_{i_1} x_{i_2} x_{i_3} \cdots x_{i_n}, $$ and
$$p_n = m_{(n)} = \sum_{i \geq 1} x_i^n $$ respectively.  For $k<0$, we define $p_k = e_k =0$, and $p_0 = e_0 = 1$.
Then we can define
$$e_\lambda = e_{\lambda_1}e_{\lambda_2}\cdots e_{\lambda_l}$$
and $$p_\lambda = p_{\lambda_1}p_{\lambda_2}\cdots p_{\lambda_l}$$
for partition $\lambda \vdash n$.
The {\em Schur functions} are defined as:
$$ s_{\lambda}=\mathrm{det}(e_{\lambda^\prime_i-i+j}), $$
where $i,j \in \{1,2,\ldots,l\}$
and $\lambda'$ is a conjugate partition of a given partition $\lambda=(\lambda_1,\ldots,\lambda_l)$.

In 1995, Stanley introduced a generalization of the chromatic symmetric polynomial of a graph---the \textit{chromatic symmetric function}. We will recall its basic definition. For a more in-depth cover refer to his original paper, \cite{stanley1995symmetric}. 

 Let $G = (V, E, \psi)$ be a (not necessarily simple) graph. Let $V = \{v_1,v_2,\ldots,v_n\}$, we define \textit{chromatic symmetric function} of $G$ by
\[
X_G = \sum_{\kappa} x_{\kappa(v_1)}x_{\kappa(v_2)}\ldots x_{\kappa(v_n)}
\]
where the sum ranges over all proper colorings of $G$.

Throughout this paper, we will refer to Stanley's chromatic symmetric functions as ``ordinary chromatic symmetric functions".

In this paper, we will mostly be concerned with simple graphs, although we will look at graphs with loops in some later sections. From now on, unless otherwise stated, whenever we say graphs, we \textit{always} mean simple graphs.

Throughout this paper, the multinomial coefficients $${n \choose n_1, n_2, ..., n_k} := \frac{n!}{n_1! n_2! ... n_k!}$$ will be used for several times, where $n = \sum_{j=1}^k n_j$. For notational cleanness, however, we require only that $n \geq \sum_{j=1}^k n_j$ and we interpret that $${n \choose n_1, n_2, ..., n_k} := \frac{n!}{n_1! n_2! ... n_k! (n - n_1 - n_2 - ... - n_k)!}.$$

\section{\textit{H}-chromatic symmetric functions}
\label{basics}

Before we explore the properties of $H$-chromatic symmetric functions we will introduce some additional concepts.
As in Definition \ref{dfn.HCSF}, computing the $H$-chromatic symmetric functions requires considering all labelings $\phi: V(H) \rightarrow [n]$. The total number of possible labelings is $n!$, which grows rapidly as $n \rightarrow \infty$. Nevertheless, the following Lemma allows us to compute $X_G^H$ using only one labeling.

\begin{lem}\label{lem.singlerep}
Let $H$ be a graph on $n$ vertices with an arbitrary fixed labeling $\phi$. Suppose $G$ is a graph on $k$ vertices, and $\lambda = (\lambda_1, \lambda_2, ..., \lambda_l) \vdash k$. Let $d_\lambda$ denote the number of $(H,\phi)$-colorings of $G$ of type $\lambda$. Then, $$X_G^H = \sum_{\lambda \vdash k, l(\lambda) \leq n} c_\lambda m_\lambda,$$ where $$c_\lambda = \frac{d_\lambda n!}{{n \choose {r_1(\lambda), r_2(\lambda), ..., r_k(\lambda)}}}.$$ Note that $n \geq \sum_{i=1}^k r_i(\lambda)$ and we recall the convention of writing multinomial coefficients introduced in Section \ref{background}.
\end{lem}

\begin{proof}
Since each labeling $\phi$ of $H$ induces the same number of $(H,\phi)$-colorings of type $\lambda$, there will be $d_\lambda n!$ terms of type $\lambda$ in $w_G^H$, where $w_G^H$ is defined in Definition \ref{dfn.HCSF}. To compute the number of distinct monomials of type $\lambda$ in indeterminates $x_1, x_2, ..., x_n$, we choose $r_1$ from the $n$ indeterminates and assign them degree 1, and then $r_2$ from the rest of $(n-r_1)$ indeterminates and assign them degree 2. This process is done sequentially for $j = 1, 2, ..., k$, where in each iteration, $r_j$ indeterminates are chosen from the remaining $(n-\sum_{i=1}^{j-1} r_i)$ and assigned degree $j$. Therefore, there will be ${n \choose {r_1(\lambda), r_2(\lambda), ..., r_k(\lambda)}}$ distinct monomials of type $\lambda$. Hence, $c_\lambda$ is the coefficient of any particular monomial of type $\lambda$ in $w_G^H$. The conclusion follows from the definition of $X_G^H$.
\end{proof}

In light of Lemma \ref{lem.singlerep}, it is convenient to define the $n$-augmented monomials as follows. We include the $n$ in the name to distinguish from the augmented monomials used by Stanley \cite{stanley1995symmetric}.

\begin{dfn}\label{dfn.augmonomial}
Given a number $n \in \N$ and a partition $\lambda \vdash k$ with $l(\lambda) \leq n$, we define the \textit{n-augmented monomial} $m_\lambda^n$ by 
\begin{equation}\label{eqnmulti}
m_\lambda^n = \frac{n!}{{n \choose {r_1(\lambda), r_2(\lambda), ..., r_k(\lambda)}}}m_\lambda.
\end{equation}
\end{dfn}

Suppose $G, H$ are graphs on $k, n$ vertices, respectively. Clearly, when $k \leq n$, the set $\{m_\lambda^n\}_{\lambda \vdash k}$ is a basis for $\Lambda^k$. When $k>n$, the only non-vanishing monomials in $X_G^H$ must have type of length no larger than $n$. Therefore, $X_G^H \in \spann\{m_\lambda\}_{\lambda \vdash k, l(\lambda) \leq n} = \spann\{m^n_\lambda\}_{\lambda \vdash k, l(\lambda) \leq n}$. Therefore, we can always write $X_G^H$ uniquely as $\sum_{\lambda \vdash k, l(\lambda) \leq n} d_\lambda m_\lambda^n$. In the rest of this paper, when we write $X_G^H = \sum_{\lambda \vdash k} d_\lambda m_\lambda^n$, we always assume $d_\lambda = 0$ for every $l(\lambda) > k$, although $m_\lambda^n$ has not been formally defined. Lemma \ref{lem.singlerep} says $d_\lambda$ is the number of $(H,\phi)$-colorings of $G$ of type $\lambda$ for any fixed $\phi$.\\

It is worth mentioning the caveat that the transformation between $m_\lambda$ and $m_\lambda^n$ is not uniform over all $\lambda$. For example, if we write $X_G^H = \sum_\lambda c_\lambda m_\lambda = \sum_\lambda d_\lambda m_\lambda^n$, then it is \textit{not} true in general that $c_{\lambda_1}/d_{\lambda_1} = c_{\lambda_2}/d_{\lambda_2}$ for any pair of partitions $\lambda_1, \lambda_2 \vdash k$. Furthermore, if we have $n_1, n_2 \geq l(\lambda)$, then $$m_\lambda^{n_1} = \frac{(n_1 - \sum r_j(\lambda))!}{(n_2 - \sum r_j(\lambda))!}m_\lambda^{n_2} = \frac{(n_1 - l(\lambda))!}{(n_2 - l(\lambda))!}m_\lambda^{n_2}.$$ This relation, albeit elegant, is seldom used in the rest of this paper.\\

A natural question is whether we can actually obtain $H$-chromatic symmetric functions that are not ordinary chromatic symmetric functions?  Indeed, we can.   To see this, we observe that an ordinary chromatic symmetric function of degree $k$ always contains the monomial $m_{(1^k)}$. By containing the monomial $m_{(1^k)}$ we mean that if we write the symmetric function in terms of the monomial basis, the coefficient of $m_{(1^k)}$ is non-zero. On the other hand, every monomial $m_\lambda$ can be written as a multiple of a $H$-chromatic symmetric function, as the following proposition shows.

\begin{prop}\label{prop.monomial}
Given a partition $\lambda = (\lambda_1, \lambda_2, ..., \lambda_l) \vdash k$, denote by $K_\lambda$ the complete $l$-partite graph whose parts contain $\lambda_1, \lambda_2, ..., \lambda_l$ vertices, respectively. Then, $$X_{K_\lambda}^{K_l} = l! m_\lambda^l.$$

Note that in particular, if $\lambda = (n)$, then \begin{equation*}
    X_{\overline{K_{n}}}^{K_1} = m_{(n)}
\end{equation*}
\end{prop}

Given Lemma \ref{lem.singlerep}, the proof of Proposition \ref{prop.monomial} is obvious, and thus omitted. We shall see a stronger version of the above proposition in Section \ref{completebasis}.

In the setting of ordinary chromatic symmetric functions, we have the property that $X_{G_1 \uplus G_2} = X_{G_1}X_{G_2}$ \cite{stanley1995symmetric}. This is, however, not true in general for $H$-chromatic symmetric functions when we fix a $H$. The easiest example to show this is  $G_1 = G_2 = H = K_1$, in which case $X_{G_1}^H X_{G_2}^H$ contains the monomial $m_{(1,1)}$, whereas $X_{G_1 \uplus G_2}^H$ does not. Another nice property that holds for ordinary chromatic symmetric functions but not $H$-chromatic symmetric functions is the weighted deletion-contraction result. To formulate the exact statement the notion of weighted chromatic symmetric functions will be useful, which we will not introduce here. Interested readers should refer to Crew and Spirkl's paper for a detailed discussion \cite{crew2020deletion}. The main difference in $H$-colorings is that a proper $H$-coloring of $G \setminus (u,v)$ that colors $u, v$ differently is not necessarily a proper $H$-coloring of $G$. On the other hand, if we replace $H$-coloring by ordinary coloring, then the implication in the previous statement will hold, and is a key observation used in the proof of the weighted deletion-contraction result for ordinary chromatic symmetric functions.\\

In the previous paragraph, we see that there is very little we can say when $G$ is written as a disjoint union of its components. When we consider the connected components of $H$, there is a simple but useful result. The restriction here is that $G$ has to be connected.

\begin{prop}\label{prop.disjointH}
Let $G, H_1, H_2, ..., H_l$ be connected graphs, and $H = \biguplus_{j=1}^l H_j$. Let $n_j = \abs{V(H_j)}$ and $n = \abs{V(H)}$. For each $1 \leq j \leq l$, we write $X_G^{H_j} = \sum_{\lambda \vdash k} c^j_\lambda m^{n_j}_\lambda$. Then, $$X_G^H = \sum_{\lambda \vdash k} c_\lambda m_\lambda^n,$$ where $c_\lambda = \sum_{j=1}^l c^j_\lambda$.
\end{prop}

\begin{proof}
Fix a labeling $\phi$ on $H$. Since $G$ is connected, we can only use one component of $H$ to color $G$. Hence, the number of proper $(H,\phi)$-colorings of $G$ is the sum of the number of proper $(H_j,\phi_{|V(H_j)})$-colorings of $G$. The conclusion follows from Lemma \ref{lem.singlerep}.
\end{proof}

\subsection{Some fundamental $H$-chromatic equivalence and uniqueness results}
In this subsection, we deal with problems involving situations when $H$-chromatic symmetric functions of two graphs $G_1$ and $G_2$ for a fixed $H$ are not equal. We explore results pertaining to when $X_{G_1}^{H} \neq X_{G_2}^{H}$.

\begin{dfn} \label{def.generalequiv}
We say that two graphs $G_1$ and $G_2$ are \textit{$H$-chromatically equivalent} if and only if $X_{G_1}^{H} = X_{G_2}^{H}$. We note that this is an equivalence relation that partitions the set of graphs into equivalence classes. Elements in the same equivalence class are $H$-chromatically equivalent to each other.
\end{dfn}

\begin{dfn}
 We say that a graph $G$ is \textit{$H$-chromatically unique} in a set of graphs $S$ if there is no graph $G_0 \in S$ such that $X_{G}^{H} = X_{G_0}^{H}$.
\end{dfn}

\begin{prop}\label{prop.difnocol}
    If two $H$-colorable graphs $G_1$ and $G_2$ differ in the minimum number of colors needed to $H$-color them or the maximum number of colors that can possibly $H$-color them, they cannot have the same $H$-chromatic symmetric function. 
    \begin{proof}
       Suppose two $H$-colorable graphs $G_1$ and $G_2$ differ in the minimum (resp.\ maximum) number of colors needed to $H$-color them. 
       
       Suppose $$X_{G_1}^H = \sum_{\lambda_a \vdash k} c_{\lambda_a} m_{\lambda_a}.$$
       
       Let $h_{1, min}$ (resp. $h_{1, max}$) denote the minimum (resp.\ maximum) number of colors needed to $H$-color $G_1$.
       
       Then there exists some $m_{\lambda_{a_i}}$, $m_{\lambda_{a_j}}$ such that $\lambda_{a_i}, \lambda_{a_j}$ have length $h_{1, min}$ (resp. $h_{1, max}$) 
       
       Suppose $$X_{G_2}^H = \sum_{\lambda_b \vdash k} c_{\lambda_b} m_{\lambda_b}.$$
       
       Let $h_{2, min}$ (resp. $h_{2, max}$) denote the minimum (resp.\ maximum) number of colors needed to $H$-color $G_2$. 
       
       Then there exists some $m_{\lambda_{b_i}}$, $m_{\lambda_{b_j}}$ such that $\lambda_{b_i}, \lambda_{b_j}$ have length $h_{2, min}$ (resp. $h_{2, max}$) 
       
       Without loss of generality, suppose $h_{1, min} < h_{2, min}$ or $h_{1, max} < h_{2, max}$.
       
       If $h_{1, min} < h_{2, min}$ or $h_{1, max} < h_{2, max}$, then $m_{\lambda_{b_1}} , m_{\lambda_{b_2}} , m_{\lambda_{b_3}}, \dots$ will span monomial symmetric functions corresponding to partitions with length $h_{2, min}$ or $h_{2, max}$ but not those of length $h_{1, min}$ and $h_{1, max}$.
       
       If the two graphs differ in the minimum (resp.\ maximum) number of colors needed to $H$-color them, then when we write $X_{G_1}^H$ and $X_{G_2}^H$ as linear combinations of monomial symmetric functions whose partitions correspond to viable $H$-colorings, we would see terms corresponding to partitions of different lengths so $X_{G_1}^H$ and $X_{G_2}^H$ cannot possibly be the same.
       
    \end{proof}
\end{prop}

\begin{cor}\label{cor.difbi}
    Suppose $G_1$ and $G_2$ are both $H$-colorable, i.e. $X_{G_1}^{H} , X_{G_2}^{H}$ are non-zero. If $G_1$ is bipartite and $G_2$ is not bipartite then $G_1$ and $G_2$ cannot have the same $H$-chromatic symmetric functions.
    \begin{proof}
       $G_1$ is 2-colorable, whereas $G_2$ is not. Applying Proposition \ref{prop.difnocol} above gives us our result.
    \end{proof}
\end{cor}

\begin{prop}\label{prop.unequalbipartitions} Let $H$ be a non-edgeless graph.
    Suppose $B_1$ is a connected bipartite graph such that $V(B_{1})$ is partitioned into two disjoint sets $V_{1} , V_{2}$. Suppose $B_2$ is a connected bipartite graph such that $V(B_{2})$ is partitioned into two disjoint sets $W_{1} , W_{2}$. Suppose $\{ |V_{1}|, |V_{2}|\} \neq \{ |W_{1}|, |W_{2}|\}$, i.e. the bipartitions of $B_1$ and $B_2$ are not the same. Then $B_1$ and $B_2$ are not $H$-chromatically equivalent, i.e. $X_{B_1}^{H} \neq X_{B_2}^{H}$.

\begin{proof}
   Consider the partition $\lambda_1 = (|V_{1}|, |V_{2}|)$ that corresponds to 2-coloring $B_1$. Consider the partition $\lambda_2 = (|W_{1}|, |W_{2}|)$ that corresponds to 2-coloring $B_2$. Now the bipartition of a connected bipartite graph is unique, so its 2-coloring is unique up to switching colors. It follows that these are the only partitions corresponding to 2-coloring $B_1$ and $B_2$. The partitions of length 2 that correspond to 2-coloring $B_1$ and $B_2$ are not the same, so the corresponding $m_{\lambda}$s in $X_{B_1}^{H}$ and $X_{B_2}^{H}$ are not the same. Hence $X_{B_1}^{H}$ and $X_{B_2}^{H}$ cannot possibly be the same.
\end{proof}
\end{prop}

\begin{rem}
We note that disconnected bipartite graphs can have non-unique bipartitions. So the analogue of Proposition \ref{prop.unequalbipartitions} for non-connected graphs $G_1 , G_2$ would involve comparing the set of possible bipartitions of graphs $G_1 , G_2$.
\end{rem}

\begin{prop}
Let $H_1$ and $H_2$ be two graphs that both have $k$ vertices, but different number of edges. Suppose the number of edges in $H_1$ is given by $|E(H_1 )|$ and the number of edges in $H_2$ is given by $|E(H_2 )|$. Let $G$ be a complete bipartite graph on at least two vertices. Suppose $G$ is $H_1$-colorable and $H_2$-colorable, i.e. $X_{G}^{H_1} , X_{G}^{H_2}$ are non-zero. Then $X_G^{H_1} \neq X_G^{H_2}$.
\label{prop311}
\end{prop}

\begin{proof}
 Let $G = K_{m,n}$ for some $m,n$. Consider the coefficient of the $k$-augmented monomial $m_{(m,n)}^k$. The number of ways of coloring $G$ using two colors, which must by definition be adjacent in $H_1$ and $H_2$, is equivalent to counting the number of edges in each of $H_1$ and $H_2$. If $|E(H_1)| \neq |E(H_2)|$, then they will generate different coefficients of $m_{(m,n)}^k$. And since they are on the same number of vertices, then $X_G^{H_1}$ and $X_G^{H_2}$ will have different coefficients for $m_{(m,n)}^k$. 
\end{proof}

\begin{cor} \label{cor.staruniq} Suppose $H$ is a non-edgeless graph. The star graph is $H$-chromatically unique in the set of connected graphs.
\begin{proof}
   Suppose we partition the vertices of the star into two sets, $V_1$ and   $V_2$. The star $S_k$ is the only connected bipartite graph that corresponds to the case where $|V_{1}| = 1, |V_{2}| = k-1$. From Proposition \ref{prop.unequalbipartitions}, the star is not $H$-chromatically equivalent to other bipartite graphs. Similarly, Proposition \ref{prop.difnocol} implies the star is not $H$-chromatically equivalent to non-bipartite graphs.
\end{proof}
\end{cor}

\begin{cor} \label{cor.knuniq} Suppose a graph $H$ can be used to $H$-color the complete graph $K_n$, i.e. $X_{K_n}^{H} \neq 0$.
Then, $K_n$ is $H$-chromatically unique in the set of connected graphs.
\begin{proof}
$K_n$ is the only $n$-partite graph on $n$ vertices. All other graphs with $n$ vertices are colorable with $n-1$ colors. So applying Proposition \ref{prop.difnocol} gives us our result.
\end{proof}
\end{cor}

In fact, this idea of $H$-chromatic uniqueness can be made even stronger.

\begin{prop}\label{prop.linind}
Consider a finite set of graphs $\{ G_{i} \}$ with corresponding $H$-chromatic symmetric functions $\{ X_{G_{i}}^{H} \}$. 

Suppose we write each $X_{G_{i}}^{H}$ in the form $$X_{G_i}^H = \sum_{\lambda_{i} \vdash k} c_{\lambda_{i}} m_{\lambda_i}$$

 If each $X_{G_{j}}^{H}$ written in this form contains a $m_{\lambda_{i}}$ term (with fixed partition and nonzero coefficient) that the other $H$-chromatic symmetric functions in the set $\{ X_{G_{i}}^{H} \} \setminus X_{G_{j}}^{H}$ do not have, then the functions in the set $\{ X_{G_{i}}^{H} \}$ are linearly independent. 

\begin{proof}
Since the monomial symmetric functions $\{ m_\lambda \}$ form a basis for $\Lambda$, the monomial symmetric function $m_{\lambda_i}$ cannot be written as a linear combination of other monomial symmetric functions $\{m_{\lambda_{k}}\}$ where $\lambda_k \neq \lambda_i$.
\end{proof}
\end{prop}

\begin{cor}\label{cor.monotermlinind}

Suppose we write some non-zero $X_{G}^{H}$ as follows: $$X_{G}^H = \sum_{\lambda \vdash k} c_{\lambda} m_{\lambda}$$

If $G$ is $H$-chromatically unique in a set $S$ due to having a component $m_{\lambda}$ (with nonzero coefficient) that no other $H$-chromatic symmetric function of a graph in $S$ has, then $X_{G}^{H}$ cannot be written as a linear combination of other $H$-chromatic symmetric functions of graphs in $S$. Note that this is a stronger condition than $H$-chromatic uniqueness. 
\end{cor}

\begin{proof}
This follows from Proposition \ref{prop.linind}.
\end{proof}

\begin{cor}\label{cor.bipartlinind} Suppose $H$ is a non-edgeless graph.
The set of $H$-chromatic symmetric functions for connected bipartite graphs on the same number of vertices with different bipartitions is linearly independent. Let $B_{i,j}$ denote a bipartite graph with bipartition $\{ i,j \}$. Suppose $i+j = n$. Then $\{ X_{B_{i,j}}^{H}\}_{i = 1, \dots, \lfloor \frac{n}{2} \rfloor, j = n-i}$ is linearly independent.
\begin{proof}
Apply Proposition \ref{prop.unequalbipartitions} and Proposition \ref{prop.linind}.
\end{proof}
\end{cor}

\begin{cor}
The $H$-chromatic symmetric functions for graphs with different number of vertices are linearly independent. Consider a finite set of graphs $\{ G_{i} \}$ with $G_i$ having $i$ vertices and corresponding $H$-chromatic symmetric function $ X_{G_{i}}^{H} $. Suppose each $G_i$ is $H$-colorable, i.e. each $X_{G_{i}}^{H}$ is non-zero. Then $\{ X_{G_{i}}^{H} \}_{i=1}^{n}$ is linearly independent.
\begin{proof}
Apply Proposition \ref{prop.linind}.
\end{proof}
\end{cor}

\begin{cor}\label{cor.chromnumlinind}
The set of $H$-chromatic symmetric functions for graphs with the same number of vertices but different chromatic numbers is linearly independent.
Consider a finite set of graphs $\{ G_{i} \}$ with $G_i$ having chromatic number $i$ and corresponding $H$-chromatic symmetric function $ X_{G_{i}}^{H} $. Suppose each $G_i$ is $H$-colorable, i.e. each $X_{G_{i}}^{H}$ is non-zero. Then $\{ X_{G_{i}}^{H} \}_{i=1}^{n}$ is linearly independent.
\begin{proof}
Apply Proposition \ref{prop.difnocol} and Proposition \ref{prop.linind}.
\end{proof}
\end{cor}

\begin{cor}\label{cor.starlinind} Suppose $H$ is a graph such that an $n$-vertex star graph $S_n$ and an $n$-vertex complete graph $K_n$ are $H$-colorable, i.e. $X_{S_n}^H , X_{K_n}^H$ are non-zero. 
Then $X_{S_n}^H$ and $X_{K_n}^H$ cannot be written as a linear combination of $H$-chromatic symmetric functions of other connected graphs.
\begin{proof}
Apply Proposition \ref{cor.staruniq}, and Corollaries \ref{cor.knuniq}, \ref{cor.monotermlinind}. The $H$-chromatic symmetric function of the star $S_n$ is the only $H$-chromatic symmetric function of a $n$-vertex connected graph that contains the term $m_{(1,n-1)}$ when written as a linear combination of monomial symmetric functions $m_\lambda$. The $H$-chromatic symmetric function of a complete graph $K_k$ is the only $H$-chromatic symmetric function of a $k$-vertex connected graph that has no monomials with partition length less than $k$.
\end{proof}
\end{cor}

Note that even though these results on $H$-chromatic equivalence and $H$-chromatic uniqueness are true for cases where $G, H$ are simple graphs, if we allow $H$ to have a loop, then Propositions \ref{prop.difnocol}, \ref{prop.unequalbipartitions}, and Corollaries \ref{cor.difbi}, \ref{cor.staruniq}, \ref{cor.knuniq},  \ref{cor.bipartlinind}, \ref{cor.chromnumlinind}, \ref{cor.starlinind} fail. If $H$ were a graph with just one vertex and a loop, then we could $H$-color all the vertices in some graph $G$ the same color. Thus $G$ would have the same $H$-chromatic symmetric function as other graphs with the same number of vertices as $G$. 


\subsection{Automorphisms of $G$ and $X_{G}^{H}$}
\label{auto}

We know that two non-isomorphic graphs can have the same chromatic symmetric functions \cite{crew2020vertex, stanley1995symmetric}. Given a $H$, can we always find two non-isomorphic graphs having the same $H$-chromatic symmetric function? The answer is affirmative: indeed, if $H$ has $n$ vertices, then any graph with a chromatic number larger than $n$ will have a $H$-chromatic symmetric function of $0$. A more interesting question is the following problem:

\begin{prob} Given a pair of non-isomorphic graphs $G_1, G_2$, can we always find a $H$ such that $X_{G_1}^H \neq X_{G_2}^H$?
\end{prob}

To answer this question, it will be instructive to note that the coefficient of $m_{(1^k)}$ in $X_G^H$ is closely related to the number of automorphisms of $G$.

\begin{prop}\label{prop.automorphism}
Let $G, H$ be two graphs and $X_G^H = \sum_{\lambda \vdash \abs{V(G)}} c_\lambda m_\lambda^{\abs{V(H)}}$. Then, $$c_{(1^{\abs{V(G)}})} = S(G, H) \abs{\aut(G)},$$ where $S(G, H)$ is the number of (not necessarily induced) subgraphs of $H$ that are isomorphic to $G$ and $\aut(G)$ is the set of automorphisms of $G$.
\end{prop}

\begin{proof}
We define an embedding of $G$ into $H$ to be an injective map $\psi: V(G) \rightarrow V(H)$ such that $u \sim_G v$ implies $\psi(u) \sim_H \psi(v)$. Given each embedding $\psi$, we can obtain a subgraph of $H$ isomorphic to $G$ by letting the vertex set be $\psi(V(G))$ and the edge set be $\{(\psi(u), \psi(v)) \mid (u, v) \in E(G)\}$. We denote this graph by $\Psi(\psi)$. In particular, $\Psi(\psi)$ is isomorphic to $G$. Conversely, given a subgraph $H'$ of $H$ that is isomorphic to $G$, and let $\psi: V(G) \rightarrow V(H')$ be an isomorphism. Then, we claim that $\Phi := \{\psi \circ \phi \mid \phi \in \aut(G)\}$ is exactly the set of embeddings whose image under $\Psi$ is $H'$. It is clear that every element in $\Phi$ is an embedding whose image under $\Psi$ is $H'$. Conversely, let $\psi'$ be an embedding of $G$ into $H$ and that $\Phi(\psi') = H'$. Then, $\psi^{-1} \circ \psi'$ is an automorphism of $G$, so $\psi' \in \Phi$. Hence, the total number of embeddings of $G$ into $H$ is $S(G, H) \abs{\aut(G)}$.

Now, we count the number of embeddings of $G$ into $H$ in another way. If we fix a labeling $\phi$ of $H$, then given an $H$-coloring of $G$ of type $(1^{\abs{V(G)}})$, it corresponds naturally to an injection from $V(G)$ to $V(H)$, and it is an embedding by the definition of an $H$-coloring. 
Given an embedding $\psi$, if we use the $\phi(\psi(u))$ to color $u$, then we will get a proper $H$-coloring. Hence, $c_{(1^{\abs{V(G)}})}$ is the number of embeddings of $G$ into $H$, and $c_{(1^{\abs{V(G)}})} = S(G, H) \abs{\aut(G)}$.
\end{proof}

\begin{cor}\label{cor.autcor}
Let $G_1, G_2$ be two non-isomorphic graphs on $k_1, k_2$ vertices, respectively. Then,
\begin{enumerate}
    \item If we write $X_{G_1}^{G_1} = \sum_{\lambda \vdash k_1} c_\lambda m_\lambda^{k_1}$. Then, $c_{(1^{k_1})}$ is the number of automorphisms of $G_1$.
    \item Let $X_{G_1}^{G_2} = \sum_{\lambda \vdash k_1} c_\lambda m_\lambda$. Then, $c_{(1^{k_1})} = 0$ if either $k_1 > k_2$ or $k_1 = k_2$ and $\abs{E(G_1)} \geq \abs{E(G_2)}$.
    \item There exists a graph $H$ such that $X_{G_1}^H \neq X_{G_2}^H$.
\end{enumerate}
\end{cor}

\begin{proof}
The first claim is immediate from Proposition \ref{prop.automorphism}. The second claim is obvious when $k_1 > k_2$. When $k_1 = k_2$ and $\abs{E(G_1)} \geq \abs{E(G_2)}$, $G_2$ cannot possibly contain an subgraph isomorphic to $G_1$ and the conclusion follows from again from Proposition \ref{prop.automorphism}. To prove the third claim, without loss of generality, we may assume the assumptions in the second claim holds. (Otherwise, interchange $G_1$ and $G_2$.) Now let $H = G_2$ and use the first two statements.
\end{proof}

Corollary \ref{cor.autcor} answers the question that we posed before Proposition \ref{prop.automorphism}. In fact, we can say even more. In the next section, we will show that given finitely many graphs $G_1, G_2, ..., G_l$, there exists a $H$ such that $X_{G_i}^H \neq X_{G_j}^H$ for every $i \neq j$. The constructive proof will be heavily based on Proposition \ref{prop.disjointH} and Corollary \ref{cor.autcor}. We now formalize these discussions.

\section{Graph distinguishers and uniform distinguishability of graphs}
\label{distinguishers}

As shown in the previous section, given any two non-isomorphic graphs $G_1, G_2$, there exists an $H$ such that $X_{G_1}^H \neq X_{G_2}^H$. In this section, we look at a generalization of this problem.

\begin{prob}\label{prob.distinguisher}
Given a set of (mutually non-isomorphic) graphs $\mathcal{G}$, what is the smallest size of a set of graphs $\mathcal{H}$ such that for any $G_1, G_2 \in \mathcal{G}$, there exists an $H \in \mathcal{H}$ with $X_{G_1}^H \neq X_{G_2}^H$?
\end{prob}

Let us give such a $\mathcal{H}$ a name and introduce additional terminology.

\begin{dfn}\label{dfn.distinguisher}
Let $\mathcal{G}$ be a set of (mutually non-isomorphic) graphs. A set of graphs $\mathcal{H}$ is called a \textit{$\mathcal{G}$-distinguisher} if for any $G_1, G_2 \in \mathcal{G}$, there exists an $H \in \mathcal{H}$ such that $X_{G_1}^H \neq X_{G_2}^H$. $\mathcal{G}$ is said to be \textit{finitely distinguishable} if there exists a finite $\mathcal{G}$-distinguisher. $\mathcal{G}$ is said to be \textit{uniformly distinguishable} if there exists a $\mathcal{G}$-distinguisher that contains exactly one graph.
\end{dfn}

Clearly, uniform distinguishability implies finite distinguishability, and the results in the last section shows that $\mathcal{G}$ is uniformly distinguishable if $\abs{\mathcal{G}} = 2$ and, consequently, $\mathcal{G}$ is finitely distinguishable if $\mathcal{G}$ is finite. To extend these results, we will have to construct new $H$'s. When both $\mathcal{G}$ and $\mathcal{H}$ are sets of connected graphs, Proposition \ref{prop.disjointH} seems a useful tool.
We introduce a lemma which shows we can assume a finite distinguisher consists only of connected graphs.

\begin{lem}\label{lem.connecteddistinguisher}
A set of connected graphs $\mathcal{G}$ is finitely distinguishable if and only if there exists a finite $\mathcal{G}$-distinguisher that consists of connected graphs only.
\end{lem}

\begin{proof}
 Suppose $\mathcal{H}$ is a finite distinguisher of $\mathcal{G}$. Let $\mathcal{H}'$ be the set of all connected components of graphs in $\mathcal{H}$. Let $G_1, G_2 \in \mathcal{G}$ be given, and let $H \in \mathcal{H}$ be such that $X_{G_1}^H \neq X_{G_2}^H$. Let $H'_1, H'_2, ..., H'_k$ be its connected components. By Proposition \ref{prop.disjointH}, there must be an $H'_j$ such that $X_{G_1}^{H'_j} \neq X_{G_2}^{H'_j}$. So $\mathcal{H}'$ is a finite $\mathcal{G}$-distinguisher. The other direction is obvious.
\end{proof}

Before we write down the main result in this section we will introduce a final piece of terminology.

\begin{dfn}\label{dfn.uniformboundedness}
Let $\mathcal{G}$ and $\mathcal{H}$ be two sets of graphs. We say that $\mathcal{G}$ is \textit{uniformly bounded with respect to $\mathcal{H}$} if the set $$\mathcal{C}_G^H := \{c_{G, \lambda}^H \mid G \in \mathcal{G}, H \in \mathcal{H}, X_G^H = \sum_{\lambda \vdash \abs{V(G)}} c_{G, \lambda}^H m_\lambda^{\abs{V(H)}}\}$$ is bounded.
\end{dfn}

In particular, when $\mathcal{G}, \mathcal{H}$ are finite, then $\mathcal{G}$ is uniformly bounded with respect to $\mathcal{H}$. We shall also remark that when $\mathcal{G}$ is uniformly bounded with respect to a finite set $\mathcal{H}$, then it is clear that $\mathcal{G}$ is also uniformly bounded with respect to $\mathcal{H}'$, which is defined in the proof of Lemma \ref{lem.connecteddistinguisher}.

\begin{prop}\label{prop.connecteddistinguisher}
Let $\mathcal{G}$ be a set of connected graphs. Suppose there exists a finite $\mathcal{G}$-distinguisher $\mathcal{H}$, and $\mathcal{G}$ is uniformly bounded by $\mathcal{H}$. Then, $\mathcal{G}$ is uniformly distinguishable.
\end{prop}

\begin{proof}
 By Lemma \ref{lem.connecteddistinguisher} and the remark before this proposition, it is safe to assume that $\mathcal{H} = \{H_1, H_2, ..., H_l\}$ is a finite $\mathcal{G}$-distinguisher that consists of connected graphs, and that $\mathcal{G}$ is uniformly bounded by $\mathcal{H}$. Let $N \in \N$ be an upperbound of $\mathcal{C}_G^H$ as in Definition \ref{dfn.uniformboundedness}. Now, set $$H := \biguplus_{j = 1}^l (2N)^j H_j,$$where $nG = \uplus_{j = 1}^n G$ is the graph obtained by taking the disjoint union of $n$ copies of $G$. We claim that $\{H\}$ is a $\mathcal{G}$-distinguisher.

Let $G_1, G_2 \in \mathcal{G}$ be given. We may assume $\abs{V(G_1)} = \abs{V(G_2)}$, for otherwise the degree of $X_{G_1}^H$ and that of $X_{G_2}^H$ are different. Note also, in particular, that in order for $\mathcal{H}$ to be a $\mathcal{G}$-distinguisher, it is clearly not the case that $X_{G_1}^H = X_{G_2}^H = 0$. We write $X_{G_1}^{H_{i}} = \sum_\lambda a_{i, \lambda} m^{\abs{V(H_i)}}_\lambda$ and $X_{G_2}^{H_{i}} = \sum_\lambda b_{i, \lambda} m^{\abs{V(H_i)}}_\lambda$. By Proposition \ref{prop.disjointH}, $X_{G_1}^H = \sum_{\lambda} a_\lambda m^{\abs{V(H)}}_\lambda$ and $X_{G_2}^H = \sum_{\lambda} b_\lambda m^{\abs{V(H)}}_\lambda$, where $a_\lambda = \sum_{j = 1}^l (2N)^j a_{j, \lambda}$ and $b_\lambda = \sum_{j = 1}^l (2N)^j b_{j, \lambda}$.

We let $m$ be the smallest number such that $\{H_m\}$ is a $\{G_1, G_2\}$-distinguisher, and let $\lambda$ be such that $a_{m, \lambda} \neq b_{m, \lambda}$. Then,
\begin{align*}
0 &< \abs{\sum_{j = 1}^m (2N)^j a_{j, \lambda} - \sum_{j = 1}^m (2N)^j b_{j, \lambda}}\\
&= \abs{(2N)^m a_{m, \lambda} - (2N)^m b_{m, \lambda}}\\
&= 2^m N^m \abs{a_{m, \lambda} - b_{m, \lambda}}\\
&\leq 2^m N^{m+1}.
\end{align*}
But $\sum_{j = m+1}^l (2N)^j a_{j, \lambda} - \sum_{j = m+1}^l (2N)^j b_{j, \lambda}$ is a multiple of $2^{m+1}N^{m+1}$. So $a_\lambda - b_\lambda$ cannot possibly vanish. Hence, $X_{G_1}^H \neq X_{G_2}^H$.
\end{proof}

As remarked before, when $\mathcal{G}, \mathcal{H}$ are both finite, $\mathcal{G}$ is uniformly bounded with respect to $\mathcal{H}$. Hence, we have the following corollary.

\begin{cor}\label{cor.connecteddistinguisher}
Any finite set of connected graphs $\mathcal{G}$ is uniformly distinguishable. Equivalently, for any $k \in \N$, the set $\{G \mid \abs{V(G)} \leq k, \textit{ $G$ is connected}\}$ is uniformly distinguishable.
\end{cor}

\begin{proof}
The second claim follows from the first one and the fact that there are finitely many (connected) graphs with no more than $k$-vertices up to isomorphism. To show the first claim, let $\mathcal{G}$ be a finite set of connected graphs. Then, there exists a finite distinguisher of $\mathcal{G}$ by Corollary \ref{cor.autcor}. Hence, $\mathcal{G}$ is uniformly distinguishable by Proposition \ref{prop.connecteddistinguisher} and the remark before this corollary.
\end{proof}

In the previous discussions, we always assume that $\mathcal{G}$ is a set of connected graphs. Of course, a result that generalizes to arbitrary graphs will be more desirable. It turns out that the assumption that $\mathcal{G}$ contains only connected graphs in Corollary \ref{cor.connecteddistinguisher} is unnecessary. To remove this assumption, we will need to use a somewhat different strategy. In particular, we will be looking at the coefficient of $m_{(1^k)}$. Hence, instead of Proposition \ref{prop.disjointH}, the key tool here is Corollary \ref{cor.autcor}.

\begin{prop}\label{prop.generaldistinguisher}
Any finite set of (not necessarily connected) graphs $\mathcal{G}$ is uniformly distinguishable.
\end{prop}

\begin{proof}
 Given a set of graphs $\mathcal{G}$, and let $\mathcal{H} = \{H_1, H_2, ..., H_l\}$ be the set of all connected components of graphs in $\mathcal{G}$ with isomorphic copies being identified. Without loss of generality, assume the graphs in $\mathcal{H}$ are indexed with respect to the following two rules:
\begin{enumerate}
\item If $i < j$, then $\abs{V(H_i)} \geq \abs{V(H_j)}$;
\item If $i < j$ and $\abs{V(H_i)} = \abs{V(H_j)}$, then $\abs{E(H_i)} \geq \abs{E(H_j)}$.
\end{enumerate}
Therefore, by Corollary \ref{cor.autcor}, a graph in $\mathcal{H}$ is not colorable by any other graph that comes after it in $\mathcal{H}$.

Now we construct an $H$ we claim is a distinguisher of  $\mathcal{G}$.
Let $C_1$ be the maximum number of connected components of a graph in $\mathcal{G}$, $C_2 := \max \{\abs{V(H_j)}\}_{j=1}^l$, and $C_3$ be the number of non-negative integer-valued vectors of length $l$ whose sum of the entries are no larger than $C_1$. Let $C_4$ be the maximum number of proper $H$-colorings of $G$ of type $(1^{\abs{V(G)}})$ one can obtain if $G \in \mathcal{G}$ and $H$ is a graph with at most $C_1C_2$ vertices. Clearly, $C_j$ ($1 \leq j \leq 4$) are all positive integers. Define $\{M_j\}_{j=1}^l \subset \N$ recursively as follows.
\begin{enumerate}
\item $M_1 = C_3C_4 + C_1 + 1$;
\item If $j \geq 2$, then $M_j = C_3C_4(\prod_{i = 1}^{j - 1} M_i)^{C_1} + C_1$.
\end{enumerate}
Hence, we have $C_1 = M_1 \leq M_2 \leq ... \leq M_l$. Now, we set $$H := \biguplus_{j=1}^l M_j H_j,$$ and we claim that $H$ is a distinguisher for $\mathcal{G}$.

Suppose $G_1, G_2 \in \mathcal{G}$. Since $H$ contains an isomorphic copy of each graph in $\mathcal{G}$, every graph in $\mathcal{G}$ is necessarily $H$-colorable. Thus, it is safe to assume that $k:=\abs{V(G_1)} = \abs{V(G_2)} = k$. If we write $X_{G_1}^H = \sum_{\lambda \vdash k} a_\lambda m_\lambda^{\abs{V(H)}}$ and $X_{G_2}^H = \sum_{\lambda \vdash k} b_\lambda m_\lambda^{\abs{V(H)}}$, then it suffices to show that $a_{(1^k)} \neq b_{(1^k)}$. We note that $a := a_{(1^k)}$ is the number of proper $H$-colorings of $G_1$ of type $(1^k)$. Suppose $G_1$ has $m$ components $H_{j_1}, H_{j_2}, ..., H_{j_m}$, where $j_1 \leq j_2 \leq ... \leq j_m$. Then, $G_1$ can be associated naturally with a vector of length $l$, $p(G_1) := (p_1, p_2, ..., p_l)$, where $p_j$ is the number of isomorphic copies of $H_j$ in $G_1$. Now, we can obtain a proper $H$-coloring of type $(1^k)$ if we choose $p_j$ copies of $H_j$ from $H$ to color \textit{the} $p_j$ copies of $H_j$ in $G_1$. This is always possible since $M_j \geq C_1$ for every $j$, and it will give us $\prod_{j=1}^l (\alpha_j)^{p_j} P^{M_j}_{p_j}$ colorings of type $(1^k)$, where $\alpha_j$ is the number of automorphisms of $H_j$ and $P$ is the standard notation for the number of permutations: $$P_{n_1}^{n_2} := \frac{n_2!}{(n_2-n_1)!}.$$ We can similarly define $p(G_2) = (q_1, q_2, ..., q_l)$ and obtain $\prod_{j=1}^l (\alpha_j)^{q_j} P^{M_j}_{q_j}$ such colorings of type $(1^k)$. We shall now study the two vectors $p(G_1)$ and $p(G_2)$.

Given two vectors $s = (s_1, s_2, ..., s_l), t = (t_1, t_2, ..., t_l)$, we say that $s$ is reverse dominated by $t$, and write $s <_R t$, if and only if $s \neq t$ and if $N$ is the largest number for which $s_N \neq t_N$, we have $s_N < t_N$. We shall also allow $>_R, \leq_R,$ and $ \geq_R$ to be defined naturally. Suppose $s <_R t$ and $\sum s_j, \sum t_j \leq C_1$. Let $N$ be the largest number for which $s_N < t_N$. We have $s_N < t_N \leq C_1$. Therefore,
\begin{align*}
\frac{P_{t_N}^{M_N}}{P_{s_N}^{M_N}} &= \frac{M_N!}{(M_N - t_N)!} \frac{(M_N - s_N)!}{M_N!}\\
&\geq M_N - s_N\\
&> M_N - C_1\\
&\geq C_3C_4(\prod_{i = 1}^{N - 1} M_i)^{C_1}.
\end{align*}
Hence, we have
\begin{align*}
\frac{\prod_{j=1}^l P_{t_j}^{M_j}}{\prod_{j=1}^l P_{s_j}^{M_j}} &= \frac{P_{t_N}^{M_N}}{P_{s_N}^{M_N}} \; \frac{\prod_{j=1}^{N-1} P_{t_j}^{M_j}}{\prod_{j=1}^{N-1} P_{s_j}^{M_j}} \\
&> C_3C_4(\prod_{i = 1}^{N - 1} M_i)^{C_1} \frac{\prod_{j=1}^{N-1} P_{t_j}^{M_j}}{\prod_{j=1}^{N-1} P_{s_j}^{M_j}} \\
&\geq \frac{C_3C_4(\prod_{i = 1}^{N - 1} M_i)^{C_1}}{\prod_{j=1}^{N-1} P_{s_j}^{M_j}} \\
&= C_3C_4 \prod_{j = 1}^{N - 1} \frac{M_j^{C_1}}{P_{s_j}^{M_j}}\\
&\geq C_3C_4.
\end{align*}
In particular, $\prod_{j=1}^l P_{t_j}^{M_j} > C_3C_4 \prod_{j=1}^l P_{s_j}^{M_j} \geq \prod_{j=1}^l P_{s_j}^{M_j}$.

Clearly, $p(G_1) \neq p(G_2)$, and without loss of generality, we assume $p(G_1) <_R p(G_2)$. We shall prove that $a_{(1^k)} < \prod_{j=1}^l P^{M_j}_{q_j} \leq \prod_{j=1}^l \alpha_j^{q_j} P^{M_j}_{q_j} \leq b_{(1^k)}$. The only non-trivial inequality is the first one. As a first step we will show the following:
\begin{align*}
a_{(1^k)} &= \sum_{v = (v_1, v_2, ..., v_l)} \beta_v \prod_{j = 1}^l {{M_j} \choose {v_j}}\\
&\leq \sum_{v = (v_1, v_2, ..., v_l)} \beta_v \prod_{j = 1}^l P_{v_j}^{M_j},
\end{align*}
where $v$ ranges over all non-negative integer-valued vectors whose sum is no more than $m$ and $\beta_v$ is the total number of $\biguplus_{j=1}^l v_j H_j$-colorings of type $(1^k)$ of $G_1$, and that all components in $H_v := \biguplus_{j=1}^l v_j H_j$ are used in the coloring. To see this, we write $$\sum_{v = (v_1, v_2, ..., v_l)} \beta_v \prod_{j = 1}^l {{M_j} \choose {v_j}} = \sum_{s = 1}^m \sum_{\substack{\abs{v} = s \\ v = (v_1, v_2, ..., v_l)}} \beta_v \prod_{j = 1}^l {{M_j} \choose {v_j}},$$ where $\abs{v} = \sum_{j=1}^l v_j$. On the right-hand side, the first sum controls the number of components in $H$ that we use to color $G$, which shall never be larger than $m$. The second sum controls which (up to isomorphism) component and how many of it we choose. There are $\prod_{j = 1}^l {{M_j} \choose {v_j}}$ ways we can obtain these components from $H$, and $\beta_v$ gives the number of colorings of type $(1^k)$ using these components.

As we have noted before, in order to have a proper $H_j$-coloring of type $(1^{\abs{V(H_{j_i})}})$ of $H_{j_i}$, we must have that $j \leq j_i$. Hence, suppose $v = (v_1, v_2, ..., v_l)$ $>_R p(G_1)$ and let $N$ be the largest number for which $v_N > p_N$. The $v_l$ components of $H_l$ in $H_v$ can only be used to color the $p_l$ components of $H_l$ in $G_1$. It then follows from induction that for any $j > N$, the $v_j$ components of $H_j$ in $H_v$ can only be used to color the $p_j$ components of $H_j$ in $G_1$. Now, at least one component of $H_N$ in $H_v$ cannot be used to color any component in $G_1$. Consequently, $\beta_v = 0$. Hence,
\begin{align*}
a_{(1^k)} &\leq  \sum_{v = (v_1, v_2, ..., v_l)} \beta_v \prod_{j = 1}^l P_{v_j}^{M_j}\\
&\leq C_4 \sum_{\substack{v = (v_1, v_2, ..., v_l) \\ v \; \leq_R \; p(G_1)}} \prod_{j = 1}^l P_{v_j}^{M_j}\\
&\leq C_4 C_3 \prod_{j = 1}^l P_{p_j}^{M_j}\\
&< \prod_{j = 1}^l P_{q_j}^{M_j}.
\end{align*}

Hence, $X_{G_1}^H \neq X_{G_2}^H$.
\end{proof}

Before we look into the next topic, we remark that, in general, one shall not expect that an infinite set of graphs is uniformly distinguishable, even if we assume that $\mathcal{G}$ only contains connected graphs. Indeed, if the chromatic numbers of graphs in $\mathcal{G}$ are not bounded from above, then for any $H$, there must be infinitely many graphs in $\mathcal{G}$ that are not $H$-colorable. This argument leads to the following result.

\begin{prop}\label{prop.notdistinguishable}
Let $\mathcal{G}$ be a set of graphs, and suppose $\{\chi(G) \mid G \in \mathcal{G}\}$ is not bounded. Then, $\mathcal{G}$ is not finitely distinguishable.
\end{prop}

\section{$K_{k_{1},k_{2}}$-chromatic symmetric functions}
\label{completebipartite}

It is clear that the benefit of discussing $H$-chromatic symmetric functions over ordinary chromatic symmetric functions is our ability to vary $H$. On the other hand, we have seen in Section \ref{basics} that, for a general $H$, the $H$-chromatic symmetric functions do not inherit many nice properties from the ordinary chromatic symmetric functions. Nevertheless, many interesting results can be found by considering specific choices of $H$. By doing so, we can actually find explicit expressions for the coefficients in the monomial basis for the $H$-chromatic symmetric functions---something we cannot necessarily do for an arbitrary $H$. As our focus is on uniqueness results, we will choose graphs for which we have interesting uniqueness results. One such graph is the complete bipartite graph. We begin by first constructing an explicit formula for $K_{m,n}$-chromatic symmetric functions. 

\begin{prop}
Let $H = K_{h_1,h_2}$ for some $h_1,h_2 \geq 1$. Then a connected graph $G$ is $H$-colorable if and only if $G$ is bipartite. 
\end{prop}

\begin{proof}
Let $G$ be connected and bipartite. In the case that $G$ has only one vertex, the result is trivial. Suppose $G$ has at least two vertices. Then $G$ can be colored by any $H$ as long as $H$ is not edgeless. On the other hand, suppose $G$ is not bipartite, then $G$ has an odd cycle. Let $H$ be a complete bipartite graph and let $U$ and $ V $ be the two independent sets of $H$. If we can use $H$ to color $G$, then at least two consecutive vertices in this odd cycle must be colored both by some $u\in U$ or both by some $v\in V$, which is a contradiction since $U$ and  $V$ are independent sets in $H$.
\end{proof}

We need some machinery for the next result.  Let $G$ be a graph with connected components $G_1,\dots G_l$, and label the maximally independent sets of each component $V_j^i$, where $1\leq i\leq l$ and $j \in \{1,2\}$. 
We adopt the following notation:
\begin{itemize}
\setlength\itemsep{0em}
    \item $p = (p_1,\dots p_l)$, $p \in \{1,2\}^l$,
    \item $g^i_j=|V^i_j|$, $k_1 = \sum_i g_{p_i}^i$ and $k_2 = \sum_i g_{3-p_i}^i = |G| -k_1$,
    \item $\lambda = (\lambda_1,\dots,\lambda_j) \vdash k_1$ and $\mu = (\mu_1,\dots,\mu_d) \vdash k_2$,
    \item $r_i = r_i(\lambda)$ and $s_i = r_i(\mu)$.
    \item $\lambda+\mu$ is the partition of $|G|$ formed by appending $\mu$ to the end of $\lambda$,
    \item and $t_i = r_i(\lambda+\mu)$.
\end{itemize}

We will also need to recall that in each of the multinomial coefficients, say ${{k} \choose {r_1(\lambda),\dots r_{j}(\lambda)}}$, the term $k-r_1(\lambda)-...-r_j(\lambda)$ is implicitly included if nonzero (so we mean ${{k} \choose {r_1(\lambda),\dots r_{j}(\lambda),k-r_1(\lambda)-...-r_j(\lambda)}}$).

\begin{prop}\label{prop.kmnCSFgen}

Let $G$ be a graph with connected components $G_1,\dots G_l$. Let $H = K_{h_1,h_2}$ for some $h_1,h_2 \geq 1$. Then $G$ is $H$-colorable if and only if each component $G_i$ is bipartite.
Furthermore, if $G$ is $H$-colorable, then using the notation defined above, the $H$-chromatic symmetric function of $G$ is given by
$$X_G^H = \sum\limits_p \sum\limits_{\lambda,\mu} a_{\lambda,\mu}^p m_{\lambda+\mu},$$

where 

\begin{align*}a_{\lambda,\mu}^p = \frac{\begin{pmatrix} h_1 \\ r_1,\dots, r_{k_1}\end{pmatrix}\begin{pmatrix} k_1 \\ \lambda_1,\dots, \lambda_{j}\end{pmatrix}\begin{pmatrix} h_2 \\s_1,\dots, s_{k_2}\end{pmatrix}\begin{pmatrix} k_2 \\ \mu_1,\dots, \mu_{d}\end{pmatrix} (h_1+h_2)!}{\begin{pmatrix} h_1+h_2 \\ t_1,\dots, t_{k_1+k_2}\end{pmatrix}},\\\end{align*}

\end{prop}

\begin{proof}
   Give $H$ an arbitrary labeling $\phi$. By Lemma \ref{lem.singlerep}, it suffices to count the number of $(H,\phi)$-colorings of $G$. Let $p \in \{1,2\}^l$. Then we use $p$ to make the following assignment of colors to independent sets in $G$: $V_{p_i}^i$ is assigned to $V_1^H$ and $V_{3 -p_i}^i$ is assigned to $V_2^H$. From the $h_1$ colors that we have to color $k_1$ vertices partitioned by $\lambda$ there are ${{h_1}\choose {r_1,\dots , r_{k_1}}}$ ways of choosing a color for $\lambda^i$ vertices, and for each choice of colors, there are ${{k_1}\choose {\lambda_1,\dots,\lambda_j}}$ ways to color these vertices. To color the remaining $k_2$ vertices, we now have $h_2$ colors to pick from. Hence, there are ${{h_2}\choose {s_1,\dots,s_{k_2}}}$ ways of choosing a color or the $\mu_i$ vertices, and for each choice of colors there are ${{k_2}\choose {\mu_1,\dots,\mu_d}}$ ways to color the vertices. So the number of $(H,\phi)$-colorings of $G$ is given by
   
   $$\alpha_{\lambda,\mu}^p = 
   {{h_1}\choose {r_1,\dots, r_{k_1}}}
   {{k_1}\choose {\lambda_1,\dots,\lambda_j}}
   {{h_2}\choose {s_1,\dots,s_{k_2}}}
   {{k_2}\choose {\mu_1,\dots,\mu_d}}.
   $$
   
   Now by Lemma \ref{lem.singlerep}, the coefficient of $m_{\lambda+\mu}$ is given by
   
   $$a_{\lambda,\mu}^p = \frac{\alpha_{\lambda,\mu}^p (h_1+h_2)!}{{{h_1+h_2}\choose {t_1,\dots,t_{k_1+k_2}}}}.$$
   
   Hence, summing over all possible $\lambda, \mu$ and $p$, the $H$-chromatic symmetric function of $G$ is given by
   $$X_G^H = \sum\limits_p \sum\limits_{\lambda,\mu} a_{\lambda,\mu}^p m_{\lambda+\mu}.$$
\end{proof}

The following corollary is immediate by identifying $S_{n+1} = K_{1, n}$.

\begin{cor}\label{cor.starCSFgen}
Let $G$ be a bipartite graph and $G_1, G_2, ..., G_l$ be its connected components. Suppose $\abs{V(G)} = k$. Let $H = S_{n+1}$ be a star with $n \geq 1$. Let $V_1^i$ and $V_2^i$ be the two independent sets of $V(G_i)$. (In the case where $G_i$ contains only 1 vertex, we set $V_1^i$ to be the set of that vertex and $V_2^j = \emptyset$.) For each partition $\lambda \vdash k$, and each $p = (p_1, p_2, ..., p_l) \in \{1,2\}^l$, define 
\[
  a_\lambda^p =
  \begin{cases}
                                   0 & \text{, if $\sum_j g_{p_j}^j \notin \lambda$ or $l(\lambda) > n+1$} \\
                                   {n \choose {r'_1, r'_2, ..., r'_k}} {{k - \sum_j g_{p_j}^j} \choose {\lambda'_1, \lambda'_2, ..., \lambda'_{j-1}}}(n+1)! / {{n+1} \choose {r_1, r_2, ..., r_k}} & \text{, otherwise}
  \end{cases}
,
\]
where $\lambda'$ is the partition obtained from $\lambda$ by removing one part that is equal to $\sum_j g_{p_j}^j$, $r_j = r_j(\lambda)$, and $r'_j = r_j(\lambda')$. Then, we have $$X_G^H = \sum_\lambda \sum_p a_\lambda^p m_\lambda,$$ where $\lambda$ ranges over all possible partitions of $k$, and $p$ ranges over $\{1,2\}^l$.
\end{cor}

It is a good idea to make sure that these $H$-chromatic symmetric functions are indeed new---that they are not just ordinary chromatic symmetric functions. The following corollary demonstrates this:

\begin{cor}
Let $G$ be a bipartite graph (not necessarily connected) as defined in Proposition \ref{prop.kmnCSFgen}, and with no star components. Let $H = K_{m,n}$, with $m$ fixed and $m < \min\{g_j^i \ | \ 1\leq i \leq l, \ j=1 \text{ or }2\}$. Then for every $n\geq 1$, $X_G^H$ cannot be written as a chromatic symmetric function of any graph $G'$.
\end{cor}

\begin{proof}
   Suppose $G$ and $H$ are given as above. Then with the restriction that $m < \min\{g_j^i \ | \ 1\leq i \leq l, \ j=1 \text{ or }2\}$, the coefficient of $m_{(1^k)}$ in $X_G^H$ is $0$ for any $n\geq 1$ by Proposition \ref{prop.kmnCSFgen}. However, the coefficient of $m_{(1^k)}$ in the chromatic symmetric function of a graph on $k$ vertices is positive.
\end{proof}

Although we are able to use a counting argument to derive Proposition \ref{prop.kmnCSFgen}, a simple observation will allow us to find a much neater expression, as well as some interesting uniqueness results. 

\begin{cor}\label{cor.kmnCSFinkmn}
Let $G$ be a bipartite graph, and $H = K_{h_1,h_2}$ be a complete bipartite graph. Then, using the same notations as in Proposition \ref{prop.kmnCSFgen},

$$X_G^H = \frac{1}{2}\sum\limits_p X_{K_{k_1,k_2}}^H.$$

\end{cor}

\begin{proof}
   Define $\overline{p} = (3-p_1, 3-p_2, ..., 3-p_l)$. Observe that $\sum\limits_{\lambda,\mu} (a_{\lambda,\mu}^p+a_{\lambda,\mu}^{\overline{p}}) m_{\lambda+\mu}$ is simply just $X_{K_{k_1,k_2}}^H$. Summing over all possible $p$,
   \begin{align*}
       \sum\limits_p \sum\limits_{\lambda,\mu} (a_{\lambda,\mu}^p+a_{\lambda,\mu}^{\overline{p}} ) m_{\lambda+\mu} &= \sum\limits_p X_{K_{k_1,k_2}}^H
   \end{align*}
   The left hand side is equivalent to $2X_G^H$, and so the claim follows. 
\end{proof}

Next, we will discuss following general problem in the setting of $H = K_{m, n}$.

\begin{prob}\label{prob.equiv}
Given a fixed graph $H$, and two non-isomorphic graphs $G_1, G_2$. Find conditions on $G_1, G_2$ that tell whether $X_{G_1}^H = X_{G_2}^H$.
\end{prob}

In light of Definition \ref{def.generalequiv}, Problem \ref{prob.equiv} is simply asking whether $G_1$ and $G_2$ are $H$-chromatically equivalent.
Of course, one way to provide the answer is to simply compute the two chromatic symmetric functions. The computation, however, is $\mathcal{NP}$-hard in general \cite{Hell1990complexity}. Therefore, any polynomial algorithms that answer this question are advantageous.

Corollary \ref{cor.kmnCSFinkmn} hints towards a close relationship between $G$s based on their $K_{m,n}$-chromatic symmetric functions. Suppose we define an equivalence relation $\sim_{m, n}$ on all bipartite graphs as follows: $$G_1 \sim_{m, n} G_2 \Leftrightarrow X_{G_1}^{K_{m, n}} = X_{G_2}^{K_{m, n}}.$$ Then, it is natural to ask whether $\sim_{m_1, n_1} = \sim_{m_2, n_2}$ for each $m_1, m_2, n_1, n_2 \in \N$. With these two questions in mind, we will first consider the more intuitive case of connected $G$, and then progress to include general $G$.

\begin{prop}\label{prop.equivconnected}
Let $H = K_{m,n}$ be a complete bipartite graph, and $G_1, G_2$ be two connected bipartite graphs on $k$ vertices. Then, $G_1$ is $H$-chromatically equivalent to $G_2$ if and only if the sizes of the two parts of $G_1$ are the same as the sizes of the two parts of $G_2$. Moreover, the symmetric functions $\{ X_{K_{j,k-j}}^H \ | \ 1\leq j \leq \floor{\frac{k}{2}}\}$ are linearly independent as elements in $\Lambda^k$. 
\end{prop}

\begin{proof}
   Observe that for any $G \in \mathcal{G}$, $G$ is $H$-chromatically equivalent to $K_{j,k-j}$ if and only if $X_G^H$ contains the term $m_{(j,k-j)}$.
\end{proof}

Our intuition would lead us to believe that in the case of general graphs, $G_1$ and $G_2$ are $K_{m,n}$-equivalent if they have the same number of vertices, same number of components, and same vertex partitions on said components. This seems obvious from the proof of Proposition \ref{prop.kmnCSFgen}. However, the latter condition is sufficient, but not necessary. For a more careful study, we will introduce the following definitions.

\begin{dfn}\label{dfn.Sanddiffbipartite}
Let $G$ be a bipartite graph. Following the notation of Proposition \ref{prop.kmnCSFgen}, we define $$S(G) := \{\{\sum_i g_{p_i}^i \ | \ p \in \{1,2\}^l\}\}$$ and $$\text{diff}(G) := (\abs{g_1^1 - g_2^1}, \abs{g_1^2 - g_2^2}, ..., \abs{g_1^l - g_2^l}).$$
\end{dfn}

\begin{prop}\label{prop.kmncsfequivgenG}
    Let $G_1, G_2$ be (not necessarily connected) bipartite graphs on $k$ vertices, and $H = K_{m, n}$ be a complete bipartite graph. Then $G_1$ is $H$-chromatically equivalent to $G_2$ if and only if $S(G_1) = S(G_2)$.
\end{prop}

\begin{proof}
Combine Corollary \ref{cor.kmnCSFinkmn} and the linear independence result in Proposition \ref{prop.equivconnected}.
\end{proof}

Proposition \ref{prop.kmncsfequivgenG} answers our second question: $\sim_{m, n}$ is the same for all $m, n \in \N$. Let us pause and ask whether Proposition \ref{prop.kmncsfequivgenG} provides an efficient method for determining $K_{m,n}$ equivalence.  Not necessarily. 
Although the structure of $S(G)$ appears to be simpler than that of $X_G^{K_{m, n}}$, the size of $S(G)$ is still in $\OO(2^k)$, where $k$ is the number of vertices of $G$.
 The following algorithm does even better, giving a linear algorithm:

\begin{prop}\label{prop.diffvector}
Let $G_1, G_2$ be two bipartite graphs on $k$ vertices with the same number of connected components. Then, $G_1$ is $K_{m,n}$-chromatically equivalent to $G_2$ if and only if $\diff(G_1)$ is a permutation of $\diff(G_2)$.
\end{prop}

\begin{proof} In this proof, we denote the $g_i^j$'s associated with $G_1$ by $a$ and those associated with $G_2$ by $b$. Without loss of generality, we also assume that $a_1^j \leq a_2^j$ and $b_1^j \leq b_2^j$ for every $1 \leq j \leq l$, so that the absolute value signs in the definitions of $\diff(G_1) = (d_1, d_2, ..., d_l)$ and $\diff(G_2) = (f_1, f_2, ..., f_l)$ can be omitted. Since permutation of the difference vectors are allowed, we also assume that the connected components are labeled so that $\diff(G_1)$ and $\diff(G_2)$ are in (weakly) increasing order.

Next, we observe that $s_1 := \min S(G_1) = \sum a_1^j = \frac{k - \sum d_j}{2}$ and $s_2 := \min S(G_2) = \sum b_1^j = \frac{k - \sum f_j}{2}$. Moreover, $S(G_1) = \{\{s_1 + \sum_{d \in D} d \mid D \in \mathcal{P}(\diff(G_1))\}\}$ and $S(G_2) = \{\{s_2 + \sum_{f \in F} f \mid F \in \mathcal{P}(\diff(G_2))\}\}$. Here, $\mathcal{P}$ is the power set if we identify $\text{diff}(G)$ with a multiset naturally. If $d_j = f_j$ for every $1 \leq j \leq l$, then $s_1 = s_2$ and we clearly have $S(G_1) = S(G_2)$.

Conversely, assume $\diff(G_1) \neq \diff(G_2)$. If $\sum d_j \neq \sum f_j$, then $s_1 \neq s_2$ so $S(G_1) \neq S(G_2)$. So assume otherwise. Let $j$ be the smallest number such that $d_j \neq f_j$. We have $\{\{s_1 + \sum_{d \in D} d \mid D \in \mathcal{P}((d_1, d_2, ..., d_{j-1}))\}\} = \{\{s_2 + \sum_{f \in F} f \mid F \in \mathcal{P}((f_1, f_2, ..., f_{j-1}))\}\}$, where the quantity in the equality is denoted by $S'$. Now, it is clear that $\min (S(G_1) \setminus S') = s_1 + d_j \neq s_2 + f_j = \min (S(G_2) \setminus S')$. Hence, $S(G_1) \neq S(G_2)$ and the claim is proved.
\end{proof}

The beauty of Proposition \ref{prop.diffvector} is that the size of $\diff(G)$ is in $\OO(k)$, and is fairly easy to compute. In fact, the computation is not just in polynomial time, but also in linear time.

\begin{cor}\label{cor.lineartimeKmn}
Let $G_1, G_2$ be two graphs on $N$ vertices and $M_1, M_2$ edges, respectively. Then, for any complete bipartite graph $H = K_{m, n}$, it takes $\OO(N+M_1+M_2)$ time to decide if $G_1, G_2$ are $H$-equivalent.
\end{cor}

\begin{proof}
Use breadth first search or depth first search to determine if $G_1, G_2$ are both bipartite graphs, and to compute their difference vectors. It suffices to note that the entries in the difference vectors are integers in the interval $[1, k-1]$, so we may bucket sort the two difference vectors.
\end{proof}

So far our discussions have centered around a fixed $H$ and varied $G$. There are no reasons to not ask if we can find any results for a fixed $G$ and varied $H$. The problem of determining if $X_G^{H_1}$ and $X_G^{H_2}$ may not appear as interesting in the setting of $H = K_{m, n}$. Indeed, when $G$ is a non-trivial bipartite and $H_1, H_2$ have distinct number of edges, then we always have $X_G^{H_1} \neq X_G^{H_2}$. A more reasonable question to ask here concerns $X_G^H$ as vectors in the linear space $\Lambda^k$. We will first look at special case of $K_{m, n}$---namely, the star graphs $S_{n+1}$. Then, we shall briefly discuss the difficulties of dealing with general complete bipartite graphs. The following lemma is obvious from Proposition \ref{prop.kmnCSFgen}.

\begin{lem}\label{lem.stardistinctH}
Let $G$ be a connected bipartite graph and let $k_1, k_2$ be the sizes of its two maximal independent sets. Suppose $n > \max \{k_1, k_2\}$. Then, $$X_G^{S_{n+1}} = nX_G^{S_n}.$$ Moreover, when $n \leq \max\{k_1, k_2\}$, $X_G^{S_{n+1}} - n X_G^{S_n}$ contains monomials whose indexing partitions have length $n+1$ only.
\end{lem}

\begin{prop}\label{prop.KmnfixedH}
Let $G$ be a bipartite graph on $k$ vertices. Then, $\spann\{X_G^{S_{n+1}} \mid n \geq 1\} = \spann\{X_G^{S_{n+1}} \mid 1 \leq n \leq \max S(G) \}$ is a subspace of $\Lambda^{k}$ of dimension $\max S(G)$. In particular, the set $\{X_G^{S_{n+1}} \mid 1 \leq n \leq \max S(G) \}$ is linearly independent.
\end{prop}

\begin{proof} By Lemma \ref{lem.stardistinctH} and Corollary \ref{cor.kmnCSFinkmn},
$\spann\{X_G^{S_{n+1}} \mid n \geq \max S(G) \} = \spann\{X_G^{S_{S(G) + 1}}\}$. To show $X_G^{S_2}, X_G^{S_3}, ..., X_G^{S_{\max S(G) + 1}}$ are linearly independent, we note that whenever $n \leq \max S(G)$, $X_G^{S_{n+1}}$ contains a monomial whose corresponding partition has length $n+1$, whereas this is clearly impossible for a smaller star.
\end{proof}

\begin{cor}\label{cor.dimensionKmn}
For a fixed $k \geq 2$, the subspace of $\Lambda^k$ spanned by $\{X_G^{S_{n+1}} \mid \abs{V(G)} = k, n \geq 1\}$ has dimension of at most $\frac{(\ceil{k/2} + k - 1)\floor{k/2}}{2}$.
\end{cor}

\begin{proof}
By Corollary \ref{cor.kmnCSFinkmn}, $\spann \{X_G^{S_{n+1}} \mid \abs{V(G)} = k, n \geq 1\} = \spann \{X_{K_{j, k-j}}^{S_{n+1}} \mid 1 \leq j \leq \floor{\frac{k}{2}}, n \geq 1\}$. For a fixed $j$, by Proposition \ref{prop.KmnfixedH}, $\dim (\spann \{X_{K_{j, k-j}}^{S_{n+1}} \mid n \geq 1\}) = k - j$. It suffices to note that $(k-1) + (k-2) + ... + (k - \floor{\frac{k}{2}}) = \frac{(\ceil{k/2} + k - 1)\floor{k/2}}{2}$.
\end{proof}

We know that the dimension of $\Lambda^k$ is asymptotically equivalent to $\frac{1}{4k\sqrt{3}} e^{\pi \sqrt{2k/3}}$ as $k \rightarrow \infty$ \cite{Erdos1970star}. Hence, the implication  of the fact that $\dim \spann \{X_G^{S_{n+1}} \mid \abs{V(G)} = k, n \geq 1\}$ grows at a rate no faster than quadratic
is that only very few symmetric functions can be written as a linear combination of star-chromatic symmetric functions.

Where does the argument from Lemma \ref{lem.stardistinctH} to Corollary \ref{cor.dimensionKmn} start to collapse for general complete bipartite graphs? Unfortunately, it is Lemma \ref{lem.stardistinctH} that fails in the general case. In fact, we have the following result which shows it is completely hopeless to obtain a similar result using the same logic.

\begin{prop}
Suppose $G$ is a connected bipartite graph, and let $g_1, g_2$ be the sizes of its two maximal independent sets. Assume $\min \{g_1, g_2\} > 2$. Then, $X_{G}^{K_{m_1,n}}$ is not a (rational) scalar multiple of $X_{G}^{K_{m_2,n}}$ whenever $m_1 \neq m_2$, $ 2 < m_1, m_2$.\\
\label{prop515}
\end{prop}
\begin{proof}
We will compare ratios of coefficients of different monomials. Given an $m > 2$, we first compute the coefficient of $m_{(g_1,g_2)}$ in $X_G^{K_{m,n}}$. There are $2mn$ proper $K_{m,n}$-colorings of $G$ that have type $(g_1,g_2)$. So by Lemma \ref{lem.singlerep}, the coefficient of $m_{(g_1,g_2)}$ in $X_G^{K_{m,n}}$ is $c_m = (2mn)\alpha(m+n-2)!$, where $\alpha = 1$ if $g_1 \neq g_2$ and $\alpha = 2$ otherwise. In particular,

$$\frac{c_{m+1}}{c_m} = \frac{(m+1)(m+n-1)}{m}.$$

Similarly, we can compute the coefficient of $m_{(g_1-1, g_2-1,1,1)}$ in $X_G^{K_{m, n}}$, which we denote by $d_m$. To count the number of proper colorings of this type, we need to first choose two colors from each independent set of $K_{m, n}$ to use, and then choose one vertex from each independent set of $G$ to be colored differently. Hence, the total number of proper $K_{m, n}$-colorings of $G$ that have type $(g_1-1, g_2-1, 1, 1)$ is $2m(m-1)n(n-1)g_1g_2$. By Lemma \ref{lem.singlerep} again, we have $d_m = (2m(m-1)n(n-1)g_1g_2) \beta (m+n-4)!$, where $\beta = 4!$, $3!$, or $2!$, depending on the values of $g_1$ and $g_2$. In particular, $$\frac{d_{m+1}}{d_m} = \frac{(m+1)(m+n-3)}{m-1}.$$

Now, $(m+n-1)(m-1) = m^2 + mn - 2m - n + 1 > m^2 + mn - 3m = (m+n-3)m$. So $c_{m+1}/c_m > d_{m+1}/d_m$, and the result follows by induction.
\end{proof}

A similar argument to above leads to the following proposition.

\begin{prop}
Suppose $G$ is a connected bipartite graph, and let $g_1, g_2$ be the sizes of its two maximal independent sets. Assume $\abs{g_1 - g_2} \neq 1$, $\min \{g_1, g_2\} \geq 2$, and $n \geq 2$. Then, $X_{G}^{K_{m_1,n}}$ is not a (rational) scalar multiple of $X_{G}^{K_{m_2,n}}$ whenever $m_1 \neq m_2$, $m_1, m_2 \geq n$.
\label{prop516}
\end{prop}

\section{$S_{k}^{1}$-chromatic symmetric functions}
\label{augmentedstars}

We have mentioned that the definition of $H$-chromatic symmetric functions does not require  $H$ to be simple. In this section, we work with a class of non-simple $H$. As in the previous question, the major question we are trying to answer is regarding the equivalence of different $G$'s. We first define the principal subject of our discussion, which we will call {\em  augmented stars}. We remark that the stars with loops have been studied in the literature of both mathematics and physics. One of the examples is the so-called Widom-Rowlinson model \cite{Widom1970star}.

\begin{dfn}\label{dfn.augmentedstars}
An augmented star on $(n+1)$ vertices, denoted by $S_{n+1}^1$, is the star $S_{n+1}$ with an extra loop on its center.
\end{dfn}

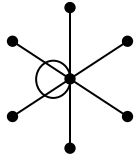
\begin{figure}[ht]
    \centering

\tikzset{every picture/.style={line width=0.75pt}} 

\begin{tikzpicture}[x=0.75pt,y=0.75pt,yscale=-1,xscale=1]

\draw  [fill={rgb, 255:red, 0; green, 0; blue, 0 }  ,fill opacity=1 ] (109.3,111.85) .. controls (109.3,110.55) and (110.35,109.5) .. (111.65,109.5) .. controls (112.95,109.5) and (114,110.55) .. (114,111.85) .. controls (114,113.15) and (112.95,114.2) .. (111.65,114.2) .. controls (110.35,114.2) and (109.3,113.15) .. (109.3,111.85) -- cycle ;
\draw  [fill={rgb, 255:red, 0; green, 0; blue, 0 }  ,fill opacity=1 ] (109.3,149.85) .. controls (109.3,148.55) and (110.35,147.5) .. (111.65,147.5) .. controls (112.95,147.5) and (114,148.55) .. (114,149.85) .. controls (114,151.15) and (112.95,152.2) .. (111.65,152.2) .. controls (110.35,152.2) and (109.3,151.15) .. (109.3,149.85) -- cycle ;
\draw  [fill={rgb, 255:red, 0; green, 0; blue, 0 }  ,fill opacity=1 ] (138.3,130.85) .. controls (138.3,129.55) and (139.35,128.5) .. (140.65,128.5) .. controls (141.95,128.5) and (143,129.55) .. (143,130.85) .. controls (143,132.15) and (141.95,133.2) .. (140.65,133.2) .. controls (139.35,133.2) and (138.3,132.15) .. (138.3,130.85) -- cycle ;
\draw  [fill={rgb, 255:red, 0; green, 0; blue, 0 }  ,fill opacity=1 ] (167.3,111.85) .. controls (167.3,110.55) and (168.35,109.5) .. (169.65,109.5) .. controls (170.95,109.5) and (172,110.55) .. (172,111.85) .. controls (172,113.15) and (170.95,114.2) .. (169.65,114.2) .. controls (168.35,114.2) and (167.3,113.15) .. (167.3,111.85) -- cycle ;
\draw  [fill={rgb, 255:red, 0; green, 0; blue, 0 }  ,fill opacity=1 ] (167.3,149.85) .. controls (167.3,148.55) and (168.35,147.5) .. (169.65,147.5) .. controls (170.95,147.5) and (172,148.55) .. (172,149.85) .. controls (172,151.15) and (170.95,152.2) .. (169.65,152.2) .. controls (168.35,152.2) and (167.3,151.15) .. (167.3,149.85) -- cycle ;
\draw    (140.65,130.85) -- (169.65,111.85) ;
\draw    (111.65,149.85) -- (140.65,130.85) ;
\draw    (140.65,130.85) -- (111.65,111.85) ;
\draw    (169.65,149.85) -- (140.65,130.85) ;
\draw  [fill={rgb, 255:red, 0; green, 0; blue, 0 }  ,fill opacity=1 ] (138.3,165.85) .. controls (138.3,164.55) and (139.35,163.5) .. (140.65,163.5) .. controls (141.95,163.5) and (143,164.55) .. (143,165.85) .. controls (143,167.15) and (141.95,168.2) .. (140.65,168.2) .. controls (139.35,168.2) and (138.3,167.15) .. (138.3,165.85) -- cycle ;
\draw  [fill={rgb, 255:red, 0; green, 0; blue, 0 }  ,fill opacity=1 ] (138.3,94.85) .. controls (138.3,93.55) and (139.35,92.5) .. (140.65,92.5) .. controls (141.95,92.5) and (143,93.55) .. (143,94.85) .. controls (143,96.15) and (141.95,97.2) .. (140.65,97.2) .. controls (139.35,97.2) and (138.3,96.15) .. (138.3,94.85) -- cycle ;
\draw    (140.65,94.85) -- (140.65,130.85) ;
\draw    (140.65,130.85) -- (140.65,165.85) ;
\draw  [draw opacity=0] (140.65,133.2) .. controls (139.76,137.37) and (136.33,140.47) .. (132.22,140.47) .. controls (127.45,140.47) and (123.57,136.26) .. (123.57,131.07) .. controls (123.57,125.88) and (127.45,121.67) .. (132.22,121.67) .. controls (136.44,121.67) and (139.95,124.95) .. (140.72,129.29) -- (132.22,131.07) -- cycle ; \draw   (140.65,133.2) .. controls (139.76,137.37) and (136.33,140.47) .. (132.22,140.47) .. controls (127.45,140.47) and (123.57,136.26) .. (123.57,131.07) .. controls (123.57,125.88) and (127.45,121.67) .. (132.22,121.67) .. controls (136.44,121.67) and (139.95,124.95) .. (140.72,129.29) ;

\end{tikzpicture}

    \caption{The augmented star $S_7^1$}
    \label{fig:my_label}
\end{figure}

We have seen that the only $K_{m,n}$-colorable graphs are the bipartite graphs. This is clearly not the case for augmented stars. Indeed, as long as $H$ contains a loop, then there is always a proper $H$-coloring of $G$ of the type $(k)$, where $k$ is the number of vertices in $G$. Nevertheless, some results in $S_{n+1}^1$-colorings can be derived from those in $S_{n+1}$-colorings. To do so, we will need to define a new class of symmetric functions that appear in Section \ref{completebipartite}, but whose formal definition has not been written down.

\begin{dfn}\label{dfn.partialstar}
Given three integers $k_1, k_2, n \geq 1$. For each $\lambda \vdash k_1+k_2$, we set 
\begin{equation}
\label{multi}
  a_\lambda =
  \begin{cases}
                                   0 & \text{, if $l(\lambda) > n+1$ or $k_1 \notin \lambda$} \\
                                   {n \choose {r'_1, r'_2, ..., r'_k}} {k_2 \choose {\lambda'_1, \lambda'_2, ..., \lambda'_{j-1}}} (n+1)! / {{n+1} \choose {r_1, r_2, ..., r_k}} & \text{, otherwise}
  \end{cases}
,
\end{equation}
where $\lambda' = (\lambda'_1, \lambda'_2, ..., \lambda'_{j-1}) \vdash k_2$ is obtained from $\lambda$ by deleting one part of $k_1$, and $r_i = r_i(\lambda)$, $r'_i = r_i(\lambda')$. We define $X_{k_1, k_2}^{S_{n+1}} = \sum_\lambda a_\lambda m_\lambda$.
\end{dfn}

So, for example, by Corollary \ref{cor.starCSFgen}, we have $X_{K_{k_1, k_2}}^{S_{n+1}} = X_{k_1, k_2}^{S_{n+1}} + X_{k_2, k_1}^{S_{n+1}}$. We emphasize that $k_1$ and $k_2$ are not symmetric in the definition above. Informally speaking, $X_{k_1, k_2}^{S_{n+1}}$ is the symmetric function obtained by forcing $k_1$ vertices to be colored by the label on a fixed vertex in $H$ and the remaining $k_2$ vertices to be colored arbitrarily by the labels on the other $n$ vertices in $H$.

Before we use $X_{k_1, k_2}^{S_{n+1}}$ to compute the $S_{n+1}^1$-chromatic symmetric functions, we introduce a lemma on the linear independence of these symmetric functions. As we have seen in Section \ref{completebipartite}, such results can be very useful in the discussion of $H$-equivalence relations.

\begin{lem}\label{lem.independenceXk1k2}
If $n \geq 2$, then $\{X_{1, k}^{S_{n+1}}, X_{2, k-1}^{S_{n+1}}, ..., X_{k, 1}^{S_{n+1}}\}$ are linearly independent in $\Lambda^{k+1}$.
\end{lem}

\begin{proof}
When $n \geq k$, the statement is obvious since $m_{(j, 1^{k-j})}$ exists in $X_{j, k-j}^{S_{n+1}}$ but vanishes in $X_{i, k-i}^{S_{n+1}}$ for every $i > j$. So assume $2 \leq n < k$. We set $$c_1 X_{1, k}^{S_{n+1}} + c_2 X_{2, k-1}^{S_{n+1}} + ... + c_k X_{k, 1}^{S_{n+1}} = 0.$$ We first observe that $m_{(j, k+1-j)}$ exists only in $X_{j, k+1-j}^{S_{n+1}}$ and $X_{k+1-j, j}^{S_{n+1}}$, and that its coefficients in the two symmetric functions are the same. Therefore, $c_j + c_{k+1-j} = 0$. When $k$ is even, we also have $c_{\frac{k}{2}} = 0$. If we set $X_j^{S_{n+1}} := X_{j, k+1-j}^{S_{n+1}} - X_{k+1-j, j}^{S_{n+1}}$ for $1 \leq j \leq \floor{\frac{k}{2}}$, then the equation can be rewritten as $$c_1 X_1^{S_{n+1}} + c_2 X_2^{S_{n+1}} + ... + c_{\floor{\frac{k}{2}}} X_{\floor{\frac{k}{2}}}^{S_{n+1}} = 0.$$

Next, we make some observations of $X_j^{S_{n+1}}$. First, for each $1 \leq j \leq \floor{\frac{k}{2}}$, the only monomial $m_\lambda$ exists in $X_{j, k+1-j}^{S_{n+1}}$ where $k+1-j$ is a part of $\lambda$ is $m_{(k+1-j, j)}$, whereas $X_{k+1-j, j}^{S_{n+1}}$ is a linear combination of monomials $m_\lambda$ such that $k+1-j$ is a part of $\lambda$. Therefore, in $X_j^{S_{n+1}}$, all monomials $m_\lambda$ with $k+1-j \in \lambda$ have non-positive coefficients, and all monomials $m_\lambda$ with $k+1-j \notin \lambda$ have non-negative coefficients. Given this observation, we shall note that for $2 \leq j \leq \floor{\frac{k}{2}}$, the monomial $m_{(k+1-j, j-1, 1)}$ appears in $X_1^{S_{n+1}}$ with a positive coefficient, in $X_{k+2-j}^{S_{n+1}}$ with a positive coefficient, in $X_{k+1-j}^{S_{n+1}}$ with a negative coefficient, and in nowhere else among $ X_1^{S_{n+1}}, X_2^{S_{n+1}}, ..., X_{\floor{\frac{k}{2}}}^{S_{n+1}}$.

Now, we claim that $c_1c_j \geq 0$ for every $2 \leq j \leq \floor{\frac{k}{2}}$. We prove it by induction. When $j = 2$, the result follows from the fact that $m_{(k-1, 1, 1)}$ appears in $X_1^{S_{n+1}}$ with a positive coefficient and in $X_2^{S_{n+1}}$ with a negative coefficient, and nowhere else, as noted above. Assume the claim holds for $j$. The monomials $m_{(k-j, j, 1)}$ appears in $X_1^{S_{n+1}}$ with a positive coefficient, in $X_j^{S_{n+1}}$ with a positive coefficient, in $X_{j+1}^{S_{n+1}}$ with a negative coefficient, and nowhere else. But $c_1c_j \geq 0$, so $c_1$ and $c_{j+1}$ cannot have different signs. The claim is proved.

Finally, consider the monomial $m_{(\floor{\frac{k+1}{2}}, \floor{\frac{k}{2}}, 1)}$. This monomial exists in $X_1^{S_{n+1}}$ since $n \geq 2$, and its coefficient is non-negative in all of $X_2^{S_{n+1}}, X_3^{S_{n+1}}, ...,$ $X_{\floor{\frac{k}{2}}}^{S_{n+1}}$, as noted at the beginning. Since $c_1c_j \geq 0$, we must have $c_1 = 0$. Now, by an induction argument, suppose $c_i = 0$ for every $i \leq j$, then $X_{j+1}^{S_{n+1}}$ is the only possibly non-vanishing term among $X_1^{S_{n+1}}, X_2^{S_{n+1}}, ..., X_{\floor{\frac{k}{2}}}^{S_{n+1}}$ that contains $m_{(k+1-j, j, 1)}$, so $c_{j+1} = 0$. Hence, $c_j = 0$ for every $1 \leq j \leq k$, and we are done.
\end{proof}

Now, we are ready to compute the $S_{n+1}^1$ chromatic symmetric functions and discuss its consequences. To do so, we will have to first introduce some notation. Given any graph $G$, denote by $\mathcal{I}(G)$ the set of all independent sets of $G$. Let $I(G)$ be the multiset $I(G) = \{\abs{S} \mid S \in \mathcal{I}(G)\}$. $\tilde{\mathcal{I}}(G)$ and $\tilde{I}(G)$ are defined similarly by requiring all independent sets to be maximal.

\begin{prop}\label{prop.augstarCSF}
Given a graph $G$ on $k$ vertices that is not edgeless, and a natural number $n \geq 1$, we have $$X_G^{S_{n+1}^1} = n! m_{(k)} + \sum_{k_1 \in I(G)} X_{k-k_1, k_1}^{S_{n+1}}.$$
\end{prop}

\begin{proof}
A proper $S_{n+1}^1$-coloring of $G$ that is not of the type $(k)$ is obtained by first picking an independent set $S$ of $G$, coloring all vertices in $S$ by the labels on the leaves, and coloring all vertices in $V(G) \setminus S$ by the label on the center of $S_{n+1}^1$. Once we find an independent set $S$, the monomials obtained by this manner sum to $X_{k-\abs{S}, \abs{S}}^{S_{n+1}}$. The values of $c(G)$ are obvious.
\end{proof}

Given Lemma \ref{lem.independenceXk1k2} and Proposition \ref{prop.augstarCSF}, the following result is immediate.

\begin{cor}\label{cor.equivalenceaugstar}
Suppose $G_1, G_2$ are two graphs on $k$ vertices, and $n \geq 2$. Then, $X_{G_1}^{S_{n+1}^1} = X_{G_2}^{S_{n+1}^1}$ if and only if $I(G_1) = I(G_2)$, if and only if $\tilde{I}(G_1) = \tilde{I}(G_2)$.
\end{cor}

\begin{proof}
Combine Lemma \ref{lem.independenceXk1k2} and Proposition \ref{prop.augstarCSF} to get the first equivalence. For the second one, it suffices to note that $N_j = \sum_{i \geq j} {i \choose j} \tilde{N}_j$, where $N_j$ is the number of independent sets of size $j$ and $\tilde{N}_i$ is the number of maximal independent sets  of size $i$.
\end{proof}

We remark that Lemma \ref{lem.independenceXk1k2} fails for $n = 1$. Indeed, $X_{j, k+1-j}^{S_2} = m_{(j, k+1-j)} = X_{k+1-j, j}^{S_2}$. Moreover, Corollary \ref{cor.equivalenceaugstar} does not hold for $n = 1$ either. The easiest counter-examples one can construct are $G_1$ and $G_2$ below in Figure \ref{fig:IG1equalsIG2ex}. Note that $I(G_1) = \{1, 1, 1, 1, 1, 2, 2, 2, 2, 2\}$, whereas $I(G_2) = \{1, 1, 1, 1, 1, 2, 2, 2, 2, 3\}$. However, we have $X_{G_1}^{S_2} = X_{G_2}^{S_2} = 5m_{(3,2)} + 5m_{(4,1)} + m_{(5)}$. Another way to see that $X_{G_1}^{S_2} = X_{G_2}^{S_2}$ without computing is by observing $X_{2,3}^{S_2} = X_{3,2}^{S_2}$.

\bigskip

\begin{figure}[ht] 
\centering
\begin{tikzpicture}[x=0.75pt,y=0.75pt,yscale=-1,xscale=1]

\draw  [fill={rgb, 255:red, 0; green, 0; blue, 0 }  ,fill opacity=1 ] (119,160.75) .. controls (119,159.78) and (119.78,159) .. (120.75,159) .. controls (121.72,159) and (122.5,159.78) .. (122.5,160.75) .. controls (122.5,161.72) and (121.72,162.5) .. (120.75,162.5) .. controls (119.78,162.5) and (119,161.72) .. (119,160.75) -- cycle ;
\draw  [fill={rgb, 255:red, 0; green, 0; blue, 0 }  ,fill opacity=1 ] (299,110.75) .. controls (299,109.78) and (299.78,109) .. (300.75,109) .. controls (301.72,109) and (302.5,109.78) .. (302.5,110.75) .. controls (302.5,111.72) and (301.72,112.5) .. (300.75,112.5) .. controls (299.78,112.5) and (299,111.72) .. (299,110.75) -- cycle ;
\draw  [fill={rgb, 255:red, 0; green, 0; blue, 0 }  ,fill opacity=1 ] (119,110.75) .. controls (119,109.78) and (119.78,109) .. (120.75,109) .. controls (121.72,109) and (122.5,109.78) .. (122.5,110.75) .. controls (122.5,111.72) and (121.72,112.5) .. (120.75,112.5) .. controls (119.78,112.5) and (119,111.72) .. (119,110.75) -- cycle ;
\draw  [fill={rgb, 255:red, 0; green, 0; blue, 0 }  ,fill opacity=1 ] (169,160.75) .. controls (169,159.78) and (169.78,159) .. (170.75,159) .. controls (171.72,159) and (172.5,159.78) .. (172.5,160.75) .. controls (172.5,161.72) and (171.72,162.5) .. (170.75,162.5) .. controls (169.78,162.5) and (169,161.72) .. (169,160.75) -- cycle ;
\draw  [fill={rgb, 255:red, 0; green, 0; blue, 0 }  ,fill opacity=1 ] (219,160.75) .. controls (219,159.78) and (219.78,159) .. (220.75,159) .. controls (221.72,159) and (222.5,159.78) .. (222.5,160.75) .. controls (222.5,161.72) and (221.72,162.5) .. (220.75,162.5) .. controls (219.78,162.5) and (219,161.72) .. (219,160.75) -- cycle ;
\draw  [fill={rgb, 255:red, 0; green, 0; blue, 0 }  ,fill opacity=1 ] (169,110.75) .. controls (169,109.78) and (169.78,109) .. (170.75,109) .. controls (171.72,109) and (172.5,109.78) .. (172.5,110.75) .. controls (172.5,111.72) and (171.72,112.5) .. (170.75,112.5) .. controls (169.78,112.5) and (169,111.72) .. (169,110.75) -- cycle ;
\draw  [fill={rgb, 255:red, 0; green, 0; blue, 0 }  ,fill opacity=1 ] (348,110.75) .. controls (348,109.78) and (348.78,109) .. (349.75,109) .. controls (350.72,109) and (351.5,109.78) .. (351.5,110.75) .. controls (351.5,111.72) and (350.72,112.5) .. (349.75,112.5) .. controls (348.78,112.5) and (348,111.72) .. (348,110.75) -- cycle ;
\draw  [fill={rgb, 255:red, 0; green, 0; blue, 0 }  ,fill opacity=1 ] (299,160.75) .. controls (299,159.78) and (299.78,159) .. (300.75,159) .. controls (301.72,159) and (302.5,159.78) .. (302.5,160.75) .. controls (302.5,161.72) and (301.72,162.5) .. (300.75,162.5) .. controls (299.78,162.5) and (299,161.72) .. (299,160.75) -- cycle ;
\draw  [fill={rgb, 255:red, 0; green, 0; blue, 0 }  ,fill opacity=1 ] (348,160.75) .. controls (348,159.78) and (348.78,159) .. (349.75,159) .. controls (350.72,159) and (351.5,159.78) .. (351.5,160.75) .. controls (351.5,161.72) and (350.72,162.5) .. (349.75,162.5) .. controls (348.78,162.5) and (348,161.72) .. (348,160.75) -- cycle ;
\draw  [fill={rgb, 255:red, 0; green, 0; blue, 0 }  ,fill opacity=1 ] (399,160.75) .. controls (399,159.78) and (399.78,159) .. (400.75,159) .. controls (401.72,159) and (402.5,159.78) .. (402.5,160.75) .. controls (402.5,161.72) and (401.72,162.5) .. (400.75,162.5) .. controls (399.78,162.5) and (399,161.72) .. (399,160.75) -- cycle ;
\draw    (119,110.75) -- (170.75,110.75) ;
\draw    (120.75,112.5) -- (120.75,162.5) ;
\draw    (120.75,160.75) -- (170.75,160.75) ;
\draw    (170.75,160.75) -- (222.5,160.75) ;
\draw    (170.75,110.75) -- (220.75,160.75) ;
\draw    (349.75,160.75) -- (302.5,160.75) ;
\draw    (300.75,112.5) -- (300.75,162.5) ;
\draw    (349.75,109) -- (349.75,162.5) ;
\draw    (400.75,160.75) -- (351.5,160.75) ;
\draw    (300.75,110.75) -- (349.75,160.75) ;
\draw    (349.75,110.75) -- (300.75,160.75) ;

\draw (161,173.4) node [anchor=north west][inner sep=0.75pt]    {$G_{1}$};
\draw (340,174.4) node [anchor=north west][inner sep=0.75pt]    {$G_{2}$};

\end{tikzpicture}
\caption{Example of $G_1$ and $G_2$ such that $I(G_1)=I(G_2).$}
\label{fig:IG1equalsIG2ex}
\end{figure}
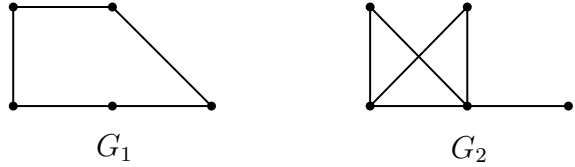
\bigskip

When we have two graphs $G_1$ and $G_2$ and would like to know if they have the same chromatic symmetric functions with respect to an augmented star $S^1_{n+1}$, is Corollary \ref{cor.equivalenceaugstar} an easier characterization (i.e. it suggests a polynomial algorithm to the problem) of this equivalence relation? It is not quite clear yet. It seems hard to show if the problem is $\mathcal{NP}$-hard: even though computing $I(G)$ is $\mathcal{NP}$-hard, it is not clear yet whether comparing $I(G_1)$ to $I(G_2)$ still is. Nevertheless, we shall see that Corollary \ref{cor.equivalenceaugstar} does lead to some necessary \textit{or} sufficient conditions of $X_{G_1}^{S_{n+1}^1} = X_{G_2}^{S_{n+1}^1}$ that will be hard to see directly from Proposition \ref{prop.augstarCSF}.

\begin{cor}\label{cor.augstarpartialcors}
Suppose $n \geq 2$. Then, 
\begin{enumerate}
\item If $G_1, G_2$ are two (simple) graphs on $k$ vertices with $X_{G_1}^{S_{n+1}^1} = X_{G_2}^{S_{n+1}^1}$, then $G_1, G_2$ have the same number of edges.
\item All (simple) graphs in $\mathcal{G}_m := \{\overline{G} \mid \abs{V(G)} = k, \abs{E(G)} = m \textit{, $G$ is $C_3$-free}\}$ have the same $S_{n+1}^1$-chromatic symmetric function, which is not equal to the $S_{n+1}^1$-chromatic symmetric function of any graph not in $\mathcal{G}_m$.
\item For each $1 \leq j \leq k$, all graphs in $\mathcal{T}_j := \{\overline{T} \mid T$  is a forest on $k$ vertices with $j$ connected components $\}$ have the same $S_{n+1}^1$-chromatic symmetric function. In particular, all tree-complements (on the same number of vertices) have the same $S_{n+1}^1$-chromatic symmetric function.
\end{enumerate}
\end{cor}

\begin{proof}
Suppose $X_{G_1}^{S_{n+1}^1} = X_{G_2}^{S_{n+1}^1}$, then $G_1$ and $G_2$ have the same number of non-edges (i.e. independent sets of size 2), and therefore the same number of edges. If $G_1, G_2$ are  $C_3$-free, then $G_1, G_2$ contain no cliques of size larger than $2$, so $\mathcal{I}(\overline{G_1}) = \mathcal{I}(\overline{G_2})$ if and only if $G_1, G_2$ have the same number of edges. If $\overline{G} \notin \mathcal{G}_m$, then either $\overline{G}$ has different number of edges, or $G$ contains a $C_3$ subgraph. Either way, $G$ cannot have the same $S_{n+1}^1$-chromatic symmetric function as those of graphs in $\mathcal{G}_m$. 3 is a special case of 2.
\end{proof}

\section{Classical bases of $\Lambda$ as $H$-chromatic symmetric functions of complete multipartite graphs}
\label{completebasis}

In this section, we will show that classical bases of the space of symmetric functions, that is, the monomial symmetric functions, power sum symmetric functions, and elementary symmetric functions, can be realized as $H$-chromatic symmetric functions.

As promised in Section \ref{basics}, we show a stronger result of Proposition \ref{prop.monomial}.
We denote by $K_\lambda$ the complete $l(\lambda)$-partite graph whose parts contain $\lambda_1, \lambda_2, ..., \lambda_l$ vertices, respectively. We denote by $K_n^-$ the graph obtained by deleting an edge from the complete graph on $n$ vertices.

\begin{prop}\label{prop.mulpartmonomials}
\textit{Let $\lambda \vdash k$, $n = l(\lambda)$, and $H$ be a graph that does not contain $K_{n+1}^-$ as a subgraph. Let $m$ be the number of subgraphs of $H$ that are isomorphic to $K_n$. Then, $$X_{K_\lambda}^H = m \cdot n! m_\lambda^{\abs{V(H)}} = m \cdot n! (\prod_{j = 1}^{k} r_j(\lambda)!) (\abs{V(H)} - \sum_{j=1}^k r_j(\lambda))! m_\lambda.$$}
\begin{proof}
We denote the vertices in $K_\lambda$ by $v_i^j, 1 \leq i \leq n, 1 \leq j \leq \lambda_i$, where $\{v_i^j\}_{j=1}^{\lambda_i}$ is an independent set. Fixing a labeling on $H$, we first show that $\lambda$ is the only type of proper $H$-coloring possible. Let $u_i$ be the vertex in $H$ whose label is used to color $v_i^1$. Then, $\{u_i\}_{i=1}^n$ induces a clique $K_n$ in $H$. Assume, by way of contradiction, that $v_l^j$ is colored by the label on $u_l' \neq u_l$. Then, $u_l' \sim_H u_i$ for every $i \neq l$. But then $\{u_i\}_{i=1}^n \cup \{u_l'\}$ induces $K_{n+1}^-$ or $K_{n+1}$, depending on whether $u_l \sim_H u_l'$, which is a contradiction either way. For each subgraph of $H$ that is isomorphic to $K_n$, we can obtain $n!$ distinct proper $H$-colorings of $G$ of type $\lambda$. Hence, under the fixed labeling of $H$, we can get $m \cdot n!$ such colorings in total. The first equality is proved. The second equality follows from the definition of $n$-augmented monomials. 
\end{proof}
\end{prop}

\begin{exm}
We give an example illustrating the concept in the proof for Proposition \ref{prop.mulpartmonomials}. 
Consider a partition $\lambda = (4,3,2,2) \vdash 11$. We'll show that if $G = K_{4,3,2,2}$ and $H = K_4$, then $X_{G}^{H} = c_{(4,3,2,2)} m_{(4,3,2,2)}$. 

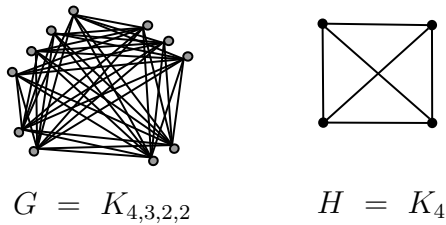
\begin{figure}[ht]
    \centering
    \begin{tikzpicture}[x=0.75pt,y=0.75pt,yscale=-1,xscale=1]

\draw  [fill={rgb, 255:red, 0; green, 0; blue, 0 }  ,fill opacity=1 ] (179.6,71.64) .. controls (179.6,70.46) and (180.53,69.51) .. (181.67,69.51) .. controls (182.81,69.51) and (183.73,70.46) .. (183.73,71.64) .. controls (183.73,72.81) and (182.81,73.77) .. (181.67,73.77) .. controls (180.53,73.77) and (179.6,72.81) .. (179.6,71.64) -- cycle ;
\draw  [fill={rgb, 255:red, 0; green, 0; blue, 0 }  ,fill opacity=1 ] (234.43,22.43) .. controls (234.43,21.26) and (235.36,20.31) .. (236.5,20.31) .. controls (237.64,20.31) and (238.56,21.26) .. (238.56,22.43) .. controls (238.56,23.61) and (237.64,24.56) .. (236.5,24.56) .. controls (235.36,24.56) and (234.43,23.61) .. (234.43,22.43) -- cycle ;
\draw  [fill={rgb, 255:red, 0; green, 0; blue, 0 }  ,fill opacity=1 ] (179.27,21.72) .. controls (179.27,20.54) and (180.2,19.59) .. (181.34,19.59) .. controls (182.48,19.59) and (183.4,20.54) .. (183.4,21.72) .. controls (183.4,22.89) and (182.48,23.85) .. (181.34,23.85) .. controls (180.2,23.85) and (179.27,22.89) .. (179.27,21.72) -- cycle ;
\draw  [fill={rgb, 255:red, 0; green, 0; blue, 0 }  ,fill opacity=1 ] (234.54,71.82) .. controls (234.54,70.64) and (235.46,69.69) .. (236.6,69.69) .. controls (237.74,69.69) and (238.67,70.64) .. (238.67,71.82) .. controls (238.67,72.99) and (237.74,73.94) .. (236.6,73.94) .. controls (235.46,73.94) and (234.54,72.99) .. (234.54,71.82) -- cycle ;
\draw [fill={rgb, 255:red, 0; green, 0; blue, 0 }  ,fill opacity=1 ]   (181.34,23.85) -- (181.67,69.51) ;
\draw    (183.4,21.72) -- (234.43,22.43) ;
\draw    (236.5,24.56) -- (236.6,69.69) ;
\draw    (183.73,71.64) -- (234.54,71.82) ;
\draw    (183.33,70.33) -- (234.67,24.11) ;
\draw    (182.67,23.44) -- (235.33,69.67) ;
\draw  [fill={rgb, 255:red, 155; green, 155; blue, 155 }  ,fill opacity=1 ] (22.6,46.12) .. controls (22.6,44.95) and (23.6,44) .. (24.82,44) .. controls (26.05,44) and (27.04,44.95) .. (27.04,46.12) .. controls (27.04,47.29) and (26.05,48.24) .. (24.82,48.24) .. controls (23.6,48.24) and (22.6,47.29) .. (22.6,46.12) -- cycle ;
\draw  [fill={rgb, 255:red, 155; green, 155; blue, 155 }  ,fill opacity=1 ] (42.28,25.2) .. controls (42.28,24.03) and (43.28,23.08) .. (44.5,23.08) .. controls (45.73,23.08) and (46.72,24.03) .. (46.72,25.2) .. controls (46.72,26.37) and (45.73,27.32) .. (44.5,27.32) .. controls (43.28,27.32) and (42.28,26.37) .. (42.28,25.2) -- cycle ;
\draw  [fill={rgb, 255:red, 155; green, 155; blue, 155 }  ,fill opacity=1 ] (32.28,35.6) .. controls (32.28,34.43) and (33.28,33.48) .. (34.5,33.48) .. controls (35.73,33.48) and (36.72,34.43) .. (36.72,35.6) .. controls (36.72,36.77) and (35.73,37.71) .. (34.5,37.71) .. controls (33.28,37.71) and (32.28,36.77) .. (32.28,35.6) -- cycle ;
\draw  [fill={rgb, 255:red, 155; green, 155; blue, 155 }  ,fill opacity=1 ] (53.5,15.19) .. controls (53.5,14.02) and (54.5,13.07) .. (55.72,13.07) .. controls (56.95,13.07) and (57.94,14.02) .. (57.94,15.19) .. controls (57.94,16.36) and (56.95,17.31) .. (55.72,17.31) .. controls (54.5,17.31) and (53.5,16.36) .. (53.5,15.19) -- cycle ;
\draw  [fill={rgb, 255:red, 155; green, 155; blue, 155 }  ,fill opacity=1 ] (26.73,74.81) .. controls (27.62,74.06) and (28.99,74.21) .. (29.78,75.15) .. controls (30.56,76.09) and (30.48,77.46) .. (29.58,78.21) .. controls (28.69,78.96) and (27.32,78.81) .. (26.53,77.87) .. controls (25.74,76.94) and (25.83,75.57) .. (26.73,74.81) -- cycle ;
\draw  [fill={rgb, 255:red, 155; green, 155; blue, 155 }  ,fill opacity=1 ] (101.4,30.42) .. controls (101.4,29.25) and (102.43,28.3) .. (103.7,28.3) .. controls (104.97,28.3) and (106,29.25) .. (106,30.42) .. controls (106,31.59) and (104.97,32.54) .. (103.7,32.54) .. controls (102.43,32.54) and (101.4,31.59) .. (101.4,30.42) -- cycle ;
\draw  [fill={rgb, 255:red, 155; green, 155; blue, 155 }  ,fill opacity=1 ] (90.82,22.6) .. controls (90.82,21.43) and (91.85,20.48) .. (93.12,20.48) .. controls (94.39,20.48) and (95.41,21.43) .. (95.41,22.6) .. controls (95.41,23.77) and (94.39,24.71) .. (93.12,24.71) .. controls (91.85,24.71) and (90.82,23.77) .. (90.82,22.6) -- cycle ;
\draw  [fill={rgb, 255:red, 155; green, 155; blue, 155 }  ,fill opacity=1 ] (34.3,84.43) .. controls (35.19,83.68) and (36.56,83.83) .. (37.35,84.77) .. controls (38.14,85.71) and (38.05,87.08) .. (37.15,87.83) .. controls (36.26,88.58) and (34.89,88.43) .. (34.1,87.49) .. controls (33.32,86.56) and (33.4,85.19) .. (34.3,84.43) -- cycle ;
\draw  [fill={rgb, 255:red, 155; green, 155; blue, 155 }  ,fill opacity=1 ] (94.12,92.16) .. controls (93.49,91.17) and (93.83,89.84) .. (94.86,89.18) .. controls (95.9,88.53) and (97.25,88.8) .. (97.87,89.79) .. controls (98.5,90.78) and (98.16,92.11) .. (97.12,92.76) .. controls (96.09,93.42) and (94.74,93.15) .. (94.12,92.16) -- cycle ;
\draw  [fill={rgb, 255:red, 155; green, 155; blue, 155 }  ,fill opacity=1 ] (111.07,38.26) .. controls (111.07,37.09) and (112.1,36.14) .. (113.37,36.14) .. controls (114.64,36.14) and (115.67,37.09) .. (115.67,38.26) .. controls (115.67,39.43) and (114.64,40.38) .. (113.37,40.38) .. controls (112.1,40.38) and (111.07,39.43) .. (111.07,38.26) -- cycle ;
\draw  [fill={rgb, 255:red, 155; green, 155; blue, 155 }  ,fill opacity=1 ] (104.68,85.96) .. controls (104.05,84.98) and (104.39,83.64) .. (105.42,82.99) .. controls (106.46,82.33) and (107.81,82.61) .. (108.43,83.6) .. controls (109.05,84.58) and (108.72,85.92) .. (107.68,86.57) .. controls (106.65,87.23) and (105.3,86.95) .. (104.68,85.96) -- cycle ;
\draw [color={rgb, 255:red, 0; green, 0; blue, 0 }  ,draw opacity=1 ][fill={rgb, 255:red, 155; green, 155; blue, 155 }  ,fill opacity=1 ]   (26.78,47.56) -- (29.78,75.15) ;
\draw [color={rgb, 255:red, 0; green, 0; blue, 0 }  ,draw opacity=1 ][fill={rgb, 255:red, 155; green, 155; blue, 155 }  ,fill opacity=1 ]   (26.78,47.56) -- (37.35,84.77) ;
\draw [fill={rgb, 255:red, 155; green, 155; blue, 155 }  ,fill opacity=1 ]   (26.78,47.56) -- (94.86,89.18) ;
\draw [fill={rgb, 255:red, 155; green, 155; blue, 155 }  ,fill opacity=1 ]   (26.78,47.56) -- (105.42,82.99) ;
\draw [fill={rgb, 255:red, 155; green, 155; blue, 155 }  ,fill opacity=1 ]   (26.78,47.56) -- (111.24,39.56) ;
\draw [fill={rgb, 255:red, 155; green, 155; blue, 155 }  ,fill opacity=1 ]   (26.78,47.56) -- (100.65,31.33) ;
\draw [fill={rgb, 255:red, 155; green, 155; blue, 155 }  ,fill opacity=1 ]   (26.78,47.56) -- (90.76,23.78) ;
\draw [fill={rgb, 255:red, 155; green, 155; blue, 155 }  ,fill opacity=1 ]   (37.07,35.6) -- (90.76,23.78) ;
\draw [fill={rgb, 255:red, 155; green, 155; blue, 155 }  ,fill opacity=1 ]   (37.07,35.6) -- (100.65,31.33) ;
\draw [fill={rgb, 255:red, 155; green, 155; blue, 155 }  ,fill opacity=1 ]   (37.07,35.6) -- (111.24,39.56) ;
\draw [fill={rgb, 255:red, 155; green, 155; blue, 155 }  ,fill opacity=1 ]   (36.72,35.6) -- (105.42,82.99) ;
\draw [fill={rgb, 255:red, 155; green, 155; blue, 155 }  ,fill opacity=1 ]   (37.67,36.67) -- (94.86,89.18) ;
\draw [fill={rgb, 255:red, 155; green, 155; blue, 155 }  ,fill opacity=1 ]   (36.72,35.6) -- (37.35,84.77) ;
\draw [fill={rgb, 255:red, 155; green, 155; blue, 155 }  ,fill opacity=1 ]   (37.67,36.67) -- (29.78,75.15) ;
\draw [fill={rgb, 255:red, 155; green, 155; blue, 155 }  ,fill opacity=1 ]   (45.89,26.67) -- (29.78,75.15) ;
\draw [fill={rgb, 255:red, 155; green, 155; blue, 155 }  ,fill opacity=1 ]   (45.89,26.67) -- (37.35,84.77) ;
\draw [fill={rgb, 255:red, 155; green, 155; blue, 155 }  ,fill opacity=1 ]   (45.13,27.32) -- (90.76,23.78) ;
\draw [fill={rgb, 255:red, 155; green, 155; blue, 155 }  ,fill opacity=1 ]   (46.57,26.67) -- (100.65,31.33) ;
\draw [fill={rgb, 255:red, 155; green, 155; blue, 155 }  ,fill opacity=1 ]   (46.57,26.67) -- (111.24,39.56) ;
\draw [fill={rgb, 255:red, 155; green, 155; blue, 155 }  ,fill opacity=1 ]   (45.89,26.67) -- (94.86,89.18) ;
\draw [fill={rgb, 255:red, 155; green, 155; blue, 155 }  ,fill opacity=1 ]   (45.89,26.67) -- (105.42,82.99) ;
\draw [fill={rgb, 255:red, 155; green, 155; blue, 155 }  ,fill opacity=1 ]   (55.72,17.31) -- (29.78,75.15) ;
\draw [fill={rgb, 255:red, 155; green, 155; blue, 155 }  ,fill opacity=1 ]   (55.72,17.31) -- (37.35,84.77) ;
\draw [fill={rgb, 255:red, 155; green, 155; blue, 155 }  ,fill opacity=1 ]   (55.72,17.31) -- (94.86,89.18) ;
\draw [fill={rgb, 255:red, 155; green, 155; blue, 155 }  ,fill opacity=1 ]   (55.72,17.31) -- (105.42,82.99) ;
\draw [fill={rgb, 255:red, 155; green, 155; blue, 155 }  ,fill opacity=1 ]   (56.76,17.31) -- (90.76,23.78) ;
\draw [fill={rgb, 255:red, 155; green, 155; blue, 155 }  ,fill opacity=1 ]   (57.85,18.22) -- (100.65,31.33) ;
\draw [fill={rgb, 255:red, 155; green, 155; blue, 155 }  ,fill opacity=1 ]   (55.72,17.31) -- (111.24,39.56) ;
\draw [fill={rgb, 255:red, 155; green, 155; blue, 155 }  ,fill opacity=1 ]   (37.35,84.77) -- (94.86,89.18) ;
\draw [fill={rgb, 255:red, 155; green, 155; blue, 155 }  ,fill opacity=1 ]   (37.35,84.77) -- (105.42,82.99) ;
\draw [fill={rgb, 255:red, 155; green, 155; blue, 155 }  ,fill opacity=1 ]   (37.35,84.77) -- (111.24,39.56) ;
\draw [fill={rgb, 255:red, 155; green, 155; blue, 155 }  ,fill opacity=1 ]   (37.35,84.77) -- (100.65,31.33) ;
\draw [fill={rgb, 255:red, 155; green, 155; blue, 155 }  ,fill opacity=1 ]   (37.35,84.77) -- (90.76,23.78) ;
\draw [fill={rgb, 255:red, 155; green, 155; blue, 155 }  ,fill opacity=1 ]   (29.78,75.15) -- (94.86,89.18) ;
\draw [fill={rgb, 255:red, 155; green, 155; blue, 155 }  ,fill opacity=1 ]   (29.78,75.15) -- (105.42,82.99) ;
\draw [fill={rgb, 255:red, 155; green, 155; blue, 155 }  ,fill opacity=1 ]   (29.78,75.15) -- (111.24,39.56) ;
\draw [fill={rgb, 255:red, 155; green, 155; blue, 155 }  ,fill opacity=1 ]   (29.78,75.15) -- (100.65,31.33) ;
\draw [fill={rgb, 255:red, 155; green, 155; blue, 155 }  ,fill opacity=1 ]   (29.78,75.15) -- (90.76,23.78) ;
\draw [fill={rgb, 255:red, 155; green, 155; blue, 155 }  ,fill opacity=1 ]   (90.76,23.78) -- (94.86,89.18) ;
\draw [fill={rgb, 255:red, 155; green, 155; blue, 155 }  ,fill opacity=1 ]   (94.86,89.18) -- (101.4,30.42) ;
\draw [fill={rgb, 255:red, 155; green, 155; blue, 155 }  ,fill opacity=1 ]   (94.86,89.18) -- (111.24,39.56) ;
\draw [fill={rgb, 255:red, 155; green, 155; blue, 155 }  ,fill opacity=1 ]   (105.42,82.99) -- (90.76,23.78) ;
\draw [fill={rgb, 255:red, 155; green, 155; blue, 155 }  ,fill opacity=1 ]   (105.42,82.99) -- (101.4,30.42) ;
\draw [fill={rgb, 255:red, 155; green, 155; blue, 155 }  ,fill opacity=1 ]   (105.42,82.99) -- (111.24,39.56) ;

\draw (23.78,105.29) node [anchor=north west][inner sep=0.75pt]    {$G\ =\ K_{4,3,2,2}$};
\draw (176.89,105.07) node [anchor=north west][inner sep=0.75pt]    {$H\ =\ K_{4} \ $};

\end{tikzpicture}
    \caption{The $G$ and $H$ whose $X_{G}^{H}$ corresponds to $m_{4,3,2,2}$}
    \label{fig:k4322andk4}
\end{figure}

Let's denote the maximally independent sets of $G$ by $V_1 , V_2 , V_3 , V_4$. Suppose $|V_{1}| = 4, |V_{2}| = 3, |V_{3}| = 2, |V_{4}| = 2$. Denote the vertices in $V_i$ by $\{v_{ij}\}_{j=1}^{|V_{i}|}$.

Now let's suppose we choose one vertex from each maximally independent set $V_i$. Suppose we choose $v_{11} \in V_{1} , v_{23} \in V_{2}, v_{31} \in V_{3}, v_{42} \in V_{4}$. If we look at the induced subgraph created by these vertices, we see that it is $K_4$ since $v_{ij} \in V_{i}$ must be connected to every $v_{ab} \in V_{a}$ (assuming $a \neq i$).

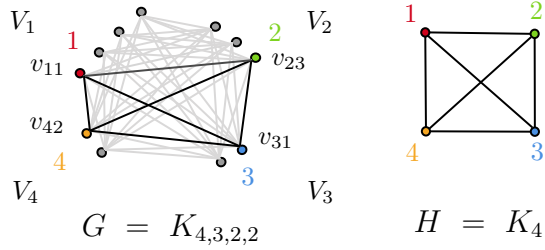
\begin{figure}[ht]
    \centering
    \begin{tikzpicture}[x=0.75pt,y=0.75pt,yscale=-1,xscale=1]

\draw  [fill={rgb, 255:red, 245; green, 166; blue, 35 }  ,fill opacity=1 ] (220.6,74.64) .. controls (220.6,73.46) and (221.53,72.51) .. (222.67,72.51) .. controls (223.81,72.51) and (224.73,73.46) .. (224.73,74.64) .. controls (224.73,75.81) and (223.81,76.77) .. (222.67,76.77) .. controls (221.53,76.77) and (220.6,75.81) .. (220.6,74.64) -- cycle ;
\draw  [fill={rgb, 255:red, 126; green, 211; blue, 33 }  ,fill opacity=1 ] (275.43,25.43) .. controls (275.43,24.26) and (276.36,23.31) .. (277.5,23.31) .. controls (278.64,23.31) and (279.56,24.26) .. (279.56,25.43) .. controls (279.56,26.61) and (278.64,27.56) .. (277.5,27.56) .. controls (276.36,27.56) and (275.43,26.61) .. (275.43,25.43) -- cycle ;
\draw  [fill={rgb, 255:red, 208; green, 2; blue, 27 }  ,fill opacity=1 ] (220.27,24.72) .. controls (220.27,23.54) and (221.2,22.59) .. (222.34,22.59) .. controls (223.48,22.59) and (224.4,23.54) .. (224.4,24.72) .. controls (224.4,25.89) and (223.48,26.85) .. (222.34,26.85) .. controls (221.2,26.85) and (220.27,25.89) .. (220.27,24.72) -- cycle ;
\draw  [fill={rgb, 255:red, 74; green, 144; blue, 226 }  ,fill opacity=1 ] (275.54,74.82) .. controls (275.54,73.64) and (276.46,72.69) .. (277.6,72.69) .. controls (278.74,72.69) and (279.67,73.64) .. (279.67,74.82) .. controls (279.67,75.99) and (278.74,76.94) .. (277.6,76.94) .. controls (276.46,76.94) and (275.54,75.99) .. (275.54,74.82) -- cycle ;
\draw [fill={rgb, 255:red, 0; green, 0; blue, 0 }  ,fill opacity=1 ]   (222.34,26.85) -- (222.67,72.51) ;
\draw    (224.4,24.72) -- (275.43,25.43) ;
\draw    (277.5,27.56) -- (277.6,72.69) ;
\draw    (224.73,74.64) -- (275.54,74.82) ;
\draw    (224.33,73.33) -- (275.67,27.11) ;
\draw    (223.67,26.44) -- (276.33,72.67) ;
\draw  [color={rgb, 255:red, 0; green, 0; blue, 0 }  ,draw opacity=1 ][fill={rgb, 255:red, 208; green, 2; blue, 27 }  ,fill opacity=1 ] (46,45.32) .. controls (46,44.15) and (47,43.2) .. (48.22,43.2) .. controls (49.45,43.2) and (50.44,44.15) .. (50.44,45.32) .. controls (50.44,46.49) and (49.45,47.44) .. (48.22,47.44) .. controls (47,47.44) and (46,46.49) .. (46,45.32) -- cycle ;
\draw  [color={rgb, 255:red, 0; green, 0; blue, 0 }  ,draw opacity=1 ][fill={rgb, 255:red, 155; green, 155; blue, 155 }  ,fill opacity=1 ] (65.68,24.4) .. controls (65.68,23.23) and (66.68,22.28) .. (67.9,22.28) .. controls (69.13,22.28) and (70.12,23.23) .. (70.12,24.4) .. controls (70.12,25.57) and (69.13,26.52) .. (67.9,26.52) .. controls (66.68,26.52) and (65.68,25.57) .. (65.68,24.4) -- cycle ;
\draw  [color={rgb, 255:red, 0; green, 0; blue, 0 }  ,draw opacity=1 ][fill={rgb, 255:red, 155; green, 155; blue, 155 }  ,fill opacity=1 ] (55.68,34.8) .. controls (55.68,33.63) and (56.68,32.68) .. (57.9,32.68) .. controls (59.13,32.68) and (60.12,33.63) .. (60.12,34.8) .. controls (60.12,35.97) and (59.13,36.91) .. (57.9,36.91) .. controls (56.68,36.91) and (55.68,35.97) .. (55.68,34.8) -- cycle ;
\draw  [color={rgb, 255:red, 0; green, 0; blue, 0 }  ,draw opacity=1 ][fill={rgb, 255:red, 155; green, 155; blue, 155 }  ,fill opacity=1 ] (76.9,14.39) .. controls (76.9,13.22) and (77.9,12.27) .. (79.12,12.27) .. controls (80.35,12.27) and (81.34,13.22) .. (81.34,14.39) .. controls (81.34,15.56) and (80.35,16.51) .. (79.12,16.51) .. controls (77.9,16.51) and (76.9,15.56) .. (76.9,14.39) -- cycle ;
\draw  [color={rgb, 255:red, 0; green, 0; blue, 0 }  ,draw opacity=1 ][fill={rgb, 255:red, 245; green, 166; blue, 35 }  ,fill opacity=1 ] (50.13,74.01) .. controls (51.02,73.26) and (52.39,73.41) .. (53.18,74.35) .. controls (53.96,75.29) and (53.88,76.66) .. (52.98,77.41) .. controls (52.09,78.16) and (50.72,78.01) .. (49.93,77.07) .. controls (49.14,76.14) and (49.23,74.77) .. (50.13,74.01) -- cycle ;
\draw  [color={rgb, 255:red, 0; green, 0; blue, 0 }  ,draw opacity=1 ][fill={rgb, 255:red, 155; green, 155; blue, 155 }  ,fill opacity=1 ] (124.8,29.62) .. controls (124.8,28.45) and (125.83,27.5) .. (127.1,27.5) .. controls (128.37,27.5) and (129.4,28.45) .. (129.4,29.62) .. controls (129.4,30.79) and (128.37,31.74) .. (127.1,31.74) .. controls (125.83,31.74) and (124.8,30.79) .. (124.8,29.62) -- cycle ;
\draw  [color={rgb, 255:red, 0; green, 0; blue, 0 }  ,draw opacity=1 ][fill={rgb, 255:red, 155; green, 155; blue, 155 }  ,fill opacity=1 ] (114.22,21.8) .. controls (114.22,20.63) and (115.25,19.68) .. (116.52,19.68) .. controls (117.79,19.68) and (118.81,20.63) .. (118.81,21.8) .. controls (118.81,22.97) and (117.79,23.91) .. (116.52,23.91) .. controls (115.25,23.91) and (114.22,22.97) .. (114.22,21.8) -- cycle ;
\draw  [color={rgb, 255:red, 0; green, 0; blue, 0 }  ,draw opacity=1 ][fill={rgb, 255:red, 155; green, 155; blue, 155 }  ,fill opacity=1 ] (57.7,83.63) .. controls (58.59,82.88) and (59.96,83.03) .. (60.75,83.97) .. controls (61.54,84.91) and (61.45,86.28) .. (60.55,87.03) .. controls (59.66,87.78) and (58.29,87.63) .. (57.5,86.69) .. controls (56.72,85.76) and (56.8,84.39) .. (57.7,83.63) -- cycle ;
\draw  [color={rgb, 255:red, 0; green, 0; blue, 0 }  ,draw opacity=1 ][fill={rgb, 255:red, 155; green, 155; blue, 155 }  ,fill opacity=1 ] (117.52,91.36) .. controls (116.89,90.37) and (117.23,89.04) .. (118.26,88.38) .. controls (119.3,87.73) and (120.65,88) .. (121.27,88.99) .. controls (121.9,89.98) and (121.56,91.31) .. (120.52,91.96) .. controls (119.49,92.62) and (118.14,92.35) .. (117.52,91.36) -- cycle ;
\draw  [color={rgb, 255:red, 0; green, 0; blue, 0 }  ,draw opacity=1 ][fill={rgb, 255:red, 126; green, 211; blue, 33 }  ,fill opacity=1 ] (134.47,37.46) .. controls (134.47,36.29) and (135.5,35.34) .. (136.77,35.34) .. controls (138.04,35.34) and (139.07,36.29) .. (139.07,37.46) .. controls (139.07,38.63) and (138.04,39.58) .. (136.77,39.58) .. controls (135.5,39.58) and (134.47,38.63) .. (134.47,37.46) -- cycle ;
\draw  [color={rgb, 255:red, 0; green, 0; blue, 0 }  ,draw opacity=1 ][fill={rgb, 255:red, 74; green, 144; blue, 226 }  ,fill opacity=1 ] (128.08,85.16) .. controls (127.45,84.18) and (127.79,82.84) .. (128.82,82.19) .. controls (129.86,81.53) and (131.21,81.81) .. (131.83,82.8) .. controls (132.45,83.78) and (132.12,85.12) .. (131.08,85.77) .. controls (130.05,86.43) and (128.7,86.15) .. (128.08,85.16) -- cycle ;
\draw [color={rgb, 255:red, 217; green, 216; blue, 216 }  ,draw opacity=1 ]   (50.18,46.76) -- (60.75,83.97) ;
\draw [color={rgb, 255:red, 217; green, 216; blue, 216 }  ,draw opacity=1 ]   (50.18,46.76) -- (118.26,88.38) ;
\draw [color={rgb, 255:red, 213; green, 213; blue, 213 }  ,draw opacity=1 ]   (50.18,46.76) -- (124.05,30.53) ;
\draw [color={rgb, 255:red, 217; green, 216; blue, 216 }  ,draw opacity=1 ]   (50.18,46.76) -- (114.16,22.98) ;
\draw [color={rgb, 255:red, 217; green, 216; blue, 216 }  ,draw opacity=1 ]   (60.47,34.8) -- (114.16,22.98) ;
\draw [color={rgb, 255:red, 217; green, 216; blue, 216 }  ,draw opacity=1 ]   (60.47,34.8) -- (124.05,30.53) ;
\draw [color={rgb, 255:red, 217; green, 216; blue, 216 }  ,draw opacity=1 ]   (60.47,34.8) -- (134.64,38.76) ;
\draw [color={rgb, 255:red, 217; green, 216; blue, 216 }  ,draw opacity=1 ]   (60.12,34.8) -- (128.82,82.19) ;
\draw [color={rgb, 255:red, 217; green, 216; blue, 216 }  ,draw opacity=1 ]   (61.07,35.87) -- (118.26,88.38) ;
\draw [color={rgb, 255:red, 217; green, 216; blue, 216 }  ,draw opacity=1 ]   (60.12,34.8) -- (60.75,83.97) ;
\draw [color={rgb, 255:red, 217; green, 216; blue, 216 }  ,draw opacity=1 ]   (61.07,35.87) -- (53.18,74.35) ;
\draw [color={rgb, 255:red, 217; green, 216; blue, 216 }  ,draw opacity=1 ]   (69.29,25.87) -- (53.18,74.35) ;
\draw [color={rgb, 255:red, 217; green, 216; blue, 216 }  ,draw opacity=1 ]   (69.29,25.87) -- (60.75,83.97) ;
\draw [color={rgb, 255:red, 217; green, 216; blue, 216 }  ,draw opacity=1 ]   (68.53,26.52) -- (114.16,22.98) ;
\draw [color={rgb, 255:red, 217; green, 216; blue, 216 }  ,draw opacity=1 ]   (69.97,25.87) -- (124.05,30.53) ;
\draw [color={rgb, 255:red, 217; green, 216; blue, 216 }  ,draw opacity=1 ]   (69.97,25.87) -- (134.64,38.76) ;
\draw [color={rgb, 255:red, 217; green, 216; blue, 216 }  ,draw opacity=1 ]   (69.98,26.23) -- (118.96,88.75) ;
\draw [color={rgb, 255:red, 217; green, 216; blue, 216 }  ,draw opacity=1 ]   (69.29,25.87) -- (128.82,82.19) ;
\draw [color={rgb, 255:red, 217; green, 216; blue, 216 }  ,draw opacity=1 ]   (79.12,16.51) -- (53.18,74.35) ;
\draw [color={rgb, 255:red, 217; green, 216; blue, 216 }  ,draw opacity=1 ]   (79.12,16.51) -- (60.75,83.97) ;
\draw [color={rgb, 255:red, 217; green, 216; blue, 216 }  ,draw opacity=1 ]   (79.12,16.51) -- (118.26,88.38) ;
\draw [color={rgb, 255:red, 217; green, 216; blue, 216 }  ,draw opacity=1 ]   (79.12,16.51) -- (128.82,82.19) ;
\draw [color={rgb, 255:red, 217; green, 216; blue, 216 }  ,draw opacity=1 ]   (80.16,16.51) -- (114.16,22.98) ;
\draw [color={rgb, 255:red, 217; green, 216; blue, 216 }  ,draw opacity=1 ]   (81.25,17.42) -- (124.05,30.53) ;
\draw [color={rgb, 255:red, 217; green, 216; blue, 216 }  ,draw opacity=1 ]   (79.12,16.51) -- (134.64,38.76) ;
\draw [color={rgb, 255:red, 217; green, 216; blue, 216 }  ,draw opacity=1 ]   (60.75,83.97) -- (118.26,88.38) ;
\draw [color={rgb, 255:red, 217; green, 216; blue, 216 }  ,draw opacity=1 ]   (60.75,83.97) -- (128.82,82.19) ;
\draw [color={rgb, 255:red, 217; green, 216; blue, 216 }  ,draw opacity=1 ]   (60.75,83.97) -- (134.64,38.76) ;
\draw [color={rgb, 255:red, 217; green, 216; blue, 216 }  ,draw opacity=1 ]   (60.75,83.97) -- (124.05,30.53) ;
\draw [color={rgb, 255:red, 217; green, 216; blue, 216 }  ,draw opacity=1 ]   (60.75,83.97) -- (114.16,22.98) ;
\draw [color={rgb, 255:red, 217; green, 216; blue, 216 }  ,draw opacity=1 ]   (53.18,74.35) -- (118.26,88.38) ;
\draw [color={rgb, 255:red, 217; green, 216; blue, 216 }  ,draw opacity=1 ]   (53.18,74.35) -- (124.05,30.53) ;
\draw [color={rgb, 255:red, 217; green, 216; blue, 216 }  ,draw opacity=1 ]   (53.18,74.35) -- (114.16,22.98) ;
\draw [color={rgb, 255:red, 217; green, 216; blue, 216 }  ,draw opacity=1 ]   (114.16,22.98) -- (118.26,88.38) ;
\draw [color={rgb, 255:red, 217; green, 216; blue, 216 }  ,draw opacity=1 ]   (118.26,88.38) -- (124.8,29.62) ;
\draw [color={rgb, 255:red, 217; green, 216; blue, 216 }  ,draw opacity=1 ]   (118.26,88.38) -- (134.64,38.76) ;
\draw [color={rgb, 255:red, 217; green, 216; blue, 216 }  ,draw opacity=1 ]   (128.82,82.19) -- (114.16,22.98) ;
\draw [color={rgb, 255:red, 217; green, 216; blue, 216 }  ,draw opacity=1 ]   (128.82,82.19) -- (124.8,29.62) ;
\draw [color={rgb, 255:red, 0; green, 0; blue, 0 }  ,draw opacity=1 ]   (128.82,82.19) -- (134.64,38.76) ;
\draw [color={rgb, 255:red, 0; green, 0; blue, 0 }  ,draw opacity=1 ]   (53.18,74.35) -- (128.82,82.19) ;
\draw [color={rgb, 255:red, 0; green, 0; blue, 0 }  ,draw opacity=1 ]   (53.18,74.35) -- (134.64,38.76) ;
\draw [color={rgb, 255:red, 0; green, 0; blue, 0 }  ,draw opacity=1 ]   (50.18,46.76) -- (128.82,82.19) ;
\draw [color={rgb, 255:red, 74; green, 74; blue, 74 }  ,draw opacity=1 ]   (50.18,46.76) -- (134.64,38.76) ;
\draw [color={rgb, 255:red, 0; green, 0; blue, 0 }  ,draw opacity=1 ]   (50.18,46.76) -- (53.18,74.35) ;

\draw (47.58,115.29) node [anchor=north west][inner sep=0.75pt]    {$G\ =\ K_{4,3,2,2}$};
\draw (215.14,113.07) node [anchor=north west][inner sep=0.75pt]    {$H\ =\ K_{4} \ $};
\draw (21.6,36.8) node [anchor=north west][inner sep=0.75pt]  [font=\footnotesize]  {$v_{11}$};
\draw (143.2,34) node [anchor=north west][inner sep=0.75pt]  [font=\footnotesize]  {$v_{23}$};
\draw (136.6,73.6) node [anchor=north west][inner sep=0.75pt]  [font=\footnotesize]  {$v_{31}$};
\draw (21.2,66) node [anchor=north west][inner sep=0.75pt]  [font=\footnotesize]  {$v_{42}$};
\draw (39.85,22.95) node [anchor=north west][inner sep=0.75pt]  [font=\footnotesize,color={rgb, 255:red, 208; green, 2; blue, 27 }  ,opacity=1 ]  {$1$};
\draw (141.83,18.75) node [anchor=north west][inner sep=0.75pt]  [font=\footnotesize,color={rgb, 255:red, 126; green, 211; blue, 33 }  ,opacity=1 ]  {$2$};
\draw (128.15,91.65) node [anchor=north west][inner sep=0.75pt]  [font=\footnotesize,color={rgb, 255:red, 74; green, 144; blue, 226 }  ,opacity=1 ]  {$3$};
\draw (33.2,83.88) node [anchor=north west][inner sep=0.75pt]  [font=\footnotesize,color={rgb, 255:red, 245; green, 166; blue, 35 }  ,opacity=1 ]  {$4$};
\draw (210.2,9.2) node [anchor=north west][inner sep=0.75pt]  [font=\footnotesize,color={rgb, 255:red, 208; green, 2; blue, 27 }  ,opacity=1 ]  {$1$};
\draw (273.6,7.8) node [anchor=north west][inner sep=0.75pt]  [font=\footnotesize,color={rgb, 255:red, 126; green, 211; blue, 33 }  ,opacity=1 ]  {$2$};
\draw (274.14,79.02) node [anchor=north west][inner sep=0.75pt]  [font=\footnotesize,color={rgb, 255:red, 74; green, 144; blue, 226 }  ,opacity=1 ]  {$3$};
\draw (211,79.4) node [anchor=north west][inner sep=0.75pt]  [font=\footnotesize,color={rgb, 255:red, 245; green, 166; blue, 35 }  ,opacity=1 ]  {$4$};
\draw (11.2,12.2) node [anchor=north west][inner sep=0.75pt]  [font=\footnotesize]  {$V_{1}$};
\draw (160.8,12.4) node [anchor=north west][inner sep=0.75pt]  [font=\footnotesize]  {$V_{2}$};
\draw (12.4,98.4) node [anchor=north west][inner sep=0.75pt]  [font=\footnotesize]  {$V_{4}$};
\draw (161.2,98) node [anchor=north west][inner sep=0.75pt]  [font=\footnotesize]  {$V_{3}$};

\end{tikzpicture}
    \caption{A Colored $K_4$ subgraph in $K_{4,3,2,2}$}
    \label{fig:k4subgraphk4322}
\end{figure}

Now suppose we keep $v_{11} \in V_{1} , v_{23} \in V_{2}, v_{42} \in V_{4}$ but replace $v_{31} \in V_{3}$ with $v_{32} \in V_{3}$. Now suppose without loss of generality that $v_{11}$ is $H$-colored 1, $v_{23}$ is $H$-colored 2, and $v_{42}$ is $H$-colored 4, then the only option for $H$-coloring $v_{32}$ is the color 3 because on the graph $H$, the only vertex adjacent to the vertices colored 1,2,4 is the vertex colored 3.

\begin{figure}[ht]
    \centering
    \begin{tikzpicture}[x=0.75pt,y=0.75pt,yscale=-1,xscale=1]
\draw  [fill={rgb, 255:red, 245; green, 166; blue, 35 }  ,fill opacity=1 ] (218.6,89.64) .. controls (218.6,88.46) and (219.53,87.51) .. (220.67,87.51) .. controls (221.81,87.51) and (222.73,88.46) .. (222.73,89.64) .. controls (222.73,90.81) and (221.81,91.77) .. (220.67,91.77) .. controls (219.53,91.77) and (218.6,90.81) .. (218.6,89.64) -- cycle ;
\draw  [fill={rgb, 255:red, 126; green, 211; blue, 33 }  ,fill opacity=1 ] (273.43,40.43) .. controls (273.43,39.26) and (274.36,38.31) .. (275.5,38.31) .. controls (276.64,38.31) and (277.56,39.26) .. (277.56,40.43) .. controls (277.56,41.61) and (276.64,42.56) .. (275.5,42.56) .. controls (274.36,42.56) and (273.43,41.61) .. (273.43,40.43) -- cycle ;
\draw  [fill={rgb, 255:red, 208; green, 2; blue, 27 }  ,fill opacity=1 ] (218.27,39.72) .. controls (218.27,38.54) and (219.2,37.59) .. (220.34,37.59) .. controls (221.48,37.59) and (222.4,38.54) .. (222.4,39.72) .. controls (222.4,40.89) and (221.48,41.85) .. (220.34,41.85) .. controls (219.2,41.85) and (218.27,40.89) .. (218.27,39.72) -- cycle ;
\draw  [fill={rgb, 255:red, 74; green, 144; blue, 226 }  ,fill opacity=1 ] (273.54,89.82) .. controls (273.54,88.64) and (274.46,87.69) .. (275.6,87.69) .. controls (276.74,87.69) and (277.67,88.64) .. (277.67,89.82) .. controls (277.67,90.99) and (276.74,91.94) .. (275.6,91.94) .. controls (274.46,91.94) and (273.54,90.99) .. (273.54,89.82) -- cycle ;
\draw [fill={rgb, 255:red, 0; green, 0; blue, 0 }  ,fill opacity=1 ]   (220.34,41.85) -- (220.67,87.51) ;
\draw    (222.4,39.72) -- (273.43,40.43) ;
\draw    (275.5,42.56) -- (275.6,87.69) ;
\draw    (222.73,89.64) -- (273.54,89.82) ;
\draw    (222.33,88.33) -- (273.67,42.11) ;
\draw    (221.67,41.44) -- (274.33,87.67) ;
\draw  [color={rgb, 255:red, 0; green, 0; blue, 0 }  ,draw opacity=1 ][fill={rgb, 255:red, 208; green, 2; blue, 27 }  ,fill opacity=1 ] (44,60.32) .. controls (44,59.15) and (45,58.2) .. (46.22,58.2) .. controls (47.45,58.2) and (48.44,59.15) .. (48.44,60.32) .. controls (48.44,61.49) and (47.45,62.44) .. (46.22,62.44) .. controls (45,62.44) and (44,61.49) .. (44,60.32) -- cycle ;
\draw  [color={rgb, 255:red, 0; green, 0; blue, 0 }  ,draw opacity=1 ][fill={rgb, 255:red, 155; green, 155; blue, 155 }  ,fill opacity=1 ] (63.68,39.4) .. controls (63.68,38.23) and (64.68,37.28) .. (65.9,37.28) .. controls (67.13,37.28) and (68.12,38.23) .. (68.12,39.4) .. controls (68.12,40.57) and (67.13,41.52) .. (65.9,41.52) .. controls (64.68,41.52) and (63.68,40.57) .. (63.68,39.4) -- cycle ;
\draw  [color={rgb, 255:red, 0; green, 0; blue, 0 }  ,draw opacity=1 ][fill={rgb, 255:red, 155; green, 155; blue, 155 }  ,fill opacity=1 ] (53.68,49.8) .. controls (53.68,48.63) and (54.68,47.68) .. (55.9,47.68) .. controls (57.13,47.68) and (58.12,48.63) .. (58.12,49.8) .. controls (58.12,50.97) and (57.13,51.91) .. (55.9,51.91) .. controls (54.68,51.91) and (53.68,50.97) .. (53.68,49.8) -- cycle ;
\draw  [color={rgb, 255:red, 0; green, 0; blue, 0 }  ,draw opacity=1 ][fill={rgb, 255:red, 155; green, 155; blue, 155 }  ,fill opacity=1 ] (74.9,29.39) .. controls (74.9,28.22) and (75.9,27.27) .. (77.12,27.27) .. controls (78.35,27.27) and (79.34,28.22) .. (79.34,29.39) .. controls (79.34,30.56) and (78.35,31.51) .. (77.12,31.51) .. controls (75.9,31.51) and (74.9,30.56) .. (74.9,29.39) -- cycle ;
\draw  [color={rgb, 255:red, 0; green, 0; blue, 0 }  ,draw opacity=1 ][fill={rgb, 255:red, 245; green, 166; blue, 35 }  ,fill opacity=1 ] (48.13,89.01) .. controls (49.02,88.26) and (50.39,88.41) .. (51.18,89.35) .. controls (51.96,90.29) and (51.88,91.66) .. (50.98,92.41) .. controls (50.09,93.16) and (48.72,93.01) .. (47.93,92.07) .. controls (47.14,91.14) and (47.23,89.77) .. (48.13,89.01) -- cycle ;
\draw  [color={rgb, 255:red, 0; green, 0; blue, 0 }  ,draw opacity=1 ][fill={rgb, 255:red, 155; green, 155; blue, 155 }  ,fill opacity=1 ] (122.8,44.62) .. controls (122.8,43.45) and (123.83,42.5) .. (125.1,42.5) .. controls (126.37,42.5) and (127.4,43.45) .. (127.4,44.62) .. controls (127.4,45.79) and (126.37,46.74) .. (125.1,46.74) .. controls (123.83,46.74) and (122.8,45.79) .. (122.8,44.62) -- cycle ;
\draw  [color={rgb, 255:red, 0; green, 0; blue, 0 }  ,draw opacity=1 ][fill={rgb, 255:red, 155; green, 155; blue, 155 }  ,fill opacity=1 ] (112.22,36.8) .. controls (112.22,35.63) and (113.25,34.68) .. (114.52,34.68) .. controls (115.79,34.68) and (116.81,35.63) .. (116.81,36.8) .. controls (116.81,37.97) and (115.79,38.91) .. (114.52,38.91) .. controls (113.25,38.91) and (112.22,37.97) .. (112.22,36.8) -- cycle ;
\draw  [color={rgb, 255:red, 0; green, 0; blue, 0 }  ,draw opacity=1 ][fill={rgb, 255:red, 155; green, 155; blue, 155 }  ,fill opacity=1 ] (55.7,98.63) .. controls (56.59,97.88) and (57.96,98.03) .. (58.75,98.97) .. controls (59.54,99.91) and (59.45,101.28) .. (58.55,102.03) .. controls (57.66,102.78) and (56.29,102.63) .. (55.5,101.69) .. controls (54.72,100.76) and (54.8,99.39) .. (55.7,98.63) -- cycle ;
\draw  [color={rgb, 255:red, 0; green, 0; blue, 0 }  ,draw opacity=1 ][fill={rgb, 255:red, 74; green, 144; blue, 226 }  ,fill opacity=1 ] (115.52,106.36) .. controls (114.89,105.37) and (115.23,104.04) .. (116.26,103.38) .. controls (117.3,102.73) and (118.65,103) .. (119.27,103.99) .. controls (119.9,104.98) and (119.56,106.31) .. (118.52,106.96) .. controls (117.49,107.62) and (116.14,107.35) .. (115.52,106.36) -- cycle ;
\draw  [color={rgb, 255:red, 0; green, 0; blue, 0 }  ,draw opacity=1 ][fill={rgb, 255:red, 126; green, 211; blue, 33 }  ,fill opacity=1 ] (132.47,52.46) .. controls (132.47,51.29) and (133.5,50.34) .. (134.77,50.34) .. controls (136.04,50.34) and (137.07,51.29) .. (137.07,52.46) .. controls (137.07,53.63) and (136.04,54.58) .. (134.77,54.58) .. controls (133.5,54.58) and (132.47,53.63) .. (132.47,52.46) -- cycle ;
\draw  [color={rgb, 255:red, 0; green, 0; blue, 0 }  ,draw opacity=1 ][fill={rgb, 255:red, 74; green, 144; blue, 226 }  ,fill opacity=1 ] (126.08,100.16) .. controls (125.45,99.18) and (125.79,97.84) .. (126.82,97.19) .. controls (127.86,96.53) and (129.21,96.81) .. (129.83,97.8) .. controls (130.45,98.78) and (130.12,100.12) .. (129.08,100.77) .. controls (128.05,101.43) and (126.7,101.15) .. (126.08,100.16) -- cycle ;
\draw [color={rgb, 255:red, 217; green, 216; blue, 216 }  ,draw opacity=1 ]   (48.18,61.76) -- (58.75,98.97) ;
\draw [color={rgb, 255:red, 213; green, 213; blue, 213 }  ,draw opacity=1 ]   (48.18,61.76) -- (122.05,45.53) ;
\draw [color={rgb, 255:red, 217; green, 216; blue, 216 }  ,draw opacity=1 ]   (48.18,61.76) -- (112.16,37.98) ;
\draw [color={rgb, 255:red, 217; green, 216; blue, 216 }  ,draw opacity=1 ]   (58.47,49.8) -- (112.16,37.98) ;
\draw [color={rgb, 255:red, 217; green, 216; blue, 216 }  ,draw opacity=1 ]   (58.47,49.8) -- (122.05,45.53) ;
\draw [color={rgb, 255:red, 217; green, 216; blue, 216 }  ,draw opacity=1 ]   (58.47,49.8) -- (132.64,53.76) ;
\draw [color={rgb, 255:red, 217; green, 216; blue, 216 }  ,draw opacity=1 ]   (58.12,49.8) -- (126.82,97.19) ;
\draw [color={rgb, 255:red, 217; green, 216; blue, 216 }  ,draw opacity=1 ]   (59.07,50.87) -- (116.26,103.38) ;
\draw [color={rgb, 255:red, 217; green, 216; blue, 216 }  ,draw opacity=1 ][fill={rgb, 255:red, 155; green, 155; blue, 155 }  ,fill opacity=1 ]   (58.12,49.8) -- (58.75,98.97) ;
\draw [color={rgb, 255:red, 217; green, 216; blue, 216 }  ,draw opacity=1 ]   (59.07,50.87) -- (51.18,89.35) ;
\draw [color={rgb, 255:red, 217; green, 216; blue, 216 }  ,draw opacity=1 ]   (67.29,40.87) -- (51.18,89.35) ;
\draw [color={rgb, 255:red, 217; green, 216; blue, 216 }  ,draw opacity=1 ]   (67.29,40.87) -- (58.75,98.97) ;
\draw [color={rgb, 255:red, 217; green, 216; blue, 216 }  ,draw opacity=1 ]   (66.53,41.52) -- (112.16,37.98) ;
\draw [color={rgb, 255:red, 217; green, 216; blue, 216 }  ,draw opacity=1 ]   (67.97,40.87) -- (122.05,45.53) ;
\draw [color={rgb, 255:red, 217; green, 216; blue, 216 }  ,draw opacity=1 ]   (67.97,40.87) -- (132.64,53.76) ;
\draw [color={rgb, 255:red, 217; green, 216; blue, 216 }  ,draw opacity=1 ]   (67.98,41.23) -- (116.96,103.75) ;
\draw [color={rgb, 255:red, 217; green, 216; blue, 216 }  ,draw opacity=1 ]   (67.29,40.87) -- (126.82,97.19) ;
\draw [color={rgb, 255:red, 217; green, 216; blue, 216 }  ,draw opacity=1 ]   (77.12,31.51) -- (51.18,89.35) ;
\draw [color={rgb, 255:red, 217; green, 216; blue, 216 }  ,draw opacity=1 ]   (77.12,31.51) -- (58.75,98.97) ;
\draw [color={rgb, 255:red, 217; green, 216; blue, 216 }  ,draw opacity=1 ]   (77.12,31.51) -- (116.26,103.38) ;
\draw [color={rgb, 255:red, 217; green, 216; blue, 216 }  ,draw opacity=1 ]   (77.12,31.51) -- (126.82,97.19) ;
\draw [color={rgb, 255:red, 217; green, 216; blue, 216 }  ,draw opacity=1 ]   (78.16,31.51) -- (112.16,37.98) ;
\draw [color={rgb, 255:red, 217; green, 216; blue, 216 }  ,draw opacity=1 ]   (79.25,32.42) -- (122.05,45.53) ;
\draw [color={rgb, 255:red, 217; green, 216; blue, 216 }  ,draw opacity=1 ]   (77.12,31.51) -- (132.64,53.76) ;
\draw [color={rgb, 255:red, 217; green, 216; blue, 216 }  ,draw opacity=1 ]   (58.75,98.97) -- (116.26,103.38) ;
\draw [color={rgb, 255:red, 217; green, 216; blue, 216 }  ,draw opacity=1 ]   (58.75,98.97) -- (126.82,97.19) ;
\draw [color={rgb, 255:red, 217; green, 216; blue, 216 }  ,draw opacity=1 ]   (58.75,98.97) -- (132.64,53.76) ;
\draw [color={rgb, 255:red, 217; green, 216; blue, 216 }  ,draw opacity=1 ]   (58.75,98.97) -- (122.05,45.53) ;
\draw [color={rgb, 255:red, 217; green, 216; blue, 216 }  ,draw opacity=1 ]   (58.75,98.97) -- (112.16,37.98) ;
\draw [color={rgb, 255:red, 0; green, 0; blue, 0 }  ,draw opacity=1 ]   (51.18,89.35) -- (116.26,103.38) ;
\draw [color={rgb, 255:red, 217; green, 216; blue, 216 }  ,draw opacity=1 ]   (51.18,89.35) -- (122.05,45.53) ;
\draw [color={rgb, 255:red, 217; green, 216; blue, 216 }  ,draw opacity=1 ]   (51.18,89.35) -- (112.16,37.98) ;
\draw [color={rgb, 255:red, 217; green, 216; blue, 216 }  ,draw opacity=1 ]   (112.16,37.98) -- (116.26,103.38) ;
\draw [color={rgb, 255:red, 217; green, 216; blue, 216 }  ,draw opacity=1 ]   (116.26,103.38) -- (122.8,44.62) ;
\draw [color={rgb, 255:red, 217; green, 216; blue, 216 }  ,draw opacity=1 ]   (126.82,97.19) -- (112.16,37.98) ;
\draw [color={rgb, 255:red, 217; green, 216; blue, 216 }  ,draw opacity=1 ]   (126.82,97.19) -- (122.8,44.62) ;
\draw [color={rgb, 255:red, 0; green, 0; blue, 0 }  ,draw opacity=1 ]   (51.18,89.35) -- (126.82,97.19) ;
\draw [color={rgb, 255:red, 0; green, 0; blue, 0 }  ,draw opacity=1 ]   (51.18,89.35) -- (132.64,53.76) ;
\draw [color={rgb, 255:red, 0; green, 0; blue, 0 }  ,draw opacity=1 ]   (48.18,61.76) -- (126.82,97.19) ;
\draw [color={rgb, 255:red, 74; green, 74; blue, 74 }  ,draw opacity=1 ]   (48.18,61.76) -- (132.64,53.76) ;
\draw [color={rgb, 255:red, 0; green, 0; blue, 0 }  ,draw opacity=1 ]   (48.18,61.76) -- (51.18,89.35) ;
\draw [color={rgb, 255:red, 0; green, 0; blue, 0 }  ,draw opacity=1 ]   (126.82,97.19) -- (132.64,53.76) ;
\draw [color={rgb, 255:red, 0; green, 0; blue, 0 }  ,draw opacity=1 ]   (48.18,61.76) -- (116.26,103.38) ;
\draw [color={rgb, 255:red, 0; green, 0; blue, 0 }  ,draw opacity=1 ]   (116.26,103.38) -- (132.64,53.76) ;

\draw (45.58,130.29) node [anchor=north west][inner sep=0.75pt]    {$G\ =\ K_{4,3,2,2}$};
\draw (214.14,128.07) node [anchor=north west][inner sep=0.75pt]    {$H\ =\ K_{4} \ $};
\draw (19.6,51.8) node [anchor=north west][inner sep=0.75pt]  [font=\footnotesize]  {$v_{11}$};
\draw (139.6,41.2) node [anchor=north west][inner sep=0.75pt]  [font=\footnotesize]  {$v_{23}$};
\draw (133.6,91.6) node [anchor=north west][inner sep=0.75pt]  [font=\footnotesize]  {$v_{31}$};
\draw (22,79) node [anchor=north west][inner sep=0.75pt]  [font=\footnotesize]  {$v_{42}$};
\draw (37.05,39.95) node [anchor=north west][inner sep=0.75pt]  [font=\footnotesize,color={rgb, 255:red, 208; green, 2; blue, 27 }  ,opacity=1 ]  {$1$};
\draw (130.43,28.55) node [anchor=north west][inner sep=0.75pt]  [font=\footnotesize,color={rgb, 255:red, 126; green, 211; blue, 33 }  ,opacity=1 ]  {$2$};
\draw (130.95,109.25) node [anchor=north west][inner sep=0.75pt]  [font=\footnotesize,color={rgb, 255:red, 74; green, 144; blue, 226 }  ,opacity=1 ]  {$3$};
\draw (34.2,98.08) node [anchor=north west][inner sep=0.75pt]  [font=\footnotesize,color={rgb, 255:red, 245; green, 166; blue, 35 }  ,opacity=1 ]  {$4$};
\draw (208.2,24.2) node [anchor=north west][inner sep=0.75pt]  [font=\footnotesize,color={rgb, 255:red, 208; green, 2; blue, 27 }  ,opacity=1 ]  {$1$};
\draw (271.6,22.8) node [anchor=north west][inner sep=0.75pt]  [font=\footnotesize,color={rgb, 255:red, 126; green, 211; blue, 33 }  ,opacity=1 ]  {$2$};
\draw (272.14,94.02) node [anchor=north west][inner sep=0.75pt]  [font=\footnotesize,color={rgb, 255:red, 74; green, 144; blue, 226 }  ,opacity=1 ]  {$3$};
\draw (209,94.4) node [anchor=north west][inner sep=0.75pt]  [font=\footnotesize,color={rgb, 255:red, 245; green, 166; blue, 35 }  ,opacity=1 ]  {$4$};
\draw (102.6,106.6) node [anchor=north west][inner sep=0.75pt]  [font=\footnotesize]  {$v_{32}$};
\draw (12.6,9.2) node [anchor=north west][inner sep=0.75pt]  [font=\footnotesize]  {$V_{1}$};
\draw (152.2,9.6) node [anchor=north west][inner sep=0.75pt]  [font=\footnotesize]  {$V_{2}$};
\draw (152.4,113.4) node [anchor=north west][inner sep=0.75pt]  [font=\footnotesize]  {$V_{3}$};
\draw (12.4,113.2) node [anchor=north west][inner sep=0.75pt]  [font=\footnotesize]  {$V_{4}$};

\end{tikzpicture}
    \caption{Two Colored $K_4$ subgraphs in $K_{4,3,2,2}$}
    \label{fig:2k4subgraphk4322}
\end{figure}
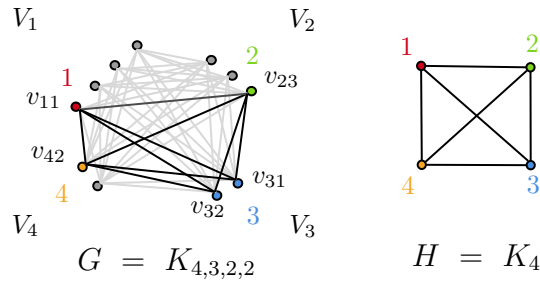

If we repeat this line of reasoning until we have exhausted all vertices in all the maximally independent sets, then we can see the only way to $H$-color $K_{4,3,2,2}$  is by coloring all the vertices in $V_1$ the color 1, all the vertices in $V_2$ the color 2, all the vertices in $V_3$ the color 3, and all the vertices in $V_4$ the color 4.

\begin{figure}[ht]
    \centering
    \begin{tikzpicture}[x=0.75pt,y=0.75pt,yscale=-1,xscale=1]

\draw  [fill={rgb, 255:red, 208; green, 2; blue, 27 }  ,fill opacity=1 ] (42.6,66.12) .. controls (42.6,64.95) and (43.6,64) .. (44.82,64) .. controls (46.05,64) and (47.04,64.95) .. (47.04,66.12) .. controls (47.04,67.29) and (46.05,68.24) .. (44.82,68.24) .. controls (43.6,68.24) and (42.6,67.29) .. (42.6,66.12) -- cycle ;
\draw  [fill={rgb, 255:red, 208; green, 2; blue, 27 }  ,fill opacity=1 ] (62.28,45.2) .. controls (62.28,44.03) and (63.28,43.08) .. (64.5,43.08) .. controls (65.73,43.08) and (66.72,44.03) .. (66.72,45.2) .. controls (66.72,46.37) and (65.73,47.32) .. (64.5,47.32) .. controls (63.28,47.32) and (62.28,46.37) .. (62.28,45.2) -- cycle ;
\draw  [fill={rgb, 255:red, 208; green, 2; blue, 27 }  ,fill opacity=1 ] (52.28,55.6) .. controls (52.28,54.43) and (53.28,53.48) .. (54.5,53.48) .. controls (55.73,53.48) and (56.72,54.43) .. (56.72,55.6) .. controls (56.72,56.77) and (55.73,57.71) .. (54.5,57.71) .. controls (53.28,57.71) and (52.28,56.77) .. (52.28,55.6) -- cycle ;
\draw  [fill={rgb, 255:red, 208; green, 2; blue, 27 }  ,fill opacity=1 ] (73.5,35.19) .. controls (73.5,34.02) and (74.5,33.07) .. (75.72,33.07) .. controls (76.95,33.07) and (77.94,34.02) .. (77.94,35.19) .. controls (77.94,36.36) and (76.95,37.31) .. (75.72,37.31) .. controls (74.5,37.31) and (73.5,36.36) .. (73.5,35.19) -- cycle ;
\draw  [fill={rgb, 255:red, 245; green, 166; blue, 35 }  ,fill opacity=1 ] (46.73,94.81) .. controls (47.62,94.06) and (48.99,94.21) .. (49.78,95.15) .. controls (50.56,96.09) and (50.48,97.46) .. (49.58,98.21) .. controls (48.69,98.96) and (47.32,98.81) .. (46.53,97.87) .. controls (45.74,96.94) and (45.83,95.57) .. (46.73,94.81) -- cycle ;
\draw  [fill={rgb, 255:red, 126; green, 211; blue, 33 }  ,fill opacity=1 ] (121.4,50.42) .. controls (121.4,49.25) and (122.43,48.3) .. (123.7,48.3) .. controls (124.97,48.3) and (126,49.25) .. (126,50.42) .. controls (126,51.59) and (124.97,52.54) .. (123.7,52.54) .. controls (122.43,52.54) and (121.4,51.59) .. (121.4,50.42) -- cycle ;
\draw  [fill={rgb, 255:red, 126; green, 211; blue, 33 }  ,fill opacity=1 ] (110.82,42.6) .. controls (110.82,41.43) and (111.85,40.48) .. (113.12,40.48) .. controls (114.39,40.48) and (115.41,41.43) .. (115.41,42.6) .. controls (115.41,43.77) and (114.39,44.71) .. (113.12,44.71) .. controls (111.85,44.71) and (110.82,43.77) .. (110.82,42.6) -- cycle ;
\draw  [fill={rgb, 255:red, 245; green, 166; blue, 35 }  ,fill opacity=1 ] (54.3,104.43) .. controls (55.19,103.68) and (56.56,103.83) .. (57.35,104.77) .. controls (58.14,105.71) and (58.05,107.08) .. (57.15,107.83) .. controls (56.26,108.58) and (54.89,108.43) .. (54.1,107.49) .. controls (53.32,106.56) and (53.4,105.19) .. (54.3,104.43) -- cycle ;
\draw  [fill={rgb, 255:red, 74; green, 144; blue, 226 }  ,fill opacity=1 ] (114.12,112.16) .. controls (113.49,111.17) and (113.83,109.84) .. (114.86,109.18) .. controls (115.9,108.53) and (117.25,108.8) .. (117.87,109.79) .. controls (118.5,110.78) and (118.16,112.11) .. (117.12,112.76) .. controls (116.09,113.42) and (114.74,113.15) .. (114.12,112.16) -- cycle ;
\draw  [fill={rgb, 255:red, 126; green, 211; blue, 33 }  ,fill opacity=1 ] (131.07,58.26) .. controls (131.07,57.09) and (132.1,56.14) .. (133.37,56.14) .. controls (134.64,56.14) and (135.67,57.09) .. (135.67,58.26) .. controls (135.67,59.43) and (134.64,60.38) .. (133.37,60.38) .. controls (132.1,60.38) and (131.07,59.43) .. (131.07,58.26) -- cycle ;
\draw  [fill={rgb, 255:red, 74; green, 144; blue, 226 }  ,fill opacity=1 ] (124.68,105.96) .. controls (124.05,104.98) and (124.39,103.64) .. (125.42,102.99) .. controls (126.46,102.33) and (127.81,102.61) .. (128.43,103.6) .. controls (129.05,104.58) and (128.72,105.92) .. (127.68,106.57) .. controls (126.65,107.23) and (125.3,106.95) .. (124.68,105.96) -- cycle ;
\draw [color={rgb, 255:red, 0; green, 0; blue, 0 }  ,draw opacity=1 ]   (46.78,67.56) -- (49.78,95.15) ;
\draw [color={rgb, 255:red, 0; green, 0; blue, 0 }  ,draw opacity=1 ]   (46.78,67.56) -- (57.35,104.77) ;
\draw    (46.78,67.56) -- (114.86,109.18) ;
\draw    (46.78,67.56) -- (125.42,102.99) ;
\draw    (46.78,67.56) -- (131.24,59.56) ;
\draw    (46.78,67.56) -- (120.65,51.33) ;
\draw    (46.78,67.56) -- (110.76,43.78) ;
\draw    (57.07,55.6) -- (110.76,43.78) ;
\draw    (57.07,55.6) -- (120.65,51.33) ;
\draw    (57.07,55.6) -- (131.24,59.56) ;
\draw    (56.72,55.6) -- (125.42,102.99) ;
\draw    (57.67,56.67) -- (114.86,109.18) ;
\draw    (56.72,55.6) -- (57.35,104.77) ;
\draw    (57.67,56.67) -- (49.78,95.15) ;
\draw    (65.89,46.67) -- (49.78,95.15) ;
\draw    (65.89,46.67) -- (57.35,104.77) ;
\draw    (65.13,47.32) -- (110.76,43.78) ;
\draw    (66.57,46.67) -- (120.65,51.33) ;
\draw    (66.57,46.67) -- (131.24,59.56) ;
\draw    (65.89,46.67) -- (114.86,109.18) ;
\draw    (65.89,46.67) -- (125.42,102.99) ;
\draw    (75.72,37.31) -- (49.78,95.15) ;
\draw    (75.72,37.31) -- (57.35,104.77) ;
\draw    (75.72,37.31) -- (114.86,109.18) ;
\draw    (75.72,37.31) -- (125.42,102.99) ;
\draw    (76.76,37.31) -- (110.76,43.78) ;
\draw    (77.85,38.22) -- (120.65,51.33) ;
\draw    (75.72,37.31) -- (131.24,59.56) ;
\draw    (57.35,104.77) -- (114.86,109.18) ;
\draw    (57.35,104.77) -- (125.42,102.99) ;
\draw    (57.35,104.77) -- (131.24,59.56) ;
\draw    (57.35,104.77) -- (120.65,51.33) ;
\draw    (57.35,104.77) -- (110.76,43.78) ;
\draw    (49.78,95.15) -- (114.86,109.18) ;
\draw    (49.78,95.15) -- (125.42,102.99) ;
\draw    (49.78,95.15) -- (131.24,59.56) ;
\draw    (49.78,95.15) -- (120.65,51.33) ;
\draw    (49.78,95.15) -- (110.76,43.78) ;
\draw    (110.76,43.78) -- (114.86,109.18) ;
\draw    (114.86,109.18) -- (121.4,50.42) ;
\draw    (114.86,109.18) -- (131.24,59.56) ;
\draw    (125.42,102.99) -- (110.76,43.78) ;
\draw    (125.42,102.99) -- (121.4,50.42) ;
\draw    (125.42,102.99) -- (131.24,59.56) ;
\draw  [fill={rgb, 255:red, 245; green, 166; blue, 35 }  ,fill opacity=1 ] (224.6,85.64) .. controls (224.6,84.46) and (225.53,83.51) .. (226.67,83.51) .. controls (227.81,83.51) and (228.73,84.46) .. (228.73,85.64) .. controls (228.73,86.81) and (227.81,87.77) .. (226.67,87.77) .. controls (225.53,87.77) and (224.6,86.81) .. (224.6,85.64) -- cycle ;
\draw  [fill={rgb, 255:red, 126; green, 211; blue, 33 }  ,fill opacity=1 ] (279.43,36.43) .. controls (279.43,35.26) and (280.36,34.31) .. (281.5,34.31) .. controls (282.64,34.31) and (283.56,35.26) .. (283.56,36.43) .. controls (283.56,37.61) and (282.64,38.56) .. (281.5,38.56) .. controls (280.36,38.56) and (279.43,37.61) .. (279.43,36.43) -- cycle ;
\draw  [fill={rgb, 255:red, 208; green, 2; blue, 27 }  ,fill opacity=1 ] (224.27,35.72) .. controls (224.27,34.54) and (225.2,33.59) .. (226.34,33.59) .. controls (227.48,33.59) and (228.4,34.54) .. (228.4,35.72) .. controls (228.4,36.89) and (227.48,37.85) .. (226.34,37.85) .. controls (225.2,37.85) and (224.27,36.89) .. (224.27,35.72) -- cycle ;
\draw  [fill={rgb, 255:red, 74; green, 144; blue, 226 }  ,fill opacity=1 ] (279.54,85.82) .. controls (279.54,84.64) and (280.46,83.69) .. (281.6,83.69) .. controls (282.74,83.69) and (283.67,84.64) .. (283.67,85.82) .. controls (283.67,86.99) and (282.74,87.94) .. (281.6,87.94) .. controls (280.46,87.94) and (279.54,86.99) .. (279.54,85.82) -- cycle ;
\draw [fill={rgb, 255:red, 0; green, 0; blue, 0 }  ,fill opacity=1 ]   (226.34,37.85) -- (226.67,83.51) ;
\draw    (228.4,35.72) -- (279.43,36.43) ;
\draw    (281.5,38.56) -- (281.6,83.69) ;
\draw    (228.73,85.64) -- (279.54,85.82) ;
\draw    (228.33,84.33) -- (279.67,38.11) ;
\draw    (227.67,37.44) -- (280.33,83.67) ;

\draw (43.78,125.29) node [anchor=north west][inner sep=0.75pt]    {$G\ =\ K_{4,3,2,2}$};
\draw (41.72,32.62) node [anchor=north west][inner sep=0.75pt]  [font=\footnotesize,color={rgb, 255:red, 208; green, 2; blue, 27 }  ,opacity=1 ]  {$1$};
\draw (126.77,30.55) node [anchor=north west][inner sep=0.75pt]  [font=\footnotesize,color={rgb, 255:red, 126; green, 211; blue, 33 }  ,opacity=1 ]  {$2$};
\draw (122.95,111.58) node [anchor=north west][inner sep=0.75pt]  [font=\footnotesize,color={rgb, 255:red, 74; green, 144; blue, 226 }  ,opacity=1 ]  {$3$};
\draw (36.2,102.75) node [anchor=north west][inner sep=0.75pt]  [font=\footnotesize,color={rgb, 255:red, 245; green, 166; blue, 35 }  ,opacity=1 ]  {$4$};
\draw (11.6,13.2) node [anchor=north west][inner sep=0.75pt]  [font=\footnotesize]  {$V_{1}$};
\draw (141.2,13.6) node [anchor=north west][inner sep=0.75pt]  [font=\footnotesize]  {$V_{2}$};
\draw (141.4,104.4) node [anchor=north west][inner sep=0.75pt]  [font=\footnotesize]  {$V_{3}$};
\draw (12.4,104.2) node [anchor=north west][inner sep=0.75pt]  [font=\footnotesize]  {$V_{4}$};
\draw (219.14,123.07) node [anchor=north west][inner sep=0.75pt]    {$H\ =\ K_{4} \ $};
\draw (214.2,20.2) node [anchor=north west][inner sep=0.75pt]  [font=\footnotesize,color={rgb, 255:red, 208; green, 2; blue, 27 }  ,opacity=1 ]  {$1$};
\draw (277.6,18.8) node [anchor=north west][inner sep=0.75pt]  [font=\footnotesize,color={rgb, 255:red, 126; green, 211; blue, 33 }  ,opacity=1 ]  {$2$};
\draw (278.14,90.02) node [anchor=north west][inner sep=0.75pt]  [font=\footnotesize,color={rgb, 255:red, 74; green, 144; blue, 226 }  ,opacity=1 ]  {$3$};
\draw (215,90.4) node [anchor=north west][inner sep=0.75pt]  [font=\footnotesize,color={rgb, 255:red, 245; green, 166; blue, 35 }  ,opacity=1 ]  {$4$};

\end{tikzpicture}
    \caption{Coloring All the Vertices of $K_{4,3,2,2}$}
    \label{fig:k4322colored}
\end{figure}
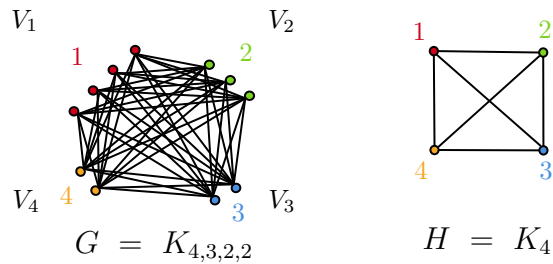

So $m_{(4,3,2,2)}$ is the only $m_{\lambda}$ term that will appear in $X_{G}^{H}$.
\end{exm}

\begin{rem} If $m_{\lambda} \neq m_{1^n}$, then it is not possible to write $m_\lambda$ as a scalar multiple of a single chromatic symmetric function $X_{G}$. Recall that if $G$ has $n$ vertices, the coefficient of $m_{1^n}$ in $X_G$ is positive. From Proposition \ref{prop.mulpartmonomials}, we can write $m_\lambda$ as a scalar multiple of a single $X_{G}^{H}$ but not $X_{G}$. So Proposition \ref{prop.mulpartmonomials} holds for $H$-chromatic symmetric functions but not chromatic symmetric functions.
\end{rem}

Now we use Proposition \ref{prop.mulpartmonomials} to prove that certain sets of $H$-chromatic symmetric functions form a basis for $\Lambda^n$.  We also express the elementary symmetric functions and the power sum symmetric functions in terms of this basis.

\begin{prop}
The following set of $H$-chromatic symmetric functions is a basis for $\Lambda^n$: \begin{equation*}
    \beta = \{ X_{G}^{H} : G = K_{\lambda_{1}, \lambda_{2} \dots \lambda_{k}} \text{ where } \lambda = (\lambda_{1}, \lambda_{2} \dots \lambda_{k}) \vdash n, \lambda \neq (n), \text{ and } H = K_{l(\lambda)} \} \cup \{ X_{\overline{K_{n}}}^{K_1} \}
\end{equation*}

\begin{proof}
From Proposition \ref{prop.mulpartmonomials}, the set $\beta$ corresponds to \begin{equation*}
    \{ c_\lambda m_{\lambda}: \lambda \vdash n, c_\lambda \neq 0\}
\end{equation*} This is a basis for $\Lambda^n$ since $\{  m_{\lambda}: \lambda \vdash n\}$ is a basis for $\Lambda^n$.
\end{proof}
\end{prop}

We also have a basis for $\Lambda$ in terms of $X_{G}^{H}$. In what follows, recall that $Par := \bigcup_{n \geq 0} \{ \lambda | \lambda \vdash n \}$.

\begin{cor}
The following set,
\begin{equation*}
    \beta = \{ X_{G}^{H} : G = K_{\lambda_{1}, \lambda_{2} \dots \lambda_{k}}, \lambda = (\lambda_{1}, \lambda_{2} \dots \lambda_{k}) \in Par, \forall n, \lambda \neq (n) \text{  and } H = K_{l(\lambda)} \} \cup \{ X_{\overline{K_{n}}}^{K_1}:  n \in \mathbb{N} \},
\end{equation*}
is a basis for $\Lambda$.
\end{cor}

\begin{cor}
Let $\lambda = (\lambda_1 , \lambda_2 \dots \lambda_{k})$. Then \begin{align*}
    e_{\lambda} 
    &= \frac{1}{(\lambda_{1})^{2}(\lambda_{2})^{2}\dots (\lambda_{k})^{2}}X_{K_{\lambda_{1}}}^{K_{\lambda_{1}}}X_{K_{\lambda_{2}}}^{K_{\lambda_{2}}} \dots X_{K_{\lambda_{k}}}^{K_{\lambda_{k}}}
\end{align*}
\begin{proof} 
First note we can write the $\{m_{\lambda}\}$ basis in terms of $X_{G}^{H}$.
Recall that $e_n = m_{1^n}$ and $e_{\lambda} = e_{\lambda_1}e_{\lambda_2} \dots e_{\lambda_k}$.
Also,
   $ e_{\lambda} = m_{1^{\lambda_{1}}}m_{1^{\lambda_{2}}} \dots m_{1^{\lambda_{k}}}.$
   
We know that there are $\lambda_{i}$ ways of $(K_{\lambda_{i}}, \phi)$-coloring the graph $G = K_{\lambda_{i}}$. There are $\lambda_{i}$ ways of fixing colorings $\phi$ on $H = K_{\lambda_{i}}$. So we know that $X_{K_{\lambda_{i}}}^{K_{\lambda_{i}}} = (\lambda_{i})^{2} m_{1^{\lambda_{i}}}$. Apply Proposition \ref{prop.mulpartmonomials}.
\end{proof}
\end{cor}

\begin{cor}
Let $\lambda = (\lambda_1 , \lambda_2 \dots \lambda_{k})$. Then, \begin{align*}
    p_{\lambda} 
    &= X_{\overline{K_{\lambda_{1}}}}^{K_{1}}X_{\overline{K_{\lambda_{2}}}}^{K_{1}} \dots X_{\overline{K_{\lambda_{k}}}}^{K_{1}}
\end{align*}
\begin{proof}
First note that we can write the $\{m_{\lambda}\}$ basis in terms of $X_{G}^{H}$.
Recall that $p_n = m_{(n)}$ and $p_{\lambda} = p_{\lambda_1}p_{\lambda_2} \dots p_{\lambda_k}$.
Also, 
$p_{\lambda} = m_{(\lambda_{1})}m_{(\lambda_{2})} \dots m_{(\lambda_{k})}$.

Apply Proposition \ref{prop.mulpartmonomials}. Note that if $H$ is only 1 vertex, then there is only one way to color the vertices of $G$, which is by coloring all the vertices of $G$ the same color.
\end{proof}
\end{cor}

\begin{cor}\label{cor.hcsflincombo} Given an $H$-chromatic symmetric function of a graph G, $X_{G}^{H}$, it is possible to find sets of graphs $\{H_j | j \in \mathbb{N}\}$ and $\{ G_i | i \in \mathbb{N} \}$ such that $X_{G}^{H}$ can be expressed as a linear combination of  $\{ X_{G_{i}}^{H_{j}} \}$. 
\begin{proof}
Proposition \ref{prop.mulpartmonomials} gives us a way of rewriting $m_{\lambda}$ in terms of some $X_{G}^{H}$. Now since $X_{G}^{H}$ is a symmetric function, it's in the span of $\{m_{\lambda} : \lambda \in Par\}$. So we can rewrite any $X_{G}^{H}$ as a linear combination of other $X_{G_{i}}^{H_{j}}$.
\end{proof}
\end{cor}

\section{$H$-chromatic symmetric functions of complete bipartite graphs}
\label{degreeseq}

In the beginning of Section \ref{completebasis}, we put restrictions on $H$ so that each part in $K_\lambda$ must be colored the same. Similarly, we can also put restrictions on $H$ so that all but except for one part in $K_\lambda$ must be colored the same. Consider $K_{n+2}^{--}$, which is defined to be the graph obtained by removing two non-adjacent edges from $K_{n+2}$. If $H$ does not contain $K_{n+2}^{--}$ as a subgraph, then it follows immediately that at most one part can be colored using different colors. In that case, we can think a proper $H$-coloring as if we first pick a subgraph of $K_{n-1}$ from $H$ to color $(n-1)$ parts of $G$. Then, we count the number of all other vertices that are adjacent to all these $(n-1)$ vertices. These are exactly the vertices that we can use to (freely) color the remaining part of $G$. For a general multi-partite $G$, this procedure is complicated to carry out and may not be fruitful. However, when $G$ is  bipartite, we are able to use this approach to simply characterize the $H$-colorings of $G$ and produce some interesting results.
We note that several of our remarks, lemmas, and corollaries in this section are results on $H$-coloring graphs, rather than specific results on chromatic symmetric functions. To the best of our knowledge these have not appeared in the graph theory literature before.

\begin{rem}
Let $c$ be a $(H,\phi)$-coloring of some graph $G$. Then $c$ corresponds to a subgraph, $C$, of $H$, which is the subgraph induced on the vertices of $H$ used in the coloring of $G$.
\end{rem}

That is, every $(H,\phi)$-coloring has an associated induced subgraph $C \subseteq H$. When we focus on complete bipartite $G$, we can determine exactly the type of subgraph $C$ is. 

\begin{lem}\label{lem.atmost1away}
Let $G$ be a complete bipartite graph, and let $H$ be a simple graph. Then for any $H$-coloring of $G$, any two colors used to color $G$ must be adjacent to one another in $H$ or adjacent to a common vertex in $H$.
\end{lem}

\begin{proof}
 First note that because $G$ is \textit{complete} bipartite, this induced subgraph must be connected. Denote the maximally independent sets of $G$ by $V_1$ and $V_2$. Trivially, since all vertices in $V_1$ are connected to all vertices in $V_2$, the colors in $H$ used to color $V_1$ must be adjacent to the colors in $H$ used to color $V_2$. So the only potential case where the lemma fails is within $V_1$ or $V_2$. Without loss of generality consider $v_1, v_2 \in V_1$, and suppose we color them by $1$ and $2$ in $H$ respectively. Then since all of $V_1$ is adjacent to all of $V_2$, both $v_1$ and $v_2$ are adjacent to some vertex $v_3$ in $V_2$, which is colored, say, $3$. Then it must be that both $1$ and $2$ are adjacent to $3$ in $H$.
\end{proof}

This observation is not complicated, but has important implications.

\begin{cor}\label{cor.HcolisScol}
Let $H$ be a graph that does not contain $C_4$ as a subgraph, and $G$ be a complete bipartite graph. Then a $H$-coloring of $G$ is an $S$-coloring of $G$ and vice versa, where $S$ is a star-subgraph of $H$.
\end{cor}

\begin{proof}
For any coloring $\phi$ of $G$, $\phi$ is associated with an induced subgraph on the vertices of $H$ used in the coloring. Further, by Lemma \ref{lem.atmost1away}, we know that these subgraphs must have diameter less than 3. With the restriction that $H$ does not contain $C_4$ as a subgraph, the only possible subgraphs of $H$ with diameter less than $3$ are stars, or stars with additional edges between the leaf vertices. In the latter case, note that by carefully choosing a central vertex and deleting unnecessary edges, these too can be reduced to stars. Conversely, if $\phi$ is an $S$-coloring of $G$ and $S$ is a subgraph of $H$, then $\phi$ is also a $H$-coloring of $G$. 
\end{proof}

In fact, if we restrict our attention from any complete bipartite $G$ to the star, we can remove the condition that $H$ does not contain $C_4$ as a subgraph. 

\begin{prop}
Let $G$ be a star graph. Then $G$ cannot be colored using every single color in $H = C_4$.
\end{prop}

\begin{proof}
 Label the vertices of $C_4$ by $1,2,3,4$ cyclically. Suppose we are able to color $G$ using all of the colors in $C_4$, and without loss of generality we color the central vertex of $G$ using $1$. Then there must exist a leaf vertex $v_1\in G$ colored $3$ by assumption. But $1$ and $3$ are not adjacent in $C_4$, so this is a contradiction.
\end{proof}

\begin{cor}\label{cor.HcolisScolStar}
Let $H$ be any simple graph and let $G$ be a star graph. Then a $H$-coloring of $G$ is an $S$-coloring of $G$ and vice versa, where $S$ is a star-subgraph of $H$.
\end{cor}

We now can define the key characteristic of a graph that will allow us to distinguish them: the degree sequence. The degree sequence and star sequence here is defined as  in Martin and Wagner \cite{MarMorWag}. 

\begin{dfn}
Let $H$ be a graph on $k$ vertices. Then we define the degree sequence of $H$ as $D = (d_1,d_2,\dots,d_k)$, where $d_i = |\{ v \in V(T) \ : \ \text{deg}(v) = i\}|$.
\end{dfn}

\begin{dfn}
Let $H$ be a graph on $k$ vertices. Then we define the star sequence of $H$ as $Z = (z_2,..., z_k)$, where $z_i = z_i(T) = $ number of $S_i$-subgraphs of $H$. 
\end{dfn}

\begin{lem}\label{degseq2starseq}
Let $H$ be a graph on $k$ vertices with degree sequence $D = (d_1,...,d_k)$. Then the star sequence of $H$, can be determined by the formula
$$z_{i+1} = \sum\limits_{i\leq j} {j \choose i} d_j$$
for $2\leq i \leq k-1$.

Note that in the case of $z_2$, the formula above double counts the number of $S_2$-subgraphs of $H$, as the argument specifies the central vertex of the star. Of course, in the case of $S_2$, the star is symmetric regardless of choice of central vertex, hence the double counting. To amend this, either use the formula above and divide the result for $z_2$ by 2, or simply set $z_2$ to be the number of edges in $H$ (a similar formula appears in \cite{MarMorWag}).

\end{lem}

\begin{lem}\label{lem.Hcharac}
Let $H$ be a graph on $k$ vertices, and let $G = S_n$ for some $n$. Then the $H$-chromatic symmetric function of $G$ is given by

$$\sum\limits_{\lambda} c_\lambda m_\lambda^n,$$
where we sum over all partitions $\lambda \vdash n$ with length no more than $k$, and where the coefficient $c_\lambda$ is given by
$c_\lambda := z_{l(\lambda)}\beta_\lambda,$
with $\beta_\lambda$ being the coefficient of $m_\lambda^n$ in the corresponding $S_{l(\lambda)}$-chromatic symmetric function of $G$. 
\end{lem}

\begin{proof}
 Every coloring of $G$ by a graph $H$ has an associated star subgraph in $H$, as noted in Corollary \ref{cor.HcolisScolStar}. Hence, for each $n$-augmented monomial $m_\lambda^n$, the coefficient counts the number of ways we can make the following choice. We need to choose an $S_{l(\lambda)}$ subgraph of $H$ to color $G$, and then we need to choose a way of using this $S_{l(\lambda)}$ to color $G$. The number of ways of doing the first is $z_{l(\lambda)}$, and the number of ways of doing the second is the same as the coefficient of $m_\lambda^n$ in the corresponding $S_{l(\lambda)}$-chromatic symmetric function of $G$.
\end{proof}

The next corollaries follow easily:

\begin{cor}
Let $H_1$ and $H_2$ both be graphs on $k$ vertices. Let $G$ be a star. If $H_1$ and $H_2$ share a degree sequence, then $X_G^{H_1} = X_G^{H_2}$.
\label{cor810}
\end{cor}

\begin{cor}
Let $H_1$ and $H_2$ both be graphs on $k$ vertices. Suppose $H_1$ and $H_2$ have different degree sequences. Then let $N$ be the first entry in the degree sequence where they differ. Then they also differ at the $N$-th entry of their star sequences ($z_{N+1}$). Let $G$ be a complete bipartite graph on $n\geq N$ vertices. Then $X_G^{H_1} \neq X_G^{H_2}$.
\label{cor811}
\end{cor}

These two results can be combined to yield the following result:

\begin{prop}
Let $\mathcal{H}$ be a set of graphs on $\leq k$ vertices, and let $G = S_k$. Then the $H$-chromatic symmetric functions of $G$ are distinct up to the degree sequences of $H \in \mathcal{H}$. 
\label{prop812}
\end{prop}

\begin{proof}
For any pair of graphs $H_1$ and $H_2$ in $\mathcal{H}$ with different degree sequences, their star sequences will differ at some $N\leq k$. Comparing $X_G^{H_1}$ and $X_G^{H_2}$ in Lemma \ref{lem.Hcharac}, the coefficients of monomials with corresponding partition length $N$ will be distinct. Since $N \leq k$ for every $H \in \mathcal{H}$, this choice of $G$ works perfectly. 
\end{proof}

Restricting $G$ to be a star instead of any complete bipartite graph on $k$ vertices allows us to classify more graphs, including those that contain $C_4$ subgraphs. However, with the restriction that we only consider $H$s that do not contain $C_4$ as a subgraph, all of the above can be translated naturally into analogues for general complete bipartite $G$. 

\section{Further avenues to explore and some conjectures}
\label{furtheravenues}

In this section, we shall present some open problems and conjectures in $H$-chromatic symmetric functions. Since the subject is still new, and experimentation with large graphs is, in general, impractical, we shall state most of them in the form of problems, as opposed to conjectures.

\subsection{Possibility of constructing a basis for $\Lambda$ using $X_{G}^{H}$ where $H$ is fixed}

It is natural to ask whether Corollary \ref{cor.hcsflincombo} holds when $G_{i}$ or $H_{j}$ is fixed (and if not, for which graphs it's possible).
In fact, it turns out it is impossible to generate a basis for $\Lambda^n$ using the $H$-chromatic symmetric functions of a fixed graph $G$.
\begin{prop}
It is not possible to generate a basis for $\Lambda^n$ using a set of $H$-chromatic symmetric functions, $\{X_{G}^{H_{i}}\}$, where $G$ is a fixed graph and $\{H_i | i \in \mathbb{N} , 1 \leq i \leq n\}$ is finite set of graphs.
\label{prop91}

\begin{proof}
Suppose $n \geq 4$. We will show it is not possible to generate a basis for $\Lambda^n$ using a set of $H$-chromatic symmetric functions $\{X_{G}^{H_{i}}\}$ where $G$ is a fixed graph.

If $G$ has an edge connecting two distinct vertices, then it is not possible to color all the vertices of $G$ the same color, i.e.\ we would not have the $m_{(n)}$ term in $X_{G}^{H}$, and $m_{(n)}$ would not be in the span of the $H$-chromatic symmetric functions $\{X_{G}^{H_{i}}\}$.  Thus  $G$ must be the graph with no edges: $\overline{K_{n}}$.

Now we observe that if $H_1$ and $H_2$ have the same number of vertices, then $X_{\overline{K_{n}}}^{H_{1}} = X_{\overline{K_{n}}}^{H_{2}}$. From the definition of $H$-coloring, since $\overline{K_{n}}$ has no edges, we can color the vertices of $\overline{K_{n}}$ however we want and the edges of $H$ do not affect the $H$-coloring of $\overline{K_{n}}$.

We can choose $X_{\overline{K_{n}}}^{\overline{K_{i}}}$ as a representative for the $H$-chromatic symmetric function of $K_{n}$ where $H$ is a graph on $i$ vertices.

Here we compute $X_{\overline{K_{n}}}^{\overline{K_{i}}}$ for $i<n$.

\begin{equation*}
    X_{\overline{K_{n}}}^{\overline{K_{i}}} = \sum_{\lambda \vdash n, l(\lambda) \leq i} i! \binom{n}{\lambda_1 , \lambda_2 \dots \lambda_k} m_{\lambda} 
\end{equation*}

Here we compute $X_{\overline{K_{n}}}^{\overline{K_{i}}}$ for $i\geq n$.

\begin{equation*}
    X_{\overline{K_{n}}}^{\overline{K_{i}}} = \sum_{\lambda \vdash n} i! \binom{n}{\lambda_1 , \lambda_2 \dots \lambda_k} m_{\lambda} 
\end{equation*}

From the above formulas, we observe a scalar multiple property for $X_{\overline{K_{n}}}^{\overline{K_{i}}}$ where $i \geq n$. If $c > d \geq 0$, then: \begin{equation*}
    X_{\overline{K_{n}}}^{\overline{K_{i+c}}} = \frac{(i+c)!}{(i+d)!}X_{\overline{K_{n}}}^{\overline{K_{i+d}}}
\end{equation*}

Let $p(n)$ denote the number of partitions of $n$.
Recall that $\dim(\Lambda^{n}) = p(n)$.

We can show that if $n \geq 4,$ $n < p(n)$. The following map on $\{1, 2, \dots n\}$ is defined on every $i \in \{1, 2, \dots n\}$: \begin{equation*}
    f: i \mapsto \lambda \text{ such that } l(\lambda) = i
\end{equation*} 
Note that if $n \geq 4, i = 2$ maps to both $(n-1,1)$ and $(n-2,2)$. Comparing the sizes of the domain and range gives us that $n < p(n)$.

If $B = \{ X_{\overline{K_{n}}}^{\overline{K_{i}}} \}$ and $|B| = \dim(\Lambda^{n}) = p(n)$, then $B$ cannot be linearly independent. Since $p(n) > n$, $B$ must contain some $X_{\overline{K_{n}}}^{\overline{K_{i+d}}}$ and $X_{\overline{K_{n}}}^{\overline{K_{i+c}}}$ ($c > d \geq 0$) such that 
\begin{equation*}
    X_{\overline{K_{n}}}^{\overline{K_{i+c}}} = \frac{(i+c)!}{(i+d)!}X_{\overline{K_{n}}}^{\overline{K_{i+d}}}
\end{equation*}

We have shown that any set of $H$-chromatic symmetric functions of $\overline{K_{n}}$ that is of size $\dim(\Lambda^{n}) = p(n)$ contains at least 2 linearly dependent $H$-chromatic symmetric functions of $\overline{K_{n}}$.

It follows that we cannot create a basis for $\Lambda^n$ using $H$-chromatic symmetric functions of a fixed $G$.
\end{proof}

But is it possible to generate a basis for $\Lambda^{n}$ using a set of $H$-chromatic symmetric functions where the $H$ is fixed rather than the $G$? 

It turns out the answer is yes. As discussed in Section 3, when you choose $H = K_n$, then you can recover the chromatic symmetric functions up to a scalar multiple. Cho and van Willigenburg showed that $\{ X_{G_{\lambda}}| \lambda = (\lambda_1 , \lambda_2 \dots \lambda_{k}) \vdash n \}$ forms a basis of $\Lambda$ \cite{cho2015chromatic}. Here, $G_i$ is a connected graph on $i$ vertices and $G_{\lambda} = G_{\lambda_1} \cup G_{\lambda_2} \dots \cup G_{\lambda_k}$.

But it is still unknown if we can create a basis for $\Lambda^n$ using $H$-chromatic symmetric functions involving a fixed $H$ and varying $G$, where $H$ is not a complete graph.

So far, we have primarily considered cases where $G, H$ are simple graphs. The proof of Proposition \ref{prop91} depends on the fact that if $G$ has an edge connecting two distinct vertices, it is not possible to $H$-color all the vertices of $G$ the same color. However, if $H$ could have a loop, this assumption is not true. An interesting question to explore is whether we can create a basis for $\Lambda^n$ using $H$-chromatic symmetric functions of a fixed $G$ where $G, H$ do not necessarily have to be simple graphs. 
\end{prop}

\subsection{Self-distinguishability of graphs} Given a graph $G$, the symmetric function $X_G^G$ never vanishes. A natural question to ask is if this ``self"-chromatic symmetric function can be used to distinguish all graphs up to isomorphism. It is known that two non-isomorphic graphs can have the same ``self''-chromatic symmetric functions. For example, if $G_1 = K_1 \uplus K_2$ and $G_2 = P_3$, then $X_{G_1}^{G_1} = X_{G_2}^{G_2} = 12m_{(1,1,1)} + 4m_{(2,1)}$.

Nevertheless, if we restrict our attention to a smaller class of graphs, then there is more we can ask. For example, we have checked that up to $k = 9$, two trees on $k$ vertices are self-distinguishable. Of course, when two trees have a different number of automorphisms, then they must have distinct ``self"-chromatic symmetric functions by Proposition \ref{prop.automorphism}. However, this condition is by no means necessary. We shall propose the following conjecture.

\begin{con}\label{con.selfdistinguish}
Let $T_1, T_2$ be two non-isomorphic trees, then $X_{T_1}^{T_1} \neq X_{T_2}^{T_2}$.
\end{con}

\subsection{Finite and Uniform distinguishability in a more general setting}

In Section \ref{distinguishers}, we have seen that any finite set of graphs is uniformly distinguishable. We have also seen that not every set of graphs is uniformly, or even finitely, distinguishable. When our $G$'s are connected, we also proved that finite distinguishability implies uniform distinguishability when a uniform bound on the coefficients of the $H$-chromatic symmetric function can be established.

However, we have not shown whether these results are sharp. Therefore, we may wish to ask the following questions.

\begin{prob}
Is finite distinguishability equivalent to uniform distinguishability? If not, can we pose restrictions on $G$'s, such as connectedness, to make it true?
\end{prob}

\begin{prob}
Is there a simple necessary and sufficient condition for a set of graph being finitely (uniformly) distinguishable? For example, we have shown that a set of graphs whose chromatic numbers are not bounded cannot be finitely distinguished. Is the converse true?
\end{prob}

\subsection{Monotonicity with respect to a basis}

One of the big open questions in chromatic symmetric function theory concerns $e$-positivity and Schur-positivity of the chromatic symmetric functions (e.g.\ \cite{Gash}, \cite{GuayPaq}, \cite{SharWa}) Clearly, we can ask similar questions here. Furthermore we say a symmetric function is {\em monotone} with respect to a basis $b$ if and only if its coefficients in the $b$-expansion are all non-negative or all non-positive.

\begin{prob}
Find some necessary or sufficient (or both) conditions on $G$ and $H$ for $X_G^H$ to be $e$-positive (monotone) or Schur-positive (monotone).
\end{prob}

People are less interested about $p$-positivity of chromatic symmetric functions since it is known that $\omega(X_G)$ is always $p$-positive, where $\omega$ is the standard involution on $\Lambda$ \cite{stanley1995symmetric}. This fact, however, does not generalize to $H$-chromatic symmetric functions. For instance, $$\omega(X_{P_4}^{C_7}) = 840p_{(4)} - 672p_{(3,1)} + 1260p_{(2,2)} - 168p_{(2,1,1)} + 84p_{(1,1,1,1)}.$$ Nevertheless, the result  does hold when $H$ is a star.

\begin{con}
Let $n \geq 1$. Suppose $H = S_{n+1}$ is a star. Then, $\omega(X_G^H)$ is $p$-monotone for every graph $G$.
\end{con}

This is trivial when $G$ is not a bipartite graph, in which case the symmetric function vanishes. We have also shown the result when $G$ is a connected bipartite and $n$ is sufficiently large. Indeed, we have the following formula that can be seen with little effort.

\begin{prop}

Suppose $G$ is a connected bipartite graph and $H = S_{n+1}$. Assume $n \geq k_1, k_2$, where $k_1, k_2$ are the sizes of the two parts of $G$. If we set $$(X_1)_G^H := \sum_{k_1 \leq j \leq k_1+k_2} (-1)^{j - k_1}n!{k_2 \choose {j - k_1}} p_{(j, 1^{k_1 + k_2 - j})},$$ and $$(X_2)_G^H := \sum_{k_2 \leq j \leq k_1+k_2} (-1)^{j - k_2}n!{k_1 \choose {j - k_2}} p_{(j, 1^{k_1 + k_2 - j})}.$$ Then, $X_G^H = (X_1)_G^H + (X_2)_G^H$.

\end{prop}

From there, we can prove the result when $G$ is a general bipartite graph and the sizes of the two parts of each connected component of $G$ have the same parity, and that $n$ is sufficiently large. However, a further generalization seems non-trivial.

\section{Acknowledgements}
This work was supported by the Canadian Tri-Council Research Support Fund. The author A.M.F. was supported by an NSERC Discovery Grant.  This research was conducted at the Fields Institute, Toronto, Canada as part of the 2020 Fields Undergraduate Summer Research Program and was funded by that program. The authors thank an anonymous reviewer for a careful reading of the paper and for helpful suggestions which have improved it.

\end{document}